\documentclass[a4paper,10pt,
twoside
]{article}

\usepackage{amsfonts}
\usepackage{amssymb}
\usepackage{amsthm}
\usepackage{amsmath, amsfonts, amsxtra}
\usepackage{verbatim} %
\usepackage{esint}
\newtheorem{theorem}{Theorem}[section]%
\newtheorem{lemma}[theorem]{Lemma}%
\newtheorem{proposition}[theorem]{Proposition}%
\newtheorem{corollary}[theorem]{Corollary}%
\newtheorem{definition}[theorem]{Definition}%
\newtheorem{remark}[theorem]{Remark}%

\makeatletter

\newenvironment{proofP}[1]{\par
  \pushQED{{\small\it Proposition {\rm #1}} {\leavevmode
  \hbox to.77778em{%
  \hfil\vrule
  \vbox to.675em{\hrule width.6em\vfil\hrule}%
  \vrule\hfil}}}%
  \normalfont \topsep6\p@\@plus6\p@\relax
  {\it \textbf{\proofname { of Proposition {\rm #1}\@addpunct{.}}}\newline\rm}
}{%
   \begin{flushright}\popQED\end{flushright}\@endpefalse
}
\makeatother

\makeatletter

\newenvironment{proofT}[1]{\par
  \pushQED{{\small\it Theorem {\rm #1}} {\leavevmode
  \hbox to.77778em{%
  \hfil\vrule
  \vbox to.675em{\hrule width.6em\vfil\hrule}%
  \vrule\hfil}}}%
  \normalfont \topsep6\p@\@plus6\p@\relax
  {\it \textbf{\proofname { of Theorem {\rm #1}\@addpunct{.}}}\newline\rm}
}{%
   \begin{flushright}\popQED\end{flushright}\@endpefalse
}
\makeatother

\makeatletter

\newenvironment{proofL}[1]{\par
  \pushQED{{\small\it Lemma {\rm #1}} {\leavevmode
  \hbox to.77778em{%
  \hfil\vrule
  \vbox to.675em{\hrule width.6em\vfil\hrule}%
  \vrule\hfil}}}%
  \normalfont \topsep6\p@\@plus6\p@\relax
  {\it \textbf{\proofname { of Lemma {\rm #1}\@addpunct{.}}}\newline\rm}
}{%
   \begin{flushright}\popQED\end{flushright}\@endpefalse
}
\makeatother

\makeatletter

\newenvironment{proofC}[1]{\par
  \pushQED{{\small\it Corollary {\rm #1}} {\leavevmode
  \hbox to.77778em{%
  \hfil\vrule
  \vbox to.675em{\hrule width.6em\vfil\hrule}%
  \vrule\hfil}}}%
  \normalfont \topsep6\p@\@plus6\p@\relax
  {\it \textbf{\proofname { of Corollary {\rm #1}\@addpunct{.}}}\newline\rm}
}{%
   \begin{flushright}\popQED\end{flushright}\@endpefalse
}
\makeatother

\usepackage{array} 
\newenvironment{ma}{\begin{array}{>{\displaystyle}r >{\displaystyle}c >{\displaystyle}l}}{\end{array}}%

\newcommand{\aleq}[1]{\prec}
\newcommand{\ageq}[1]{\succ}
\newcommand{\aeq}[1]{\approx}

\newcommand{\Cc}{{\mathbb C}}

\newcommand{\N}{{\mathbb N}}
\newcommand{\R}{{\mathbb R}}

\newcommand{\Z}{{\mathbb Z}}
\renewcommand{\S}{{\mathbb S}}

\newcommand{\im}{ \mathbf{i}}

\newcommand{\Rz}{{\mathcal{R}}} 

\newcommand{\Sw}{{\mathcal{S}}}
\newcommand{\Sws}{{\Sw^{'}}}

\newcommand{\Hf}{{H^{\frac{n}{2}}}}

\newcommand{\diam}{\operatorname{diam}}%
\newcommand{\supp}{\operatorname{supp}}%
\newcommand{\dist}{\operatorname{dist}}%
\newcommand{\conv}[1]{\operatorname{conv}\left ( #1 \right )}%

\newcommand{\ddtz}{\ensuremath{{\left. \frac{d}{dt} \right \vert_{t=0}}}}%

\newcommand{\lap}{\Delta}
\newcommand{\laps}[1]{\lap^{\frac{#1}{2}}}
\newcommand{\lapms}[1]{\lap^{-\frac{#1}{2}}}
\newcommand{\lapn}{\lap^{\frac{n}{4}}}
\newcommand{\lapmn}{\lap^{-\frac{n}{4}}}
\newcommand{\abs}[1]{{\left \vert #1 \right \vert}}%
\newcommand{\Babs}[1]{{\Big \vert #1 \Big \vert}}%
\newcommand{\brac}[1]{{\left ( #1 \right )}}%
\newcommand{\ebrac}[1]{{\left [ #1 \right ]}}%

\newcommand{\nl}{${}$\\}

\newcommand{\ontop}[2]{{\genfrac{}{}{0pt}{}{#1}{#2}}}

\newcommand{\mvintl}{\mvint \limits}
\newcommand{\intl}{\int \limits}
\newcommand{\suml}{\sum \limits}

\newcommand{\mvint}{\fint}

\def\XXint#1#2#3{{\setbox0=\hbox{$#1{#2#3}{\int}$}%
     \vcenter{\hbox{$#2#3$}}\kern-.5\wd0}}%

\newcommand{\sref}[2]{#1.\ref{#2}}

\numberwithin{equation}{section}%

\title{Regularity of $\frac{n}{2}$-harmonic maps into spheres}
\author{Armin Schikorra}

\usepackage{anysize} 
 \marginsize{4.3cm}{4.3cm}{4.5cm}{3cm}%

\usepackage[bookmarks=false]{hyperref}	

\begin{document}
\maketitle
\begin{abstract}
\noindent We prove H\"older continuity for $\frac{n}{2}$-harmonic maps from subsets of $\R^n$ into a sphere. This extends a recent one-dimensional result by F. Da Lio and T. Rivi\`{e}re to arbitrary dimensions. The proof relies on compensation effects which we quantify adapting an approach for Wente's inequality by L. Tartar, instead of Besov-space arguments which were used in the one-dimensional case. Moreover, fractional analogues of Hodge decomposition and higher order Poincar\'e inequalities as well as several localization effects for nonlocal operators similar to the fractional laplacian are developed and applied.
\\[1ex]
\end{abstract}
\tableofcontents
\thispagestyle{empty}
\section{Introduction}
In his seminal work \cite{Hel90} F. H{\'e}lein proved regularity for harmonic maps from the two-dimensional unit disk $B_1(0) \subset \R^2$ into the $m$-dimensional sphere $\S^{m-1} \subset \R^m$ for arbitrary $m \in \N$. These maps are critical points of the functional
\[
 E_2(u) := \intl_{B_1(0) \subset \R^2} \abs{\nabla u}^2, \qquad \mbox{where } u \in W^{1,2}(B_1(0),\S^{m-1}).
\]
The importance of this result is the fact that harmonic maps in two dimensions are special cases of critical points of conformally invariant variational functionals, which play an important role in physics and geometry and have been studied for a long time: H\'elein's approach is based on the discovery of a compensation phenomenon appearing in the Euler-Lagrange equations of $E_2$, using a relation between $\operatorname{div}$-$\operatorname{curl}$ expressions and the Hardy space. This kind of relation had been discovered shortly before in the special case of determinants by S. M\"uller \cite{Mul90} and was generalized by R. Coifman, P.L. Lions, Y. Meyer and S. Semmes \cite{CLMS93}. H\'{e}lein extended his result to the case where the sphere $\S^{m-1}$ is replaced by a general target manifold developing the so-called moving-frame technique which is used in order to enforce the compensation phenomenon in the Euler-Lagrange equations \cite{Hel91}. Finally, T. Rivi\`{e}re \cite{Riv06} was able to prove regularity for critical points of general conformally invariant functionals, thus solving a conjecture by S. Hildebrandt \cite{Hil82}. He used an ingenious approach based on K. Uhlenbeck's results in gauge theory \cite{Uhl82} in order to implement $\operatorname{div}$-$\operatorname{curl}$ expressions in the Euler-Lagrange equations, a technique which can be reinterpreted as an extension of H\'elein's moving frame method; see \cite{IchEnergie}.
For more details and references we refer to H\'{e}lein's book \cite{Hel02} and the extensive introduction in \cite{Riv06} as well as \cite{RivVanc09}.\\
Naturally, it is interesting to see how these results extend to other dimensions: In the four-dimensional case, regularity can be proven for critical points of the following functional, the so-called extrinsic biharmonic maps:
\[
 E_4(u) := \intl_{B_1(0) \subset \R^4} \abs{\lap u}^2, \qquad \mbox{where } u \in W^{2,2}(B_1(0),\R^m).
\]
This was done by A. Chang, L. Wang, and P. Yang \cite{CWY99} in the case of a sphere as the target manifold, and for more general targets by P. Strzelecki \cite{Sbi03}, C. Wang \cite{Wang04} and C. Scheven \cite{Scheven08}; see also T. Lamm and T. Rivi\`{e}re's paper \cite{LR08}. More generally, for all even $n \geq 6$ similar regularity results hold, and we refer to the work of A. Gastel and C. Scheven \cite{GS09} as well as the article of P. Goldstein, P. Strzelecki and A. Zatorska-Goldstein \cite{GSZG09}.\\
In odd dimensions non-local operators appear, and only two results for dimension $n=1$ are available. In \cite{DR09Sphere}, F. Da Lio and T. Rivi\`{e}re prove H\"older continuity for critical points of the functional
\[
E_1(u) = \intl_{\R^1} \abs{\lap^{\frac{1}{4}} u}^2,\qquad \mbox{defined on distributions $u$ with finite energy and $u \in \S^{m-1}$ a.e.} 
\] 
In \cite{DR209} this is extended to the setting of general target manifolds.\\
\\
In general, we consider for $n, m \in \N$ and some domain $D \subset \R^n$ the regularity of critical points on $D$ of the functional
\begin{equation}\label{eq:energy}
E_n(v) = \intl_{\R^n} \abs{\lapn v}^2,\qquad v \in H^{\frac{n}{2}}(\R^n,\R^m),\ v \in \S^{m-1} \mbox{ a.e. in $D$.}
\end{equation}
Here, $\lapn$ denotes the operator which acts on functions $v \in L^2(\R^n)$ according to
\[
\brac{\lapn v}^\wedge(\xi) = \abs{\xi}^{\frac{n}{2}}\ v^\wedge(\xi) \qquad \mbox{for almost every $\xi \in \R^n$,}
\] 
where $()^\wedge$ denotes the application of the Fourier transform. The space $H^{\frac{n}{2}}(\R^n)$ is the space of all functions $v \in L^2(\R^n)$ such that $\lapn v \in L^2(\R^n)$. The term ``critical point'' is defined as usual:
\begin{definition}[Critical Point]\label{def:critpt}
Let $u \in \Hf(\R^n,\R^m)$, $D \subset \R^n$. We say that $u$ is a critical point of $E_n(\cdot)$ on $D$ if $u(x) \in \S^{m-1}$ for almost every $x \in D$ and
\[
 \ddtz E(u_{t,\varphi}) = 0
\]
for any $\varphi \in C_0^\infty(D,\R^m)$ where $u_{t,\varphi} \in \Hf(\R^n)$ is defined as
\[
u_{t,\varphi} = \begin{cases}
               \Pi (u+t\varphi)\qquad &\mbox{in $D$,}\\
		u\qquad &\mbox{in $\R^n \backslash D$.}\\
               \end{cases}
\]
Here, $\Pi$ denotes the orthogonal projection from a tubular neighborhood of $\S^{m-1}$ into $\S^{m-1}$ defined as $\Pi(x) = \frac{x}{\abs{x}}$.
\end{definition}
If $n$ is an even number, the domain of $E_n(\cdot)$ is just the classic Sobolev space $H^{\frac{n}{2}}(\R^n) \equiv W^{\frac{n}{2},2}(\R^n)$, for odd dimensions this is a fractional Sobolev space (see Section \ref{sec:fracsobspace}). Functions in $H^{\frac{n}{2}}(\R^n)$ can contain logarithmic singularities (cf. \cite{Frehse73}) but this space embeds continuously into $BMO(\R^n)$, and even only slightly improved integrability or more differentiability would imply continuity.\\
In the light of the existing results in even dimensions and in the one-dimensional case, one may expect that similar regularity results should hold for any dimension. As a first step in that direction, we establish regularity of $n/2$-harmonic maps into the sphere.
\begin{theorem}\label{th:regul}
For any $n \geq 1$, critical points $u \in \Hf(\R^2)$ of $E_n$ on a domain $D$ are locally H\"older continuous in $D$.
\end{theorem}
Note that here -- in contrast to \cite{DR09Sphere} -- we work on general domains $D \subseteq \R^n$. This is motivated by the facts that H\"older continuity is a local property and that $\lapn$ (though it is a non-local operator) still behaves ``pseudo-local'': We impose our conditions (here: being a critical point and mapping into the sphere) only in some domain $D \subset \R^n$, and still get interior regularity within $D$.\\
Let us comment on the strategy of the proof. As said before, in all even dimensions the key tool for proving regularity is the discovery of \emph{compensation phenomena} built into the respective Euler-Lagrange equation. For example, critical points $u \in W^{1,2}(D,\S^{m-1})$ of $E_2$ satisfy the following Euler-Lagrange equation \cite{Hel90}
\begin{equation}\label{eq:ELhel}
\lap u^i = u^i \abs{\nabla u}^2, \qquad \mbox{weakly in $D$, \quad for all $i =1\ldots m$.}
\end{equation}
For mappings $u \in W^{1,2}(\R^2,\S^{m-1})$ this is a critical equation, as the right-hand side seems to lie only in $L^1$: If we had no additional information, it would seem as if the equation admitted a logarithmic singularity (for examples see, e.g., \cite{Riv06}, \cite{Frehse73}). But, using the constraint $\abs{u} \equiv 1$, one can rewrite the right-hand side of \eqref{eq:ELhel} as
\[
 u^i \abs{\nabla u}^2
= \sum_{j=1}^m \brac{u^i \nabla u^j - u^j \nabla u^i}\cdot \nabla u^j
= \sum_{j=1}^m \brac{\partial_1 B_{ij}\ \partial_2 u^j - \partial_2 B_{ij}\ \partial_1 u^j}
\]
where the $B_{ij}$ are chosen such that $\partial_1 B_{ij} = u^i \partial_2 u^j - u^j \partial_2 u^i$, and $-\partial_2 B_{ij} = u^i \partial_1 u^j - u^j \partial_1 u^i$, a choice which is possible due to Poincar\'e's Lemma and because \eqref{eq:ELhel} implies $\operatorname{div} \brac{u^i \nabla u^j - u^j \nabla u^i} = 0$ for every $i,j = 1\ldots m$. Thus, \eqref{eq:ELhel} transforms into
\begin{equation}\label{eq:ELhelTr}
 \lap u^i = \sum_{j=1}^m \brac{\partial_1 B_{ij}\ \partial_2 u^j - \partial_2 B_{ij}\ \partial_1 u^j},
\end{equation}
a form whose right-hand side exhibits a compensation phenomenon which in a similar way already appeared in the so-called Wente inequality \cite{Wente69}, see also \cite{BC84}, \cite{Tartar85}. In fact, the right-hand side belongs to the Hardy space (cf. \cite{Mul90}, \cite{CLMS93}) which is a proper subspace of $L^1$ with enhanced potential theoretic properties. Namely, members of the Hardy space behave well with Calder\'on-Zygmund operators, and by this one can conclude continuity of $u$.\\
An alternative and for our purpose more viable way to describe this can be found in L. Tartar's proof \cite{Tartar85} of Wente's inequality: Assume we have for $a,b \in L^2(\R^2)$ a solution $w \in H^1(\R^2)$ of
\begin{equation}\label{eq:wentepde}
\lap w = \partial_1 a\ \partial_2 b -  \partial_2 a\ \partial_1 b \qquad \mbox{weakly in $\R^2$}.
\end{equation}
Taking the Fourier-Transform on both sides, this is (formally) equivalent to
\begin{equation}\label{eq:tartarfourierpde}
\abs{\xi}^2 w^\wedge(\xi) = c \intl_{\R^2} a^\wedge(x)\ b^\wedge(\xi-x) \left ( x_1 (\xi_2 - x_2) - x_2 (\xi_1 - x_1)  \right )\ dx,\ \qquad \mbox{for $\xi \in \R^2$}.
\end{equation}
Now the compensation phenomena responsible for the higher regularity of $w$ can be identified with the following inequality:
\begin{equation}\label{eq:tartarcompensation}
\abs{x_1 (\xi_2 - x_2) - x_2 (\xi_1 - x_1)} \leq \abs{\xi}\abs{x}^{\frac{1}{2}}\abs{\xi-x}^\frac{1}{2}.
\end{equation}
Observe, that $\abs{x}$ as well as $\abs{\xi - x}$ appear to the power $1/2$, only. Interpreting these factors as Fourier multipliers, this means that only ``half of the gradient'', more precisely $\lap^{\frac{1}{4}}$, of $a$ and $b$ enters the equation, which implies that the right-hand side is a ``product of lower order'' operators. In fact, plugging \eqref{eq:tartarcompensation} into \eqref{eq:tartarfourierpde}, one can conclude $w^\wedge \in L^1(\R^2)$  just by H\"older's and Young's inequality on Lorentz spaces -- consequently one has proven continuity of $w$, because the inverse Fourier transform maps $L^1$ into $C^0$. As explained earlier, \eqref{eq:ELhel} can be rewritten as \eqref{eq:ELhelTr} which has the form of \eqref{eq:wentepde}, thus we have continuity for critical points of $E_2$, and by a bootstraping argument (see \cite{Tomi69}) one gets analyticity of these points.\\
\\
As in Theorem~\ref{th:regul} we prove only interior regularity, it is natural to work with localized Euler-Lagrange equations which look as follows, see Section \ref{sec:eleqn}:
\begin{lemma}[Euler-Lagrange Equations]
Let $u \in \Hf(\R^n)$ be a critical point of $E_n$ on a domain $D \subset \R^n$. Then, for any cutoff function $\eta \in C_0^\infty(D)$, $\eta \equiv 1$ on an open neighborhood of a ball $\tilde{D} \subset D$ and $w := \eta u$, we have
\begin{equation}\label{eq:eleqn}
 -\intl_{\R^n} w^i\ \lapn w^j\ \lapn \psi_{ij} = \intl_{\R^n} \lapn w^j\ H(w^i,\psi_{ij}) - \intl_{\R^n} a_{ij} \psi_{ij} , \quad \mbox{for any $\psi_{ij} = -\psi_{ji} \in C_0^\infty(\tilde{D})$},
\end{equation}
where $a_{ij} \in L^2(\R^n)$, $i,j = 1,\ldots,m$, depend on the choice of $\eta$. Here, we adopt Einstein's summation convention. Moreover, $H(\cdot,\cdot)$ is defined on $H^{\frac{n}{2}}(\R^n) \times H^{\frac{n}{2}}(\R^n)$ as 
\begin{equation}\label{eq:defH}
H(a,b) := \lapn (a b) - a \lapn b - b \lapn a, \qquad \mbox{for $a,b \in \Hf(\R^n)$}.
\end{equation}
Furthermore, $u \in \S^{m-1}$ on $D$ implies the following structure equation
\begin{equation}\label{eq:structure}
w^i \cdot \lapn w^i = -\frac{1}{2} H(w^i,w^i) + \frac{1}{2}\lapn \eta^2 \qquad \mbox{a.e. in $\R^n$.}
\end{equation}
\end{lemma}
Similar in its spirit to \cite{DR09Sphere} we use that \eqref{eq:eleqn} and \eqref{eq:structure} together control the full growth of $\lapn w$, though here we use a different argument applying an analogue of Hodge decomposition to show this, see below. Note moreover that as we have localized our Euler-Lagrange equation, we do not need further rewriting of the structure condition \eqref{eq:structure} as was done in \cite{DR09Sphere}.\\
Whereas in \eqref{eq:wentepde} the compensation phenomenon stems from the structure of the right-hand side, here it comes from the leading order term $H(\cdot,\cdot)$ appearing in \eqref{eq:eleqn} and \eqref{eq:structure}. This can be proved by Tartar's approach \cite{Tartar85}, using essentially only the following elementary ``compensation inequality'' similar in its spirit to \eqref{eq:tartarcompensation} 
\begin{equation}\label{eq:ourcompensation}
\abs{\abs{x-\xi}^p - \abs{\xi}^p - \abs{x}^p} \leq C_p \begin{cases} 
                                                    \abs{x}^{p-1} \abs{\xi} + \abs{\xi}^{p-1} \abs{x}, \qquad &\mbox{if $p > 1$,}\\
						    \abs{x}^{\frac{p}{2}} \abs{\xi}^{\frac{p}{2}}, \qquad &\mbox{if $p \in (0,1]$.}\\
                                                   \end{cases} 
\end{equation}
More precisely, we will prove in Section~\ref{sec:tart}
\begin{theorem}\label{th:compensation}
For $H$ as in \eqref{eq:defH} and $u,v \in \Hf(\R^n)$ one has
\[
 \Vert H(u,v) \Vert_{L^2(\R^n)} \leq C\ \Vert \brac{\lapn u}^\wedge \Vert_{L^2(\R^n)}\ \Vert \brac{ \lapn v}^\wedge \Vert_{L^{2,\infty}(\R^n)}.
\]
\end{theorem}
An equivalent compensation phenomenon was observed in the case $n=1$ in \cite{DR09Sphere}\footnote{In fact, all compensation phenomena appearing in \cite{DR09Sphere} can be proven by our adaption of Tartar's method using simple compensation inequalities, thus avoiding the use of paraproduct arguments (but at the expense of using the theory of Lorentz spaces).}. Note that interpreting again the terms of \eqref{eq:ourcompensation} as Fourier multipliers, it seems as if this equation (and as a consequence Theorem~\ref{th:compensation}) estimates the operator $H(u,v)$ by products of lower order operators applied to $u$ and $v$. Here, by ``products of lower order operators'' we mean products of operators whose differential order is strictly between zero and $\frac{n}{2}$ and where the two operators together give an operator of order $\frac{n}{2}$. In fact, this is exactly what happens in special cases, e.g. if we take the case $n=4$ where $\lapn = \lap$:
\[
 H(u,v) = 2 \nabla u \cdot \nabla u \qquad \mbox{if $n = 4$}.
\]
Another case we will need to control is the case where $u = P$ is a polynomial of degree less than $\frac{n}{2}$. As (at least formally) $\lapn P = 0$ this is to estimate
\[
 \lapn (Pv) - P \lapn v.
\]
This case is not contained in Theorem~\ref{th:compensation} as a non-zero polynomial does not belong to $\Hf(\R^n)$. Obviously, in the one-dimensional case $P$ is only a constant, and thus $H(P,v) \equiv 0$. In higher dimensions, as we will show in Proposition~\ref{pr:lapsmonprod2}, $H(P,v)$ is still a product of lower order operators.\\
As we are going to show in Section~\ref{ss:loolocwell}, products of lower order operators (in the way this term is defined above) ``localize well''. By that we mean that the $L^2$-norm of such a product evaluated on a ball is estimated by the product of $L^2$-norms of $\lapn$ applied to the factors evaluated at a slightly bigger ball, up to harmless error terms. As a consequence, one expects this to hold as well for the term $H(u,v)$, and in fact, we can show the following ``localized version'' of Theorem~\ref{th:compensation}, proven in Section~\ref{sec:localTart}.
\begin{theorem}[Localized Compensation Results]\label{th:localcomp}
There is a uniform constant $\gamma > 0$ depending only on the dimension $n$, such that the following holds. Let $H(\cdot,\cdot)$ be defined as in \eqref{eq:defH}. For any $v \in \Hf(\R^n)$ and $\varepsilon > 0$ there exist constants $R > 0$ and $\Lambda_1 > 0$ such that for any ball $B_r(x) \subset \R^n$, $r \in (0,R)$,
\[
\Vert H(v,\varphi) \Vert_{L^2(B_r(x))} \leq \varepsilon\ \Vert \lapn \varphi \Vert_{L^2(\R^n)} \qquad \mbox{for any $\varphi \in C_0^\infty(B_r(x))$},
\]
and
\[
\Vert H(v,v) \Vert_{L^2(B_r(x))} \leq \varepsilon\ [[v]]_{B_{\Lambda_1 r}(x)} + C_{\varepsilon,v} \sum_{k=-\infty}^\infty 2^{-\gamma \abs{k}} [[v]]_{B_{2^{k+1} r}(x) \backslash B_{2^{k} r}(x)}.
\]
Here, $[[v]]_{A}$ is a pseudo-norm, which in a way measures the $L^2$-norm of $\lapn v$ on $A \subset \R^n$. More precisely, for odd $n$
\[
[[v]]_{A} := \Vert \lapn v \Vert_{L^2(A)} + \brac{\intl_A \intl_A \abs{x-y}^{-n-1}\ \abs{\nabla^{\frac{n-1}{2}}v(x) - \nabla^{\frac{n-1}{2}} v(y)}^2\ dx\ dy}^{\frac{1}{2}},
\]
and for even $n$ we set $[[v]]_{A} := \Vert \lapn v \Vert_{L^2(A)} + \Vert \nabla^{\frac{n}{2}} v \Vert_{L^2(A)}$.
\end{theorem}%
As mentioned before, by the structure of our Euler-Lagrage equations, these local estimates control the local growth of the $\frac{n}{4}$-operator of any critical point, as we will show using an analogue of the Hodge decomposition in the fractional case, see Section \ref{ss:hodge}.
\begin{theorem}\label{th:hodge}
There are uniform constants $\Lambda_2 > 0$ and $C > 0$ such that the following holds: For any $x \in \R^n$ and any $r > 0$ we have for every $v \in L^2(\R^n)$, $\supp v \subset B_r(x)$, 
\[
\Vert v \Vert_{L^2(B_r(x))} \leq C \sup_{\varphi \in C_0^\infty(B_{\Lambda_2 r}(x))} \frac{1}{\Vert \lapn \varphi \Vert_{L^2(\R^n)}}\ \intl_{\R^n} v\ \lapn \varphi.
\]
\end{theorem}
Then, by an iteration technique adapted from the one in \cite{DR09Sphere} (see the appendix) we conclude in Section~\ref{sec:growth} that the critical point $u$ of $E_n$ lies in a Morrey-Campanato space, which implies H\"older continuity.\\
As for the sections not mentioned so far: In Section~\ref{sec:basics} we will cover some basic facts on Lorentz and Sobolev spaces. In Section~\ref{sec:poincmv} we will prove a fractional Poincar\'{e} inequality with a mean value condition of arbitrary order. In Section~\ref{sec:locEf} various localizing effects are studied. In Section~\ref{sec:homognormhn2} we compare two pseudo-norms $\Vert \lapn v \Vert_{L^2(A)}$ and $[v]_{\frac{n}{2},A}$ of $H^{\frac{n}{2}}$, and finally, in Section~\ref{sec:growth}, Theorem~\ref{th:regul} is proved.\\[1em]

Finally, let us remark the following two points: As we cut off the critical points $u$ to bounded domains, the assumption $u \in L^2(\R^n)$ is not necessary, one could, e.g., assume $u \in L^\infty(\R^n)$, $\lapn u \in L^2(\R^n)$, thus regaining a similar ``global'' result as in \cite{DR09Sphere}. 
Observe moreover, that the application of a cut-off function within $D$ to the critical point $u$ is a very brute operation, which nevertheless suffices our purposes as in this note we are only interested in \emph{interior} regularity. For the analysis of the boundary behaviour of $u$ one probably would need a more careful cut-off argument.\\[1em]

We will use fairly standard \emph{notation}:\\
As usual, we denote by $\Sw \equiv \Sw(\R^n)$ the Schwartz class of all smooth functions which at infinity go faster to zero than any quotient of polynomials, and by $\Sws \equiv \Sws(\R^n)$ its dual. For a set $A \subset \R^n$ we will denote its $n$-dimensional Lebesgue measure by $\abs{A}$, and $rA$, $r > 0$, will be the set of all points $rx \in \R^n$ where $x \in A$. By $B_r(x) \subset \R^n$ we denote the open ball with radius $r$ and center $x \in \R^n$. If no confusion arises, we will abbreviate $B_r \equiv B_r(x)$. When we speak of a multiindex $\alpha$ we will usually mean $\alpha = (\alpha_1,\ldots,\alpha_n) \in \brac{\N \cup \{0\}}^n \equiv \brac{\N_0}^n$ with length $\abs{\alpha} := \sum_{i=1}^n \alpha_i$. For such a multiindex $\alpha$ and $x = (x_1,\ldots,x_n)^T \in \R^n$ we denote by $x^\alpha = \prod_{i=1}^n \brac{x_i}^{\alpha_i}$ where we set $(x_i)^0 := 1$ even if $x_i = 0$. For a real number $p \geq 0$ we denote by $\lfloor p \rfloor$ the biggest integer below $p$ and by $\lceil p \rceil$ the smallest integer above $p$. If $p \in [1,\infty]$ we usually will denote by $p'$ the H\"older conjugate, that is $\frac{1}{p} + \frac{1}{p'} = 1$. By $f \ast g$ we denote the convolution of two functions $f$ and $g$. As mentioned before, we will denote by $f^\wedge$ the Fourier transform and by $f^\vee$ the inverse Fourier transform, which on the Schwartz class $\Sw$ are defined as
\[
 f^\wedge(\xi) := \intl_{\R^n} f(x)\ e^{-2\pi \im\ x\cdot \xi}\ dx,\quad f^\vee(x) := \intl_{\R^n} f(\xi)\ e^{2\pi \im\ \xi\cdot x}\ d\xi.
\]
By $\im$ we denote here and henceforth the imaginary unit $\im^2 = -1$. $\Rz$ is the Riesz operator which transforms $v \in \Sw(\R^n)$ according to $(\Rz v)^\wedge(\xi) := \im \frac{\xi}{\abs{\xi}} v^\wedge(\xi)$. More generally, we will speak of a zero-multiplier operator $M$, if there is a function $m \in C^\infty(\R^n \backslash \{0\})$ homogeneous of order $0$ and such that $(Mv)^\wedge(\xi) = m(\xi)\ v^\wedge(\xi)$ for all $\xi \in \R^n \backslash \{0\}$. For a measurable set $D \subset \R^n$, we denote the integral mean of an integrable function $v: D \to \R$ to be $(v)_D \equiv \mvint_D v \equiv \frac{1}{\abs{D}} \int_D v$. Lastly, our constants -- usually denoted by $C$ or $c$ --  can possibly change from line to line and usually depend on the space dimensions involved, further dependencies will be denoted by a subscript, though we will make no effort to pin down the exact value of those constants. If we consider the constant factors to be irrelevant with respect to the mathematical argument, for the sake of simplicity we will omit them in the calculations, writing $\aleq{}$, $\ageq{}$, $\aeq{}$ instead of $\leq$, $\geq$ and $=$.\\

\vspace{2ex}\noindent
{\bf Acknowledgment.} The author would like to thank Francesca Da Lio and Tristan Rivi\`{e}re for introducing him to the topic, and Pawe{\l} Strzelecki for suggesting to extend the results of \cite{DR09Sphere} to higher dimensions. Moreover, he is very grateful to his supervisor Heiko von der Mosel for the constant support and encouragement, as well as for many comments and remarks on the drafts of this work. The author is supported by the Studienstiftung des Deutschen Volkes.
\section{Lorentz-, Sobolev Spaces and Cutoff Functions}\label{sec:basics}
\subsection{Interpolation}
In the following section we will state some fundamental properties of interpolation theory, which will be used to ``translate'' results from classical Lebesgue and Sobolev spaces into the setting of Lorentz and fractional Sobolev spaces. For more on interpolation spaces, we refer to Tartar's monograph \cite{Tartar07}.\\
There are different methods of interpolation. We state here the so-called $K$-Method, only.
\begin{definition}[Interpolation by the $K$-Method]\label{def:KMethod}
(Compare \cite[Definition 22.1]{Tartar07})\\
Let $X,Y$ be normed spaces with respective norms $\Vert \cdot \Vert_X$, $\Vert \cdot \Vert_Y$ and assume that  $Z = X+Y$ is a normed space with norm
\[
\Vert z \Vert_{X+Y} := \inf_{z = x+y} \brac{\Vert x\Vert_X + \Vert y \Vert_Y}.
\]
For $t \in (0,\infty)$ and $z \in X+Y$ we denote
\[
K(z,t) = \inf_{\ontop{z = x +y}{x \in X, y \in Y}} \Vert x \Vert_X + t \Vert y \Vert_Y,
\]
and for $\theta \in (0,1)$ and $q \in [1,\infty]$,
\[
\Vert z \Vert_{[X,Y]_{\theta,q}}^q := \intl_{t=0}^\infty \brac{t^{-\theta}\ K(z,t)}^q \frac{dt}{t}.
\]
The space $[X,Y]_{\theta,q}$ with norm $\Vert \cdot \Vert_{[X,Y]_{\theta,q}}$ is then defined as every $z \in X+Y$ such that $\Vert z \Vert_{[X,Y]_{\theta,q}} < \infty$.
\end{definition}

\begin{proposition}\label{pr:bs:xyqintoxyqs}
(Compare \cite[Lemma 22.2]{Tartar07})\\
Let $X,Y,Z$ be as in Definition \ref{def:KMethod}. If $1 \leq q < q' \leq \infty$, $\theta \in (0,1)$, then
\[
[X,Y]_{\theta,q} \subset [X,Y]_{\theta,q'},
\]
and the embedding is continuous.
\end{proposition}
\begin{proofP}{\ref{pr:bs:xyqintoxyqs}}
Fix $\theta \in (0,1)$. Denote
\[
E_p := [X,Y]_{\theta,p}, \quad \mbox{for }p \in [1,\infty].
\]
Then for $q < \infty$, $t_0 > 0$, using that $K(z,t)$ is monotone rising in $t$,
\[
\begin{ma}
\Vert z \Vert_{E_q}^q &=& \intl_{t=0}^\infty t^{-\theta q} \brac{K(z,t)}^q\ \frac{dt}{t}\\
&\geq& \intl_{t=t_0}^\infty t^{-\theta q} \brac{K(z,t)}^q\ \frac{dt}{t}\\
&\geq& \brac{K(z,t_0)}^q \frac{\brac{t_0}^{-\theta q}}{\theta q},
\end{ma}
\]
that is
\[
t_0^{-\theta}\ K(z,t_0)  \aleq{} \Vert z \Vert_{E_q}, \quad \mbox{for every $t_0 > 0$},
\]
which implies
\begin{equation}\label{eq:bs:xyinftyxyq}
\Vert z \Vert_{E_\infty} \leq C_{\theta,q} \Vert z \Vert_{E_q} \quad \mbox{for any $q \in [1,\infty]$}.
\end{equation}
Thus, by H\"older inequality for $\infty > q' > q$,
\[
\begin{ma}
\Vert z \Vert_{E_{q'}}^{q'} &=& \Vert t^{-\theta} K(z,t) \Vert_{L^{q'}\brac{(0,\infty),\frac{dt}{t}}}^{q'}\\
&\aleq{}& \Vert z \Vert_{E_\infty}^{q'-q}\ \Vert z \Vert_{E_q}^q\\
&\overset{\eqref{eq:bs:xyinftyxyq}}{\aleq{}}& \Vert z \Vert_{E_q}^{q'}.
\end{ma}
\] 
\end{proofP}

The following two fundamental lemmata tell us how linear and bounded or linear and compact operators defined on the spaces $X$ and $Y$ from Definition \ref{def:KMethod} behave on the interpolated spaces.
\begin{lemma}[Interpolation Theorem]\label{la:interpolation}
(See \cite[Lemma 22.3]{Tartar07})\\
Let $X_1,Y_1,Z_1$, $X_2,Y_2,Z_2$ be as in Definition \ref{def:KMethod}. Assume there is a linear operator $T$ defined on $Z = X+Y$ such that $T: X_1 \to X_2$ and $T: Y_1 \to Y_2$ and assume there are constants $\Lambda_X, \Lambda_Y > 0$ such that
\begin{equation}\label{eq:interpol:Tx1x2y1y2}
\Vert T \Vert_{\mathcal{L}(X_1,X_2)} \leq \Lambda_X,\quad \Vert T \Vert_{\mathcal{L}(Y_1,Y_2)} \leq \Lambda_Y.
\end{equation}
Denote for $\theta \in (0,1)$ and $q \in [1,\infty]$, $E_1 := [X_1,Y_1]_{\theta,q}$ and $E_2 := [X_2,Y_2]_{\theta,q}$.
Then $T$ is a linear, bounded operator $T: E_1 \to E_2$ such that 
\[
\Vert T \Vert_{\mathcal{L}(E_1,E_2)} \leq \Lambda_X^{1-\theta} \Lambda_Y^{\theta}.
\]
\end{lemma}
\begin{proofL}{\ref{la:interpolation}}
Denote by $K_1$, $K_2$ the $K(\cdot,\cdot)$ used to define $E_1$ and $E_2$, respectively. For $z \in E_1$ and any decomposition $z = x_1 + y_1$, $x_1 \in X_1$, $y_1 \in Y_1$ we have
\[
\begin{ma}
t^{-\theta}K_2(Tz,t) &\leq& t^{-\theta} \brac{\Vert Tx_1 \Vert_{X_2} + t \Vert Ty_1 \Vert_{Y_2}}\\
&\overset{\eqref{eq:interpol:Tx1x2y1y2}}{\leq}& \Lambda_X^{1-\theta} \Lambda_Y^\theta \ \brac{\frac{\Lambda_Y}{\Lambda_X} t}^{-\theta}  \brac{\Vert x_1 \Vert_{X_1} + t \frac{\Lambda_Y}{\Lambda_X} \Vert y_1 \Vert_{Y_1}}.\\
\end{ma}
\]
Taking the infimum over all decompositions $z = x_1 + y_1$, this implies for $\gamma := \frac{\Lambda_Y}{\Lambda_X} > 0$,
\[
t^{-\theta}K_2(Tz,t) \leq \Lambda_X^{1-\theta}\ \Lambda_Y^\theta\ (\gamma t)^{-\theta} K_1(z,\gamma t).
\]
Using the definition of $E_1,E_2$, we have shown
\[
\Vert Tz \Vert_{E_2} \leq \Lambda_X^{1-\theta}\ \Lambda_Y^\theta\ \Vert z \Vert_{E_1}.
\]
\end{proofL}

\begin{lemma}[Compactness]\label{la:interpol:compact}
(See \cite[Lemma 41.4]{Tartar07})\\
Let $X,Y,Z$ be as in Definition \ref{def:KMethod}. Let moreover $G$ be a Banach space and assume there is an operator $T$ defined on $Z = X + Y$ such that $T: X \to G$ is linear and continuous and $T: Y \to G$ is linear and compact. Then for any $\theta \in (0,1)$, $q \in [1,\infty]$, $T: [X,Y]_{\theta,q} \to G$ is compact.
\end{lemma}
\begin{proofL}{\ref{la:interpol:compact}}
Fix $\theta \in (0,1)$. By Proposition~\ref{pr:bs:xyqintoxyqs} it suffices to prove the compactness of the embedding for $q = \infty$. Set $E := [X,Y]_{\theta,\infty}$. We denote by $\Lambda$ the norm of $T$ as a linear operator from $X$ to $G$.\\
Let $z_k \in E$ and assume that 
\begin{equation}\label{eq:int:zkeleq1}
\Vert z_k \Vert_{E} \leq 1 \quad \mbox{for any $k \in \N$.}
\end{equation}
If there are infinitly many $z_k = 0$, there is nothing to prove, so assume that $z_k \neq 0$ for all $k \in \N$. Pick for any $k,n \in \N$, $x^n_k, y^n_k$ such that $x^n_k + y^n_k = z_k$ and
\[
\Vert x^n_k \Vert + \frac{1}{n} \Vert y^n_k \Vert \leq 2K(z_k,\frac{1}{n}) \overset{\eqref{eq:int:zkeleq1}}{\leq} 2 \frac{1}{n^\theta}.
\] 
Consequently, for any $k,l,n \in \N$,
\[
\begin{ma}
\Vert T z_k - T z_l \Vert_G &\leq& \Vert T (x^n_k -x^n_l) \Vert_G + \Vert T(y^n_k - y^n_l) \Vert_G\\
&\leq& \Lambda \brac{\Vert x^n_k \Vert_X + \Vert x^n_l \Vert_X} + \Vert T(y^n_k - y^n_l) \Vert_G\\
&\leq& \frac{4\Lambda}{n^\theta} + \Vert T(y^n_k - y^n_l) \Vert_G,
\end{ma}
\]
and
\begin{equation}\label{eq:int:ynk}
 \Vert y^n_k \Vert_Y \leq 2 n^{1-\theta} \quad \mbox{for any $k,n \in \N$}
\end{equation}
Now we apply a Cantor diagonal sequence argument as follows: Set
\[
 \left (k_{i,0} \right )_{i=1}^\infty := \left (i \right )_{i=1}^\infty
\]
and choose for a given sequence $\left (k_{i,n} \right )_{i=1}^\infty$ a subsequence $\left (k_{i,n+1} \right )_{i=1}^\infty$ such that
\[
 k_{i,n} = k_{i,n+1} \quad \mbox{for any $1 \leq i \leq n$}
\]
and
\[
 \Vert T(y^{n+1}_{k_{i,n+1}} - y^{n+1}_{k_{j,n+1}}) \Vert_G \leq \frac{1}{n+1} \quad \mbox{for any $i,j \geq n+1$}.
\]
The latter is possible, as $T$ is a compact operator from $Y$ to $G$ and \eqref{eq:int:ynk} implies for any fixed $n+1 \in \N$ a uniform bound of $y^{n+1}_{k_i,n}$, $i \in \N$.\\
Finally for any $1 < i < j < \infty$, setting $k_i := k_{i,i+1}$
\[
 \Vert T z_{k_i} - T z_{k_j} \Vert_G \leq \frac{4\Lambda}{i^\theta} + \frac{1}{i},
\]
which implies convergence for $i \to \infty$.
\end{proofL}

\subsection{Lorentz Spaces}
In this section, we recall the definition of Lorentz spaces, which are a refinement of the standard Lebesgue-spaces. For more on Lorentz spaces, the interested reader might consider \cite{Hunt66}, \cite{Ziemer}, \cite[Section 1.4]{GrafC08}.
\begin{definition}[Lorentz Space]
Let $f: \R^n \to \R$ be a Lebesgue-measurable function. We denote
\[
d_f(\lambda) := \abs{\{ x \in \R^n\ :\ \abs{f(x)} > \lambda\}}.
\]
The decreasing rearrangement of $f$ is the function $f^\ast$ defined on $[0,\infty)$ by
\[
f^\ast(t) := \inf \{s > 0:\ d_f(s) \leq t\}.
\]
For $1 \leq p \leq \infty$, $1 \leq q \leq \infty$, the Lorentz space $L^{p,q} \equiv L^{p,q}(\R^n)$, is the set of measurable functions $f: \R^n \to \R$ such that $\Vert f \Vert_{L^{p,q}} < \infty$, where
\[
\Vert f \Vert_{L^{p,q}} := \begin{cases}
													\left (\intl_0^\infty \left (t^{\frac{1}{p}} f^\ast (t) \right )^q \frac{dt}{t} \right )^\frac{1}{q}, \quad &\mbox{if $q < \infty$,}\\
													\sup_{t > 0} t^{\frac{1}{p}} f^\ast (t),\quad &\mbox{if $q = \infty$, $p < \infty$},\\
													\Vert f \Vert_{L^\infty(\R^n)},\quad &\mbox{if $q = \infty$, $p = \infty$}.
														\end{cases}
\]
Observe that $\Vert \cdot \Vert_{L^{p,q}}$ does \emph{not} satisfy the triangle inequality.
\end{definition}
\begin{remark}
We have not defined the space $L^{\infty,q}$ for $q \in [1,\infty)$. For the sake of overview, whenever a result on Lorentz spaces is stated in a way that $L^{p,q}$ for $p = \infty$, $q \in [1,\infty]$ is admissible, we in fact only claim that result for $p = \infty$, $q = \infty$.
\end{remark}
An alternative definition of Lorentz spaces using interpolation can be stated as follows.
\begin{lemma}\label{la:lpqinterpoldef}
(See \cite[Lemma 22.6, Theorem 26.3]{Tartar07}\\
Let $q \in [1,\infty]$. For $1 < p < \infty$
\[
L^{p,q} = \left [L^1(\R^n), L^\infty(\R^n)\right ]_{1-\frac{1}{p},q},
\]
for $2 < p < \infty$
\[
L^{p,q} = \left [L^2(\R^n), L^\infty(\R^n)\right ]_{1-\frac{2}{p},q},
\]
and finally for $1 < p < 2$,
\[
L^{p,q} = \left [L^1(\R^n), L^2(\R^n)\right ]_{2-\frac{2}{p},q},
\]
and the norms of the respective spaces are equivalent.
\end{lemma}%
For H\"older's inequality on Lorentz spaces, we will need moreover the following result on the decreasing rearrangement.
\begin{proposition}\label{pr:fgastleqfaga}(See \cite[Proposition 1.4.5]{GrafC08})\\
For any $f,g \in \Sw(\R^n)$ and any $t > 0$,
\[
(fg)^\ast(2t) \leq f^\ast(t)\ g^\ast(t).
\]
\end{proposition}
\begin{proofP}{\ref{pr:fgastleqfaga}}
We have for any $s$, $s_1$, $s_2 > 0$ such that $s = s_1 s_2$,
\[
\{x \in \R^n:\ \abs{f(x)g(x)} > s \} \subset \{x \in \R^n:\ \abs{f(x)} > s_1 \} \cup \{x \in \R^n:\ \abs{g(x)} > s_2 \}, 
\]
so
\[
d_{fg}(s) \leq d_f(s_1) + d_f(s_2).
\]
Consequently, for any $t > 0$,
\[
\{s > 0:\ d_{fg}(s) \leq 2t\} \supset \{s = s_1 s_2 > 0:\ d_{f}(s_1) \leq t,\ d_{g}(s_2) \leq t\},
\]
which implies
\[
 (fg)^\ast(2t) \leq \inf \{s = s_1 s_2 > 0:\ d_{f}(s_1) \leq t,\ d_{g}(s_2) \leq t\}.
\]
Of course,
\[
 d_f(f^\ast(t) + \frac{1}{k}) \leq t,\quad \mbox{and}\quad d_g(g^\ast(t) + \frac{1}{k}) \leq t \quad \mbox{for any $k \in \N$},
\]
so for any $k \in \N$
\[
\inf \{s = s_1 s_2 > 0:\ d_{f}(s_1) \leq t,\ d_{g}(s_2) \leq t\} \leq (f^\ast(t) + \frac{1}{k})g^\ast(t) + \frac{1}{k}).
\]
We conclude by letting $k$ go to $\infty$.
\end{proofP}

\begin{proposition}[Basic Lorentz Space Operations]\label{pr:dl:lso}
Let $f \in L^{p_1,q_1}$ and $g \in L^{p_2,q_2}$, $1 \leq p_1,p_2,q_1,q_2 \leq \infty$.
\begin{itemize}
 \item [(i)] If $\frac{1}{p_1} + \frac{1}{p_2} = \frac{1}{p} \in [0,1]$ and $\frac{1}{q_1} + \frac{1}{q_2} = \frac{1}{q}$ then $f g \in L^{p,q}$ and
 \[
 \Vert fg \Vert_{L^{p,q}} \aleq{p,q} \Vert f \Vert_{L^{p_1,q_1}}\ \Vert g \Vert_{L^{p_2,q_2}}. 
 \]
 \item [(ii)] If $\frac{1}{p_1} + \frac{1}{p_2} - 1 = \frac{1}{p} > 0$ and $\frac{1}{q_1} +  \frac{1}{q_2} = \frac{1}{q}$ then $f \ast g \in L^{p,q}$ and
 \[
 \Vert f\ast g \Vert_{L^{p,q}} \aleq{p,q} \Vert f \Vert_{L^{p_1,q_1}}\ \Vert g \Vert_{L^{p_2,q_2}}. 
 \]
 
 \item [(iii)] For $p_1 \in (1,\infty)$, $f$ belongs to $L^{p_1}(\R^n)$ if and only if $f \in L^{p_1,p_1}$. The ''norms`` of $L^{p_1,p_1}$ and $L^{p_1}$ are equivalent.
 \item [(iv)] If $p_1 \in (1,\infty)$ and $q \in [q_1,\infty]$ then also $f \in L^{p_1,q}$.
 \item [(v)] Finally, $\frac{1}{\abs{\cdot}^\lambda} \in L^{\frac{n}{\lambda},\infty}$, whenever $\lambda \in (0,n)$.
 \end{itemize}
\end{proposition}
\begin{proofP}{\ref{pr:dl:lso}}
As for (i), this is proved using classical H\"older inequality and Proposition~\ref{pr:fgastleqfaga} in the following way:
\[
\begin{ma}
 &&\intl_0^\infty \brac{t^{\theta} (fg)^\ast(t) }^q \frac{dt}{t}\\
&\overset{\sref{P}{pr:fgastleqfaga}}{\aleq{}}& \intl_0^\infty \brac{t^{\theta q_1} f^\ast(t)^{q_1}\ t^{-1} }^{\frac{q}{q_1}}\ \brac{t^{\theta q_2} g^\ast(t)^{q_2}\ t^{-1} }^{\frac{q}{q_2}}\ dt\\
&\leq& \brac{\intl_0^\infty t^{\theta q_1} f^\ast(t)^{q_1} \ \frac{dt}{t}}^{\frac{q}{q_1}}\ \brac{\intl_0^\infty t^{\theta q_2} g^\ast(t)^{q_2}  \ \frac{dt}{t}}^{\frac{q}{q_2}}.
\end{ma}
\]
As for (ii), this is the result in \cite[Theorem 2.6]{ONeil63}. As for (iii), this follows by the definition of $f^\ast$. Property (iv) was proven in Proposition~\ref{pr:bs:xyqintoxyqs}.\\
Lastly we consider Property (v). One checks that
\[
 \{x\in \R^n: \abs{x}^{-\lambda} > s\} = B_{s^{-\frac{1}{\lambda}}}(0),
\]
so
\[
 \brac{\abs{\cdot}^{-\lambda}}^\ast(t) = c_n\ t^{-\frac{\lambda}{n}},
\]
which readily implies
\[
 \Vert \abs{\cdot}^{-\lambda} \Vert_{L^{p,\infty}} = c_n \sup_{t > 0} t^{\frac{1}{p}}\ t^{-\frac{\lambda}{n}},
\]
which is finite if and only if $p = \frac{n}{\lambda}$. 
\end{proofP}

As the Lorentz spaces can be defined by interpolation, see Lemma~\ref{la:lpqinterpoldef}, by the Interpolation Theorem, Lemma~\ref{la:interpolation}, the following holds.
\begin{proposition}[Fourier Transform in Lorentz Spaces] \label{pr:fourierlpest}
For any $f \in \Sw$, $p \in (1,2)$, $q \in [1,\infty]$ we have
\[
\Vert f^\wedge \Vert_{L^{p',q}} \leq C_p \Vert f \Vert_{L^{p,q}}, \quad \Vert f^\vee \Vert_{L^{p',q}} \leq C_p \Vert f \Vert_{L^{p,q}}.
\]%
Here, $\frac{1}{p'} + \frac{1}{p} = 1$.
\end{proposition}

\begin{proposition}[Scaling in Lorentz Spaces]\label{pr:scalinglorentz}
Let $\lambda > 0$ and $f \in \Sw(\R^n)$. If we denote $\tilde{f}(\cdot) := f(\lambda \cdot)$, then
\[
\Vert \tilde{f} \Vert_{L^{p,q}} = \lambda^{-\frac{n}{p}} \Vert f \Vert_{L^{p,q}}.
\]
\end{proposition}
\begin{proofP}{\ref{pr:scalinglorentz}}
We have that $d_{\tilde{f}}(s) = \lambda^{-n} d_f(s)$ for any $s > 0$ and thus $\tilde{f}^\ast(t) = f^\ast (\lambda^n t)$ for any $t > 0$. Hence, 
\[
\intl_0^\infty \left (t^{\frac{1}{p}} \tilde{f}^\ast(t) \right )^q\ \frac{dt}{t} = \lambda^{-q\frac{n}{p}} \intl_0^\infty \left ((\lambda^n t)^{\frac{1}{p}} f^\ast(\lambda t) \right )^q\ \frac{dt}{t}
= \lambda^{-q\frac{n}{p}} \Vert f \Vert_{L^{p,q}}^q. 
\]
We can conclude.
\end{proofP}

\begin{proposition}[H\"older inequality in Lorentz Spaces]\label{pr:hoeldercompactsupp}
Let $\supp f \subset \bar{D}$, where $D \subset \R^n$ is a bounded measurable set. Then, whenever $\infty > p_1 > p \geq 1$, $q \in [1,\infty]$
\begin{equation}\label{eq:bs:estvp1q1}
\Vert f \Vert_{L^{p,q}} \leq C_{p,p_1,q}\ \abs{D}^{\frac{1}{p} - \frac{1}{p_1}}\ \Vert f \Vert_{L^{p_1}}.
\end{equation}
\end{proposition}
\begin{proofP}{\ref{pr:hoeldercompactsupp}}
Denote by $\chi \equiv \chi_D$ the characteristic function of the set $D \subset \R^n$. One checks that
\[
\chi^\ast(t) = \begin{cases}
	1\quad  &\mbox{if $t < \abs{D}$},\\
	0\quad &\mbox{if $t \geq \abs{D}$.}
	\end{cases}
\]
Consequently,
\[
\Vert \chi \Vert_{L^{p_2,q_2}} \aeq{} \abs{D}^{\frac{1}{p_2}} \quad \mbox{whenever $1 \leq p_2 < \infty$, $q_2 \in [1,\infty]$}.
\]
One concludes by applying H\"older's inequality in Lorentz spaces, Proposition~\ref{pr:dl:lso} (i), choosing $q_2 = q$ and $p_2$ such that
\[
 \frac{1}{p_2} + \frac{1}{p_1} = \frac{1}{p},
\]
using also the continuous embedding $L^{p_1} \subset L^{p_1,\infty}$,
\[
 \Vert f \Vert_{L^{p,q}} = \Vert f \chi \Vert_{L^{p,q}} \aleq{} \Vert f \Vert_{L^{p_1,\infty}}\ \Vert \chi \Vert_{L^{p_2,q}} \aleq{} \Vert f \Vert_{L^{p_1}}\ \abs{D}^{\frac{1}{p}-\frac{1}{p_1}}. \]

\end{proofP}

\subsection{Fractional Sobolev Spaces}\label{sec:fracsobspace}
In the following section we will give two equivalent definitions of the fractional Sobolev space $H^s \equiv H^s(\R^n)$, $s > 0$. The first definition is motivated by the interpretation of the Laplace operator as Fourier multiplier operator.
\begin{definition}[Fractional Sobolev Spaces by Fourier Transform]\label{def:fracSob}
Let $f \in L^2(\R^n)$. We say that for some $s \geq 0 $ the function $f \in H^s \equiv H^s(\R^n)$ if and only if $\lap^\frac{s}{2} f \in L^2(\R^n)$. Here, the operator $\lap^{\frac{s}{2}}$ is defined as
\[
\lap^{\frac{s}{2}} f := \left (\abs{\cdot}^s f^\wedge\right )^\vee.
\]
The norm, under which $H^s(\R^n)$ becomes a Hilbert space is
\[
\Vert f \Vert_{H^s(\R^n)}^2 := \Vert f \Vert_{L^2(\R^n)}^2 + \Vert \laps{s} f \Vert_{L^2(\R^n)}^2. 
\]
\end{definition}
\begin{remark}\label{rem:lapsClasLap}
Observe, that the definition of $\laps{2}$ coincides with the usual laplacian only up to a multiplicative constant, but this saves us from the nuisance to deal with those standard factors in every single calculation.   
\end{remark}
\begin{remark}
Observe that $\laps{s} f$ is a real function whenever $f \in \Sw(\R^n,\R)$. In fact, this is true for any multiplier operator $M$ defined for some multiplier $m \in C^\infty(\R^n \backslash \{0\}$ as
\[
 (Mf)^\wedge(\cdot) := m(\cdot)\ f^\wedge (\cdot),
\]
once we assume the additional condition
\begin{equation}\label{eq:remfracsob:mmxibarmxi}
 \overline{m(\xi)} = m(-\xi) \quad \mbox{for any $\xi \in \R^n \backslash \{0\}$},
\end{equation}
where by $\overline{\cdot}$ we denote the complex conjugate. This again can be seen as follows:
\[
\begin{ma}
 \overline{M f} &=& \overline{\brac{m(\cdot)\ f^\wedge}^\vee}\\
&=& \brac{\overline{m(\cdot)\ f^\wedge}}^\wedge\\
&\overset{\eqref{eq:remfracsob:mmxibarmxi}}{=}& \brac{m(-\cdot)\ f^\wedge(-\cdot)}^\wedge\\
&=& \brac{ \brac{Mf}^{\wedge}(-\cdot)}^\wedge\\
&=& \brac{ \brac{Mf}^{\vee}(\cdot)}^\wedge \qquad = Mf.
\end{ma}
\]
\end{remark}
\begin{remark}
In Section \ref{ss:idlaps} we will prove an integral representation for the fractional laplacian $\laps{s}$. 
\end{remark}
On the other hand, fractional Sobolev spaces can be defined by interpolation:
\begin{lemma}[Fractional Sobolev Spaces by Interpolation]\label{la:fracsobdef2}\nl
(See \cite[Chapter 23]{Tartar07})\\
Let $s \in (0,\infty)$. Then
\[
 H^s(\R^n) = [W^{i,2}(\R^n),W^{j,2}(\R^n)]_{\theta,2},
\]
with equivalent norms, whenever $\theta = \frac{s-i}{j-i} \in(0,1)$ for $i < s < j$, $i,j \in \N_0$. 
\end{lemma}

\begin{lemma}[Compactly Supported Smooth Functions are Dense]\label{la:Tartar07:Lemma15.10}
(see \cite[Lemma 15.10.]{Tartar07})\\
The space $C_0^\infty(\R^n) \subset H^s(\R^n)$ is dense for any $s \geq 0$, t.
\end{lemma}

Our next goal is Poincar\'{e}'s inequality. As we want to use the standard blow up argument to prove it, we premise a (trivial) uniqueness and a compactness result:
\begin{lemma}[Uniqueness of solutions]\label{la:bs:uniqueness}
Let $f \in H^s(\R^n)$, $s > 0$. If $\lap^{\frac{s}{2}} f \equiv 0$, then $f \equiv 0$.
\end{lemma}
\begin{proofL}{\ref{la:bs:uniqueness}}
As $f \in H^s(\R^n)$, $f^\wedge$ exists and $f^\wedge(\xi) = \abs{\xi}^{-s} 0 = 0$ for almost every $\xi \in \R^n$. Thus, $f^\wedge \equiv 0$ as $L^2$-function and we conclude that also $f \equiv 0$.
\end{proofL}

\begin{lemma}[Compactness]\label{la:bs:compact}
Let $D \subset \R^n$ be a smoothly bounded domain, $s > 0$. Assume that there is a constant $C > 0$ and $f_k \in H^s(\R^n)$, $k \in \N$, such that for any $k \in \N$ the conditions $\supp f_k \subset \bar{D}$ and $\Vert f_k \Vert_{H^s} \leq C$ hold. Then there exists a subsequence $f_{k_i}$, such that $f_{k_i} \xrightarrow{i \to \infty} f \in H^s$ weakly in $H^s$, strongly in $L^2(\R^n)$, and pointwise almost everywhere. Moreover, $\supp f \subset \bar{D}$.
\end{lemma}
\begin{proofL}{\ref{la:bs:compact}}
Fix $D \subset \R^n$ and let $\eta \in C_0^\infty(2D)$, $\eta \equiv 1$ on $D$. Define the operator
\[
 S:\ v \in L^2(\R^n) \mapsto \eta v.
\]
As $D$ is a bounded subset, $S$ is compact as an operator $W^{j,2}(\R^n) \to L^2(\R^n)$ for any $j \in \N$ and continuous as an operator $L^2(\R^n) \to L^2(\R^n)$. Consequently, Lemma~\ref{la:interpol:compact} for $G = X = L^2$, $Y = W^{j,2}$ and Lemma~\ref{la:fracsobdef2} imply that $S$ is also a compact operator $H^s(\R^n) \to L^2(\R^n)$ for all $s \in (0,j)$. As $S$ is the identity on all functions $f \in L^2(\R^n)$ such that $\supp f \subset \bar{D}$, we conclude the proof of the claim of convergence in $L^2$ and pointwise almost everywhere, which implies also the support condition. Lastly, the weak convergence result stems from the fact that $H^s$ is a Hilbert space.
\end{proofL}%
\begin{remark}
As for weak convergence, one can prove that $f_k \to f$ weakly in $H^s(\R^n)$ for some $s > 0$ implies that $\laps{s} f_k \to \laps{s}f$ weakly in $L^2$. In fact assume that $f_k \to f$ weakly in $H^s(\R^n)$ and in particular $\Vert f_k \Vert_{H^s} \leq C$. For any $\varphi \in C_0^\infty(\R^n)$,
\[
 \intl_{\R^n} \laps{s}f_k\ \varphi = \intl_{\R^n} f_k\ \laps{s}\varphi \xrightarrow{k \to \infty} \intl_{\R^n} f\ \laps{s}\varphi \intl_{\R^n} \laps{s}f\ \varphi.
\]
Next, for any $w \in L^2(\R^n)$ and $w_\varepsilon \in C_0^\infty(\R^n)$ such that $\Vert w - w_\varepsilon\Vert_{L^2} \leq \varepsilon$,
\[
 \big \vert \intl_{\R^n} \laps{s}f_k\ w -\intl_{\R^n} \laps{s}f \ w \big \vert \leq  \big \vert \intl_{\R^n} \laps{s}f_k\ w_\varepsilon -\intl_{\R^n} \laps{s}f \ w_\varepsilon \big \vert + C \varepsilon. 
\]
Thus, letting $\varepsilon$ go to zero and $k$ to infininity, we can prove weaky convergence of $\laps{s} f_k$ in $L^2(\R^n)$.
\end{remark}

With the compactness lemma, Lemma~\ref{la:bs:compact}, we can prove Poincar\'{e}'s inequality. As in \cite[Theorem A.2]{DR09Sphere} we will use a support-condition in order to ensure compactness of the embedding $H^s(\R^n)$ into $L^2(\R^n)$ (see Lemma~\ref{la:bs:compact}). This support condition can be seen as saying that all derivatives up to order $\lfloor \frac{s}{2} \rfloor$ are zero at the boundary, therefore it is not surprising that such an inequality should hold.
\begin{lemma}[Poincar\'{e} Inequality]\label{la:poinc}
For any smoothly bounded domain $D \subset \R^n$, $s > 0$, there exists a constant $C_{D,s} > 0$ such that
\begin{equation}\label{eq:poinc}
\Vert f \Vert_{L^2(\R^n)} \leq C_{D,s}\ \Vert \laps{s} f \Vert_{L^2(\R^n)}, \quad \mbox{for all $f \in H^s(\R^n)$, $\supp f \subset \bar{D}$}. 
\end{equation}
If $D = r\tilde{D}$ for some $r > 0$, then $C_{D,s} = C_{\tilde{D},s}r^{s}$.
\end{lemma}
\begin{remark}
One checks as well, that $C_{D,s} = C_{\tilde{D},s}$ if $D$ is a mere translation of some smoothly bounded domain $\tilde{D}$. This is clear, as the operator $\laps{s}$ commutes with translations.
\end{remark}
\begin{proofL}{\ref{la:poinc}}
We proceed as in the standard blow-up proof of Poincar\'{e}'s inequality: Assume \eqref{eq:poinc} is false and that there are functions $f_k \in H^s(\R^n)$, $\supp f_k \subset \bar{D}$, such that
\begin{equation}\label{eq:poinc:small}
\Vert f_k \Vert_{L^2(\R^n)} > k \Vert \laps{s} f_k \Vert_{L^2(\R^n)}, \quad \mbox{for every $k \in \N$}.
\end{equation}
Dividing by $\Vert f_k \Vert_{L^2(\R^n)}$ we can assume w.l.o.g. that $\Vert f_k \Vert_{L^2(\R^n)} = 1$ for every $k \in \N$. Consequently, we have for every $k \in \N$
\[
\Vert f_k \Vert_{H^s(\R^n)} \aleq{} \Vert f_k \Vert_{L^2(\R^n)} + \Vert \laps{s} f_k \Vert_{L^2(\R^n)} \aleq{} 1.
\]
Modulo passing to a subsequence of $(f_k)_{k \in \N}$, we can assume by Lemma~\ref{la:bs:compact} that $f_k$ converges weakly to some $f \in H^s(\R^n)$ with $\supp f \subset \bar{D}$ and strongly in $L^2(\R^n)$. This implies, that $\Vert f \Vert_{L^2(\R^n)} = 1$ and 
\[
\Vert \laps{s} f \Vert_{L^2(\R^n)} \leq \liminf_{k \to \infty}\ \Vert \laps{s} f_k \Vert_{L^2(\R^n)} \overset{\eqref{eq:poinc:small}}{=} 0.
\] 
But this is a contradiction, as Lemma~\ref{la:bs:uniqueness} implies that $f \equiv 0$.\\
If $D = r\tilde{D}$ for some $r > 0$, we define as usual a scaled function $\tilde{f}(x) := f(rx)$ and use that
\[
 \brac{\laps{s} \tilde{f}}(x) = r^s\ \brac{\laps{s} v}(rx)
\]
in order to conclude.
\end{proofL}

A simple consequence of the ``standard Poincar\'{e} inequality'' is the following
\begin{lemma}[Slightly more general Poincar\'{e} inequality]\label{la:poincExt}
For any smoothly bounded domain $D \subset \R^n$, $0 < s \leq t$, there exists a constant $C_{D,t} > 0$ such that
\[
\Vert \laps{s} f \Vert_{L^2(\R^n)} \leq C_{D,t}\ \Vert \laps{t} f \Vert_{L^2(\R^n)}, \quad \mbox{for all $f \in H^t(\R^n)$, $\supp f \subset \bar{D}$}. 
\]
If $D = r\tilde{D}$ for some $r > 0$, then $C_{D,t} = C_{\tilde{D},t}r^{t-s}$.
\end{lemma}
\begin{proofL}{\ref{la:poincExt}}
We have
\[
\begin{ma}
 \Vert \laps{s} f \Vert_{L^2} &=& \Vert \abs{\cdot}^s\ f^\wedge \Vert_{L^2}\\
&\leq& \Vert \abs{\cdot}^t\ f^\wedge \Vert_{L^2(\R^n \backslash B_1(0))} + \Vert f^\wedge \Vert_{L^2(B_1(0))}\\
&\leq& \Vert \laps{t}f \Vert_{L^2} + \Vert f \Vert_{L^2}\\
&\overset{\sref{L}{la:poinc}}{\leq}& C_{D,t}\ \Vert \laps{t} f \Vert_{L^2}.
\end{ma}
\]
By scaling one concludes.
\end{proofL}
The following lemma can be interpreted as an existence result for the equation $\laps{s} w = v$ - or as a variant of Poincar\'{e}'s inequality:
\begin{lemma}\label{la:lapmsest2}
Let $s \in (0,n)$, $p \in [2,\infty)$ such that
\begin{equation}\label{eq:lapmsest:pcond}
\frac{n-s}{n} > \frac{1}{p} \geq \frac{n-2s}{2n}.
\end{equation}
Then for any smoothly bounded set $D \subset \R^n$ there is a constant $C_{D,s,p}$ such that for any $v \in \Sw(\R^n)$, $\supp v \subset \bar{D}$, we have $\lapms{s} v \in L^p(\R^n)$ and
\[
\Vert \lapms{s} v \Vert_{L^p(\R^n)} \leq C_{D,p,s}\ \Vert v \Vert_{L^2}.
\]
Here, $\lapms{s} v$ is defined as $(\abs{\cdot}^{-s}  v^\wedge )^\vee$. In particular, if $s \in (0,\frac{n}{2})$,
\[
\Vert \lapms{s} v \Vert_{L^2(\R^n)} \leq C_{D,s}\ \Vert v \Vert_{L^2}.
\]
If $D = r\tilde{D}$, then $C_{D,p,s} = r^{s+\frac{n}{p}-\frac{n}{2}}\ C_{\tilde{D},p,s}$.
\end{lemma}
\begin{proofL}{\ref{la:lapmsest2}}
We want to make the following reasoning rigorous:
\[
\begin{ma}
\Vert \lapms{s} v \Vert_{L^p} &\overset{\ontop{\sref{P}{pr:fourierlpest}}{p \in [2,\infty)}}{\leq}& C_p\ \Vert (\lapms{s} v)^\wedge \Vert_{L^{p',p}}\\
&=& C_p\ \Vert \abs{\cdot}^{-s}\ v^\wedge \Vert_{L^{p',p}}\\
&\overset{(\star)}{\leq}& C_p\ \Vert \abs{\cdot}^{-s} \Vert_{L^{\frac{n}{s},\infty}}\ \Vert v^\wedge \Vert_{L^{q,p}}\\
&\overset{p \geq 2}{\leq}& C_p\ \Vert \abs{\cdot}^{-s} \Vert_{L^{\frac{n}{s},\infty}}\ \Vert v^\wedge \Vert_{L^{q,2}}\\
&\overset{\ontop{\sref{P}{pr:fourierlpest}}{q \geq 2}}{\leq}& C_{p,s,q}\ \Vert v \Vert_{L^{q',2}}\\
&\overset{\ontop{\sref{P}{pr:hoeldercompactsupp}}{q' \leq 2}}{\leq}& C_{s,q}\ C_D\ \Vert v \Vert_{L^2}.
\end{ma}
\]
To do so, we need to find $q \in [2,\infty)$ such that $(\star)$ holds:
\[
\frac{1}{p'} = \frac{1}{q} + \frac{s}{n}
\]
which is possible by virtue of \eqref{eq:lapmsest:pcond}. Then the validity of $(\star)$ follows from Proposition~\ref{pr:dl:lso} and we conclude scaling as in Proposition~\ref{pr:scalinglorentz}.
\end{proofL}
The next lemma can be seen as an adaption of Hodge decomposition to the setting of the fractional laplacian:
\begin{lemma}[Hodge Decomposition]\label{la:hodge}
Let $f \in L^2(\R^n)$, $s > 0$. Then for any smoothly bounded domain $D \subset \R^n$ there are functions $\varphi \in H^{s}(\R^n)$, $h \in L^2(\R^n)$ such that
\[
 \supp \varphi \subset \bar{D},
\]
\[
 \intl_{\R^n} h\ \laps{s} \psi = 0, \quad \mbox{for all } \psi \in C_0^\infty(D),
\]
and
\[
 f = \laps{s} \varphi + h\quad \mbox{almost everywhere in $\R^n$}.
\]
Moreover, 
\begin{equation}\label{eq:hodge:est}
\Vert h \Vert_{L^2(\R^n)} + \Vert \laps{s} \varphi \Vert_{L^2(\R^n)} \leq 5 \Vert f \Vert_{L^2(\R^n)}.
\end{equation}
\end{lemma}
\begin{proofL}{\ref{la:hodge}}
Set
\[
 E(v) := \intl_{\R^n} \abs{ \laps{s}  v - f}^2,\quad \mbox{for $v \in H^s(\R^n)$ with $\supp v \subset \bar{D}$}.
\]
Then,
\begin{equation}\label{eq:hodge:coerc}
 \Vert \laps{s}  v \Vert_{L^2(\R^n)}^2 \leq 2 E(v) + 2\Vert f \Vert_{L^2(\R^n)}^2.
\end{equation}
As $D$ is smoothly bounded, Poincar\'{e}'s inequality, Lemma~\ref{la:poinc}, implies for any $v \in H^s(\R^n)$ with $\supp v \subset \bar{D}$
\[
 \Vert v \Vert_{H^s}^2 \leq C_{s,D} (E(v) + \Vert f \Vert_{L^2(\R^n)}^2).
\]
Thus $E(\cdot)$ is coercive, i.e. for an $E(\cdot)$-minimizing sequence $(\varphi_k)_{k=1}^\infty \subset H^s(\R^n)$ with $\supp \varphi_k \subset \bar{D}$  we can assume
\[
 \Vert \varphi_k \Vert_{H^s}^2 \leq C (E(0) + \Vert f \Vert_{L^2(\R^n)}^2) = 2C \Vert f \Vert_{L^2(\R^n)}^2,\quad \mbox{for every $k \in \N$}.
\]
By compactness, see Lemma~\ref{la:bs:compact}, up to taking a subsequence of $k \to \infty$, we have weak convergence of $\varphi_k$ to some $\varphi$ in $H^s(\R^n)$ and strong convergence in $L^2$, as well as $\supp \varphi \subset \bar{D}$.\\
$E(\cdot)$ is lower semi-continuous with respect to weak convergence in $H^s(\R^n)$, so $\varphi$ is a minimizer of $E(\cdot)$.\\
If we call $h := \laps{s} \varphi - f$, Euler-Lagrange-Equations give that
\[
 \intl_{\R^n} h\ \laps{s}  \psi = 0, \quad \mbox{for any $\psi \in C_0^\infty(D)$}.
\]
Estimate \eqref{eq:hodge:coerc} for $\varphi$ and the fact that $\Vert h \Vert_{L^2}^2 = E(\varphi) \leq E(0)$ imply \eqref{eq:hodge:est}. 
\end{proofL}
\begin{remark}
In fact, $h$ will satisfy enhanced local estimates, similar to estimates for harmonic function, see Lemma~\ref{la:estharmonic}.
\end{remark}

\subsection{Annuli-Cutoff Functions}\label{ss:cutoff}

We will have to localize our equations, so we introduce as in \cite{DR09Sphere} a decomposition of unity as follows: Let $\eta \equiv \eta^0 \in C_0^\infty(B_2(0))$, $\eta \equiv 1$ in $B_1(0)$ and $0 \leq \eta \leq 1$ in $\R^n$. Let furthermore $\eta^k \in C_0^\infty(B_{2^{k+1}}(0)\backslash B_{2^{k-1}}(0))$, $k \in \N$, such that $0 \leq \eta^k \leq 1$, $\sum_{k=0}^\infty \eta^k = 1$ pointwise in $\R^n$ and $\abs{\nabla^i \eta^k} \leq C_i 2^{-ki}$ for any $i \in \N_0$.\\
We call $\eta^k_{r,x} := \eta^k (\frac{\cdot - x}{r})$, though we will often omit the subscript when $x$ and $r$ should be clear from the context.\\
For the sake of completeness we sketch the construction of those $\eta^k$:
\begin{proof}[Construction of suitable cutoff functions]
Firstly, pick $\eta \equiv \eta^0 \in C_0^\infty(B_2(0))$, $\eta \equiv 1$ on, say, $B_{\frac{3}{2}}(0)$ and $\eta(x) \in [0,1]$ for any $x \in \R^n$. We set for $k \in \N$,
\begin{equation}\label{eq:defcutoff}
\eta^k(\cdot) := \left (1-\suml_{l=0}^{k-1} \eta^l(\cdot) \right ) \suml_{l=0}^{k-1} \eta^l \left (\frac{\cdot}{2} \right ).
\end{equation}
Obviously, $\eta^k$ is smooth and we have the following crucial properties
\begin{itemize}
\item[(i)] $\eta^k \in C_0^\infty(B_{2^{k+1}}(0) \backslash \overline{B_{2^{k-1}}(0)})$, if $k \geq 1$, and 
\item[(ii)] $\sum_{l=0}^k \eta^l \equiv 1$ in $B_{2^k}(0)$, for every $k \geq 0$.
\end{itemize}
Indeed, this can be shown by induction: First, one checks that (i), (ii) are true for $k = 0,1$. Then, assume that (i) and (ii) hold for some positive integer $k -1$. By (ii) we have that $1-\sum_{l=0}^{k-1} \eta_l \equiv 0$ in $B_{2^{k-1}}(0)$ and (i) implies that $\sum_{l=0}^{k-1} \eta_l \left (\frac{\cdot}{2} \right ) \equiv 0$ in $\R^n \backslash B_{2^{k-1+1}2}$. This implies (i) for $k$. Moreover,
\[
\suml_{l=0}^k \eta^l = \suml_{l=0}^{k-1} \eta^l + \left (1-\suml_{l=0}^{k-1} \eta^l \right )(\cdot) \suml_{l=0}^{k-1} \eta^l \left (\frac{\cdot}{2} \right ).
\]
By (ii) for $k-1$ on $B_{2^{k-1}2} = B_{2^{k}}$ the sum $\sum_{l=0}^{k-1} \eta^l \left (\frac{\cdot}{2} \right )$ is identically $1$ so (ii) holds for $k$ as well. Consequently, by induction (i) and (ii) hold for all $k \in \N_0$. It is easy to check that also $0 \leq \eta_k \leq 1$.\\
Moreover, one checks that $\abs{\nabla^i \eta^k} \leq C_i 2^{-ki}$ for every $i \in \N_0$: In fact, if we abbreviate $\psi^k := \sum_{l=0}^{k} \eta^k$, we have of course \[
\abs{\nabla^i \eta^k} \leq \abs{\nabla^i \psi^k} + \abs{\nabla^i \psi^{k-1}}.
\]
It is enough, to show that $\abs{\nabla^i \psi^k} \leq C_i 2^{-ki}$: We have
\[
\psi^{k} = \psi^{k-1} + (1-\psi^{k-1})(\cdot)\ \psi^{k-1} \brac{\frac{1}{2} \cdot}.
\]
By property (ii) we know that $\psi^{k} \equiv 1$ in $B_{2^{k}}$ and $\psi^{k} \equiv 0$ in $\R^n \backslash B_{2^{k+1}}$, so the gradient in those sets is trivial. On the other hand, in $B_{2^{k+1}} \backslash B_{2^{k}}$ we know that $\psi^{k-1} \equiv 0$, by property (i), hence $\psi^k = \psi^{k-1} (\frac{1}{2} \cdot)$ in this set. This implies
\[
\nabla^i \psi^k = 2^{-i} (\nabla^i \psi^{k-1}) \left (\frac{1}{2} \cdot \right ).
\]
By induction one arrives then at $\abs{\nabla^i \psi^k} \leq 2^{-ki} \Vert \nabla^i \eta^0 \Vert_{L^\infty}$.
\end{proof}
%
We want to estimate some $L^p$-Norms of $\laps{s} \eta^k_{r,x}$. In order to do so, we will need the following Proposition:
\begin{proposition}\label{pr:weirdgwedgeest}
(Cf. \cite[Exercise 2.2.14, p.108]{GrafC08})\\
For every $g \in \Sw(\R^n)$, $p \in [1,2]$, $s \geq 0$, $-\infty < \alpha < n \frac{p-2}{p} < \beta < \infty$, we have
\[
\Vert \brac{\laps{s}g}^\wedge \Vert_{L^p(\R^n)} \leq C_{\alpha,\beta,p}\ \left (\Vert \laps{s+\alpha} g \Vert_{L^2(\R^n)} + \Vert \laps{s+\beta} g \Vert_{L^2(\R^n)} \right ). 
\]
\end{proposition}
\begin{proofP}{\ref{pr:weirdgwedgeest}}
Set $q := \frac{2p}{2-p}$. We abbreviate $f := \brac{\laps{s} g}^\wedge$ and set $f = f_1 + f_2$, where $f_1 = f \chi_{B_1(0)}$. Here, $\chi_{B_1(0)}$ denotes as usual the characteristic function of $B_1(0)$. Then $f_1(x) = \abs{x}^\alpha f_1(x)\ \abs{x}^{-\alpha}$ and hence
\[
\begin{ma}
\Vert f_1(x) \Vert_{L^p(\R^n)} &\leq& \Vert \abs{\cdot}^\alpha\ f_1 \Vert_{L^2(B_1(0))}\ \Vert \abs{\cdot}^{-\alpha} \Vert_{L^{q}(B_1(0))}\\
 &\overset{q\alpha < n}{\leq}& C_\alpha \Vert \abs{\cdot}^\alpha f \Vert_{L^2(B_1(0))}.  
\end{ma}
\]
The same works for $f_2$, using that $q\beta > n$. Consequently, one arrives at
\[
\Vert f \Vert_{L^p(\R^n)} \leq C_{\alpha,\beta,p} (\Vert \abs{\cdot}^\alpha f \Vert_{L^2(\R^n)} + \Vert \abs{\cdot}^\beta f \Vert_{L^2(\R^n)}). 
\]
Replacing again $f = \brac{\laps{s} g}^\wedge$ and using that $\abs{\cdot}^\alpha \brac{\laps{s}g}^\wedge = (\laps{\alpha+s} g)^\wedge$, $\abs{\cdot}^\beta \brac{\laps{s}g}^\wedge = (\laps{\beta+s} g)^\wedge$ and then applying Plancherel Theorem for $L^2$-functions, one concludes.
\end{proofP}

\begin{proposition}\label{pr:etarkgoodest}
For any $s > 0$, $p \in [1,2]$, there is a constant $C_{s,p} > 0$, such that for any $k \in \N_0$, $x \in \R^n$, $r > 0$ denoting as usual $p' := \frac{p}{p-1}$,
\begin{equation}\label{eq:etarkgoodest:fourier}
\Vert \left (\laps{s} \eta_{r,x}^k \right )^\wedge \Vert_{L^{p}(\R^n)} \leq C_{s,p}\ (2^k r)^{-s + \frac{n}{p'}}. 
\end{equation}
In particular,
\begin{equation}\label{eq:etarkgoodest:linfty}
\Vert \laps{s} \eta_{r,x}^k \Vert_{L^{p'}(\R^n)} \leq C_{s,p}\ (2^k r)^{-s + \frac{n}{p'}}. 
\end{equation}
\end{proposition}
\begin{proofP}{\ref{pr:etarkgoodest}}
Fix $r > 0$, $k \in \N$ and $x \in \R^n$. Set $\tilde{\eta}(\cdot) := \eta^k_{r,x} (x+2^kr\cdot)$. By scaling it then suffices to show that for a uniform constant $C_{s,p} > 0$
\begin{equation}\label{eq:etarkgoodest:prescaled}
\Vert \left (\laps{s} \tilde{\eta} \right )^\wedge \Vert_{L^p(\R^n)} \leq C_{s,p}.
\end{equation}
First of all, for any $i \in \N$ there is a constant $C_i > 0$ independent of $r$, $x$, $k$ such that
\[
\Vert \tilde{\eta} \Vert_{W^{i,2}} \leq C_{i}.
\]
In fact, by the choice of the scaling for $\tilde{\eta}$, we have that $\supp \tilde{\eta} \subset B_2(0)$, $\abs{\nabla^j \tilde{\eta}} \leq C_i$ for any $1 \leq j \leq i$.\\
Consequently, as for any $\alpha,\beta \geq 0$ the spaces $H^{s+\alpha}$ and $H^{s+\beta}$ are by Lemma~\ref{la:fracsobdef2} (equivalent to) the interpolation spaces $[L^2(\R^n),W^{i,2}(\R^n)]_{\theta,2}$, for some $i = i_{\alpha,\beta} \in \N$ and $\theta \in (0,1)$, we have for any $\alpha,\beta,s \geq 0$
\begin{equation}\label{eq:etarkgoodest:etawi2est}
\Vert \tilde{\eta} \Vert_{H^{s+\alpha}} + \Vert \tilde{\eta} \Vert_{H^{s+\beta}} \leq C_{\alpha,\beta,s} \Vert \tilde{\eta}. \Vert_{W^{i,2}(\R^n)}
\end{equation}
But by Proposition~\ref{pr:weirdgwedgeest} for some admissible $\alpha, \beta \geq 0$ (depending on $p$; in the case $p = 2$ we can choose $\alpha = \beta = 0$),
\[
\begin{ma}
\Vert \left (\laps{s} \tilde{\eta} \right )^\wedge \Vert_{L^p(\R^n)} &\leq& C_{\alpha,\beta,p} (\Vert \laps{s+\alpha} \tilde{\eta} \Vert_{L^2} + \Vert \laps{s+\beta} \tilde{\eta} \Vert_{L^2})\\
&\leq& C_{\alpha,\beta,p}\ (\Vert \tilde{\eta} \Vert_{H^{s+\alpha}}\ + \Vert \tilde{\eta} \Vert_{H^{s+\beta}})\\
&\leq& C_{\alpha,\beta,p,s}.
\end{ma}
\]
Consequently, we have shown \eqref{eq:etarkgoodest:prescaled}, and by scaling back we conclude the proof of \eqref{eq:etarkgoodest:fourier}. Equation \eqref{eq:etarkgoodest:linfty} then follows by the continuity of the inverse Fourier-transform from $L^p$ to $L^{p'}$ whenever $p \in [1,2]$, see Proposition~\ref{pr:fourierlpest}.
\end{proofP}
One important consequence is, that in a weak sense $\laps{s} P$ vanishes for a polynomial $P$, if $s$ is greater than the degree of $P$:
\begin{proposition}\label{pr:lapspol}
Let $\alpha$ be a multiindex $\alpha = (\alpha_1,\ldots,\alpha_n)$, where $\alpha_i \in \N_0$, $1 \leq i \leq n$. If $s > 0$ such that $\abs{\alpha} = \suml_{i=1}^n \abs{\alpha_i} < s$ then
\[
 \lim_{R \to \infty} \intl_{\R^n} \eta_R x^\alpha\ \laps{s} \varphi = 0, \quad \mbox{for every $\varphi \in \Sw(\R^n)$}.
\]
Here, $x^\alpha := (x_1)^{\alpha_1}\cdots (x_n)^{\alpha_n}$.
\end{proposition}
\begin{proofP}{\ref{pr:lapspol}}
One checks that for some constant $c_\alpha$,
\begin{equation}\label{eq:lapspol:xalphapsi}
 x^\alpha \psi = c_\alpha \brac{\partial^\alpha \psi^\vee}^\wedge \quad \mbox{for all $\psi \in \Sw(\R^n)$}.
\end{equation}
This and the fact that for any $\psi \in \Sw(\R^n)$ we have also $\psi^\wedge \in \Sw(\R^n)$ and $x^\alpha \psi \in \Sw(\R^n)$ implies (using as well integration by parts)
\[
\begin{ma}
 &&\intl_{\R^n} \psi\ x^\alpha \laps{s} \varphi\\
&\overset{\eqref{eq:lapspol:xalphapsi}}{=}& c_\alpha \intl_{\R^n} \brac{\partial^\alpha \psi^\vee}^\wedge\ \laps{s} \varphi\\
&=& c_\alpha \intl_{\R^n} \abs{\cdot}^s \varphi^\wedge\ \partial^\alpha \psi^\vee\\
&=& \sum_{\abs{\beta} \leq \abs{\alpha}} c_{\alpha,\beta} \intl_{\R^n} m_{\alpha,\beta,s}(\cdot)\ \abs{\cdot}^{s-\abs{\alpha}+\abs{\beta}}\ \brac{\partial^\beta \varphi^\wedge}\ \psi^\wedge(-\cdot),
\end{ma}
\]
where $m_{\alpha,\beta,s} \in C^\infty(\R^n \backslash \{0\})$ is some zero multiplier. Denoting by $M_{\alpha,\beta,s}$ the respective Fourier multiplier operator with multiplier $m_{\alpha,\beta,s}$ we arrive at
\[
 \intl_{\R^n} \psi\ x^\alpha\ \laps{s} \varphi = \sum_{\abs{\beta} \leq \abs{\alpha}} c_{\alpha,\beta} \intl_{\R^n} \brac{x^\beta \varphi}\ M_{\alpha,\beta,s}\ \laps{s-\abs{\alpha}+\abs{\beta}} \psi.
\]
In particular, this is true for $\psi := \eta_R$, and we have for any $p \in (1,2)$, $R > 1$,
\[
\begin{ma}
 &&\intl_{\R^n} \eta_R\ x^\alpha \laps{s} \varphi\\
&\aleq{}& \sup_{\abs{\beta} \leq \abs{\alpha}} \Vert x^\beta \varphi \Vert_{L^p(\R^n)}\ \Vert M_{\alpha,\beta,s} \laps{s-\abs{\alpha}+\abs{\beta}} \eta_R \Vert_{L^{p'}(\R^n)}\\
&\aleq{}& C_{\alpha,\varphi,p,s}\  \sup_{\abs{\beta} \leq \abs{\alpha}} \Vert \laps{s-\abs{\alpha}+\abs{\beta}} \eta_R \Vert_{L^{p'}(\R^n)}\\
&\overset{\sref{P}{pr:etarkgoodest}}{\aleq{}}& R^{-s+\abs{\alpha}+\frac{n}{p'}}.
\end{ma}
\]
Here we used as well that multiplier operators such as $M_{\alpha,\beta,s}$ map $L^{p'}$ into $L^{p'}$ continuously for $p' \in (1,\infty)$ by H\"ormander's theorem \cite{Hoermander60}. As $\abs{\alpha} < s$, we can choose $p' \in (2,\infty)$ such that $-s+\abs{\alpha}+\frac{n}{p'} < 0$, and taking the limit $R \to \infty$ we conclude.
\end{proofP}
\begin{remark} \label{rem:cutoffPolbdd}
One can even show, that
\[
 \Vert \laps{s} \brac{\eta_{r,0} x^\alpha} \Vert_{L^p(\R^n)} \leq C_{s,p}\ r^{-s + \abs{\alpha} + \frac{n}{p}}  \quad \mbox{for any $p \in [2,\infty]$, $\abs{\alpha} < s$, $r > 0$}.
\]
This is done similar to the proof of Proposition~\ref{pr:etarkgoodest}: First one proves the claim for $r = 1$, then scaling implies the claim, using that
\[
 \eta_{r,0}(x) x^\alpha = r^\abs{\alpha} \eta_{1,0}(r^{-1} x) (r^{-1} x)^\alpha.
\]
\end{remark}

\begin{remark}
We will use Proposition~\ref{pr:lapspol} in a formal way, by saying that formally $\laps{s} x^\alpha = 0$ whenever $\abs{\alpha} < s$. Of course, as we defined the operator $\laps{s}$ on $L^2$-Functions only, this formal argument should be verified in each calculation by using that 
\[
\lim_{R\to \infty} \laps{s} \brac{\eta_R x^\alpha} = 0,
\]
where the limit will be taken in an appropriate sense. For the sake of simplicity, we will omit this recurring argument.
\end{remark}

\subsection{An Integral Definition for the Fractional Laplacian}\label{ss:idlaps}
A further definition of the fractional laplacian for small order without the use of the Fourier transform are based on the following two propositions.
\begin{proposition}\label{pr:eqdeflaps1}
Let $s \in (0,1)$. For some constant $c_n$ and any $v \in \Sw(\R^n)$,
\[
 \laps{s} v (\bar{y})= c_n \intl_{\R^n} \frac{v(x)-v(\bar{y})}{\abs{x-\bar{y}}^{n+s}}\ dx  \quad \mbox{for any $\bar{y} \in \R^n$}.
\]
\end{proposition}
\begin{proofP}{\ref{pr:eqdeflaps1}}
It is enough to prove the claim for $\bar{y} = 0$. In fact, denote by $\tau_{\bar{y}}$ the translation operator
\[
\tau_{\bar{y}} v (\cdot) := v (\cdot+\bar{y}).
\]
Then, as any multiplier operator commutes with translations, assuming the claim to be true for $\bar{y} = 0$ ,
\[
\begin{ma}
\laps{s}v (\bar{y}) &=& \laps{s} \brac{\tau_{\bar{y}}v} (0)\\
&=& c_n \intl_{\R^n} \frac{\tau_{\bar{y}}v(x)-\tau_{\bar{y}}v(0)}{\abs{x}^{n+s}}\ dx\\
&=& c_n \intl_{\R^n} \frac{v(x+\bar{y})-v(\bar{y})}{\abs{x}^{n+s}}\ dx\\
&=& c_n \intl_{\R^n} \frac{v(x)-v(\bar{y})}{\abs{x-\bar{y}}^{n+s}}\ dx,
\end{ma}
\]
where the transformation formula is valid because the integral converges absolutely as $s \in (0,1)$.\\
So let $\bar{y} = 0$, $v \in \Sw(\R^n)$. For any $R > 1 > \varepsilon > 0$ we set
\[
 \eta_R := \eta_{R,0}^0,\quad \mbox{and}\quad \eta_{4\varepsilon} = \eta_{4\varepsilon,0}^0,
\]
and decompose $v = v_1 + v_2 + v_3 + v_4$ as follows:
\[
 \begin{ma}
  v &=& \eta_{4\varepsilon} \brac{v-v(0)} + (1-\eta_{4\varepsilon}) \brac{v-v(0)} + v(0)\\
&=:& v_1 + \eta_{R} (1-\eta_{4\varepsilon}) \brac{v-v(0)} + \eta_{R} v(0)\\
&&\quad + \brac{1-\eta_{R}} \ebrac{(1-\eta_{4\varepsilon}) \brac{v-v(0)} + v(0)}\\
&=:& v_1 + v_2 + v_3 + v_4,
 \end{ma}
\]
that is
\[
 \begin{ma}
  	v_1 &=& \eta_{4\varepsilon} \brac{v-v(0)},\\
 	v_2 &=& \eta_{R} (1-\eta_{4\varepsilon}) \brac{v-v(0)},\\
	v_3 &=& \eta_{R} v(0),\\
	v_4 &=& \brac{1-\eta_{R}} \ebrac{(1-\eta_{4\varepsilon}) \brac{v-v(0)} + v(0)}\\
&=& \brac{1-\eta_{R}} \ebrac{(1-\eta_{4\varepsilon}) v + \eta_{4\varepsilon} v(0)}.
 \end{ma}
\]
Observe that $v_k \in \Sw(\R^n)$, $k=1\ldots 4$, and in particular $\laps{s} v_k$ is well defined in the sense of Definition \ref{def:fracSob}. So for any $\varphi \in C_0^\infty(B_{2\varepsilon}(0))$
\[
 \intl_{\R^n} \laps{s} v\ \varphi = I_1 + I_2 + I_3 + I_4,
\]
where
\[
 I_k := \intl_{\R^n} \laps{s} v_k\ \varphi,\quad k=1,2,3,4.
\]
First, observe that by the Lebesgue-convergence theorem,
\begin{equation}\label{eq:eqdeflaps:I4}
 \lim_{R \to \infty} I_4 = \lim_{R \to \infty} \intl_{\R^n} \brac{1-\eta_{R}} \ebrac{(1-\eta_{4\varepsilon}) v + \eta_{4\varepsilon} v(0)} \laps{s} \varphi  = 0.
\end{equation}
By Proposition~\ref{pr:etarkgoodest}, more precisely using \eqref{eq:etarkgoodest:linfty} for $p' = \infty$,
\[
 \abs{I_3} \aleq{} \abs{v(0)} \Vert \varphi \Vert_{L^1} R^{-s},
\]
so
\begin{equation}\label{eq:eqdeflaps:I3}
 \lim_{R \to \infty} I_3 = 0.
\end{equation}
As for $v_2$, 
\[
\begin{ma} 
\intl_{\R^n} \laps{s} v_2\ \varphi &=& \intl_{\R^n} \abs{\cdot}^s\ v_2^\wedge(\cdot)\ \varphi^\wedge(-\cdot)\\
&=&  \intl_{\R^n} \abs{\xi}^{s}\ \brac{v_2 \ast \varphi(-\cdot)}^\wedge(\xi)\ d\xi\\
&=& c_n \intl_{\R^n} \abs{x}^{-n-s}\ \brac{v_2 \ast \varphi(-\cdot)}(x)\ dx.
\end{ma}
\]
The last equality is true, as $\supp (v_2 \ast \varphi) \subset \R^n \backslash B_\varepsilon(0)$ and (see \cite[Theorem 2.4.6]{GrafC08})
\[
 \intl_{\R^n} \abs{\xi}^s\ \psi^\wedge(\xi)\ d\xi = c_n \intl_{\R^n} \abs{y}^{-n-s}\ \psi(y)\ dy,\quad \mbox{for any $\psi \in C_0^\infty(\R^n \backslash \{0\})$.}
\]
Consequently, as the integrals involved converge absolutely, Fubini's theorem implies
\[
\begin{ma}
&&\intl_{\R^n} \laps{s} v_2\ \varphi\\
 &=& c_n \intl_{\R^n} \intl_{\R^n} \varphi(-y)\ \frac{v_2(x-y)}{\abs{x}^{n+s}}\ dy\ dx\\
&=&c_n \intl_{B_{2\varepsilon}} \varphi(-y) \intl_{\R^n \backslash B_\varepsilon} \eta_R(x-y)(1-\eta_{4\varepsilon}(x-y)) \frac{v(x-y)-v(0)}{\abs{x}^{n+s}}\ dx\ dy.
\end{ma}
\]
By Lebesgue's dominated convergence theorem,
\begin{equation}\label{eq:eqdeflaps:I2}
\begin{ma}
\lim_{R\to \infty} I_2 &=& c_n \intl_{\R^n} \varphi(-y) \intl_{\R^n \backslash B_\varepsilon} (1-\eta_{4\varepsilon}(x-y)) \frac{v(x-y)-v(0)}{\abs{x}^{n+s}}\ dx\ dy\\
&=& c_n \intl_{\R^n} \varphi(-y) \intl_{\R^n} (1-\eta_{4\varepsilon}(x-y)) \frac{v(x-y)-v(0)}{\abs{x}^{n+s}}\ dx\ dy.
\end{ma}
\end{equation}
Together, we infer from equations \eqref{eq:eqdeflaps:I4}, \eqref{eq:eqdeflaps:I3} and \eqref{eq:eqdeflaps:I2} that for any $\varepsilon \in (0,1)$ and any $\varphi \in C_0^\infty(B_{2\varepsilon}(0))$,
\[
\begin{ma}
  \intl_{\R^n} \laps{s}v\ \varphi &=& \intl_{\R^n} \eta_{4\varepsilon} (v-v(0))\ \laps{s}\varphi\\
  &&\quad + c_n \intl_{\R^n} \varphi(-y) \intl_{\R^n} (1-\eta_{4\varepsilon}(x-y)) \frac{v(x-y)-v(0)}{\abs{x}^{n+s}}\ dx\ dy.
\end{ma}
\]
We choose a specific $\varphi := \omega \varepsilon^{-n} \eta_{\varepsilon}$, where $\omega > 0$ is chosen such that
\begin{equation}\label{eq:eqdeflaps:intvpeq1}
 \intl_{\R^n} \varphi = \intl_{\R^n} \abs{\varphi} = 1.
\end{equation}
The function $\laps{s}v$ is continuous because for $v \in \Sw(\R^n)$ in particular $(\laps{s}v)^\wedge \in L^1(\R^n)$. Consequently,
\[
\lim_{\varepsilon \to 0} \intl_{\R^n} \laps{s} v\ \varphi = \laps{s} v(0).
\]
It remains to compute the limit $\varepsilon \to 0$ of
\[
 \widetilde{I} := \intl_{\R^n} \eta_{4\varepsilon} (v-v(0))\ \laps{s}\varphi,
\]
and
\[
 \widetilde{II} := \intl_{\R^n} \varphi(-y) \intl_{\R^n} (1-\eta_{4\varepsilon}(x-y)) \frac{v(x-y)-v(0)}{\abs{x}^{n+s}}\ dx\ dy.
\]
As for $\tilde{I}$, by Proposition~\ref{pr:etarkgoodest}, that is \eqref{eq:etarkgoodest:linfty} for $p' = \infty$, applied to $\varphi$,
\[
\begin{ma}
\abs{\widetilde{I}} &\aleq{}& \varepsilon^{-n-s}\ \intl_{B_{8\varepsilon}(0)} \abs{v(y)-v(0)}\ dy\\
&\aleq{}& \Vert \nabla v \Vert_{L^\infty}\ \varepsilon^{-n-s+1} \abs{B_{8\varepsilon}}\\
&\aleq{}& \Vert \nabla v \Vert_{L^\infty}\ \varepsilon^{1-s}.
\end{ma}
\]
As $s < 1$, this implies
\[
 \lim_{\varepsilon \to 0} \widetilde{I} = 0.
\]
As for $\widetilde{II}$, we write
\[
\begin{ma}
&& \varphi(-y) (1-\eta_{4\varepsilon}(x-y)) \frac{v(x-y)-v(0)}{\abs{x}^{n+s}}\\
&=& \varphi(-y) \frac{v(x)-v(0)}{\abs{x}^{n+s}}\\
&&- \eta_{4\varepsilon}(x-y)\ \varphi(-y)\ \frac{v(x)-v(0)}{\abs{x}^{n+s}}\\
&&+  \varphi(-y) (1-\eta_{4\varepsilon}(x-y)) \frac{v(x-y)-v(x)}{\abs{x}^{n+s}}\\
&=:& ii_1 + ii_2 + ii_3.
\end{ma}
\]
By choice of $\varphi$, and by Fubini's theorem which is applicable as all integrals are absolutely convergent,
\[
 \intl_{\R^n} \intl_{\R^n} ii_1\ dy\ dx= \intl_{\R^n} \frac{v(x)-v(0)}{\abs{x}^{n+s}}\ dx.
\]
Moreover, using \eqref{eq:eqdeflaps:intvpeq1}
\[
 \intl_{\R^n} \intl_{\R^n} \abs{ii_2}\ dy\ dx \aleq{} \Vert \nabla v \Vert_{L^\infty}\ \intl_{B_{10\varepsilon}(0)} \frac{1}{\abs{x}^{n+s-1}}\ dx \aleq{} \varepsilon^{1-s},
\]
and
\[
 \intl_{\R^n} \intl_{\R^n} \abs{ii_3}\ dy\ dx \aleq{} \varepsilon\ \Vert \nabla v \Vert_{L^\infty}\ \intl_{\R^n \backslash B_{\varepsilon}(0)} \frac{1}{\abs{x}^{n+s}}\ dx \aleq{} \varepsilon^{1-s}.
\]
As a consequence, we can conclude
\[
 \lim_{\varepsilon \to 0} \widetilde{II} = \intl_{\R^n} \frac{v(x)-v(0)}{\abs{x}^{n+s}}\ dx.
\]
\end{proofP}

If $s \in [1,2)$ the integral definition for $\laps{s}$ in Proposition~\ref{pr:eqdeflaps1} is potentially non-convergent, so we will have to rewrite it as follows. 
\begin{proposition}\label{pr:eqdeflaps2}
Let $s \in (0,2)$. Then,
\[
 \laps{s} v (\bar{y})= \frac{1}{2} c_n \intl_{\R^n} \frac{v(\bar{y}-x)+v(\bar{y} + x)-2v(\bar{y})}{\abs{x}^{n+s}}\ dx.
\]
\end{proposition}

\begin{remark}
This is consistent with Proposition~\ref{pr:eqdeflaps2}. In fact, if $s \in (0,1)$
\[
\intl_{\R^n} \frac{v(y+x)-v(y)}{\abs{x}^{n+s}}\ dx= \intl_{\R^n} \frac{v(y-x)-v(y)}{\abs{x}^{n+s}} \ dx,
\]
just by transformation rule and the symmetry of the kernel $\frac{1}{\abs{x}^{n+s}}$. For this argument to be true, the condition $s \in (0,1)$ is necessary, because it guarantees the absolute convergence of the integrals above.
\end{remark}

\begin{proofP}{\ref{pr:eqdeflaps2}}
This is done analogously to Proposition~\ref{pr:eqdeflaps1}, where one replaces $v(\cdot)$ by $v(\cdot) + v(-\cdot)$ and uses that
\[
\brac{\laps{s} v}(0) = \frac{1}{2} \brac{\laps{s} \brac{v(-\cdot)}(0) + \laps{s} \brac{v(\cdot)}(0)}.
\] 
Then, the involved integrals converge for any $s \in (0,2)$, as 
\[                   
\abs{v(x)+v(-x)-2v(0)} \leq \Vert \nabla^2 v \Vert_{L^\infty}\ \abs{x}^2.
\]
\end{proofP}
\begin{proposition}\label{pr:eqpdeflapscpr}
For any $s \in (0,2)$, $v,w \in \Sw(\R^n)$
\[
 \intl_{\R^n} \laps{s}v\ w = c_n \intl_{\R^n} \intl_{\R^n} \frac{\brac{v(x)-v(y)}\ \brac{w(y)-w(x)}}{\abs{x-y}^{n+s}}\ dx\ dy.
\]
\end{proposition}
\begin{proofP}{\ref{pr:eqpdeflapscpr}}
We have for $v, w \in \Sw(\R^n)$, $x \in \R^n$ by several applications of the transformation rule
\begin{equation}\label{eq:eqpdeflapscpr:intdy}
\begin{ma}
 &&\intl_{\R^n} \brac{v(y+x)+v(y-x)-2v(y)}\ w(y)\ dy\\
&=&\intl_{\R^n} v(y+x)w(y) + v(y)\ w(y+x) - v(y)w(y) - v(y+x)w(y+x)\ dy\\
&=&\intl_{\R^n} v(y+x)\ \brac{w(y)-w(y+x)} + v(y)\ \brac{w(y+x) - w(y)}\ dy\\
&=&\intl_{\R^n} \brac{v(y+x)-v(y)}\ \brac{w(y)-w(y+x)}\ dy.
\end{ma}
\end{equation}
As all involved integrals converge absolutely and applying Fubini's theorem,
\[
\begin{ma}
&&\intl_{\R^n} \laps{s} v(y)\ w(y)\ dy\\
&\overset{\sref{P}{pr:eqdeflaps2}}{=}& c_n \intl_{\R^n} \intl_{\R^n} \frac{\brac{v(y+x)+v(y-x)-2v(y)}\ w(y)}{\abs{x}^{n+s}}\ dx\ dy\\
&=& c_n \intl_{\R^n} \intl_{\R^n} \frac{\brac{v(y+x)+v(y-x)-2v(y)}\ w(y)}{\abs{x}^{n+s}}\ dy\ dx\\
&\overset{\eqref{eq:eqpdeflapscpr:intdy}}{=}& c_n \intl_{\R^n} \intl_{\R^n} \frac{\brac{v(y+x)-v(y)}\ \brac{w(y)-w(y+x)}}{\abs{x}^{n+s}}\ dy\ dx.
\end{ma}
\]
\end{proofP}%
In particular the following equivalence result holds:
\begin{proposition}[Fractional Laplacian - Integral Definition]\label{pr:equivlaps}
Let $s \in (0,1)$. For a constant $c_n > 0$ and for any $v \in \Sw(\R^n)$
\[
\Vert \laps{s} v \Vert^2_{L^2(\R^n)} = c_n \intl_{\R^n}\ \intl_{\R^n} \frac{\abs{v(x)-v(y)}^2}{\abs{x-y}^{n+2s}}\ dx\ dy.
\]
In particular, the function 
\[
(x,y) \in \R^n \times \R^n \mapsto \frac{\abs{v(x)-v(y)}^2}{\abs{x-y}^{n+2s}} 
\]
belongs to $L^1(\R^n \times \R^n)$ whenever $v \in H^s(\R^n)$.
\end{proposition}

Next, we will introduce the pseudo-norm $[v]_{D,s}$, a quantity which for $s \in (0,1)$ actually is equivalent to the local, homogeneous $H^{s}$-norm, see \cite{Tartar07}, \cite{TaylorI}. But we will not use this fact as we will work with $s = \frac{n}{2}$ for $n \in \N$, including $n \in \N$ greater than $4$. Nevertheless, we will see in Section \ref{sec:homognormhn2} that $[v]_{D,\frac{n}{2}}$ is ``almost'' comparable to $\Vert \lapn v \Vert_{L^2(D)}$.

\begin{definition}\label{def:lochnnorm}
For a domain $D \subset \R^n$ and $s \geq 0$ we set 
\begin{equation}\label{eq:defhsloc}
 \left ([u]_{D,s} \right )^2 :=  \intl_{D} \intl_{D} \frac{\abs{\nabla^{\lfloor s \rfloor}u(z_1) - \nabla^{\lfloor s \rfloor}u(z_2)}^2}{\abs{z_1-z_2}^{n+2(s-\lfloor s \rfloor)}} \ dz_1\ dz_2
\end{equation}
if $s \not \in \N_0$. If $s \in \N_0$ we just define $[u]_{D,s} = \Vert \nabla^s u \Vert_{L^2(D)}$. 
\end{definition}
\begin{remark}
By the definition of $[\cdot]_{D,s}$ it is obvious that for any polynomial $P$ of degree less than $s$,
\[
[v+P]_{D,s} = [v]_{D,s}.
\]
\end{remark}

\section{Mean Value Poincar\'{e} Inequality of Fractional Order}\label{sec:poincmv}
\begin{proposition}[Estimate on Convex Sets] \label{pr:annulusuxmuy:convex}
Let $D$ be a convex, bounded domain and $\gamma < n+2$, then for any $v \in C^\infty(\R^n)$,
\[
 \intl_D \intl_D \frac{\abs{v(x)-v(y)}^2}{\abs{x-y}^\gamma}\ dx\ dy \leq C_{D,\gamma}\ \intl_D \abs{\nabla v(z)}^2\ dz.
\]
If $\gamma = 0$, the constant $C_{D,\gamma} = C_n\ \abs{D} \diam(D)^2$.
\end{proposition}
\begin{proofP}{\ref{pr:annulusuxmuy:convex}}
By the Fundamental Theorem of Calculus,
\[
\begin{ma}
 &&\intl_{D} \intl_{D} \frac{\abs{v(x) - v(y)}^2}{\abs{x-y}^\gamma}\ dx\ dy\\
&\leq& \intl_{t=0}^1 \intl_{D} \intl_{D} \frac{\abs{\nabla v(x+t(y-x))}^2}{\abs{x-y}^{\gamma-2}}\ dx\ dy\ dt\\
&\leq& \intl_{t=0}^{\frac{1}{2}} \intl_{D} \intl_{D} \frac{\abs{\nabla v(x+t(y-x))}^2}{\abs{x-y}^{\gamma-2}}\ dx\ dy\ dt\\
&&\quad + \intl_{t=\frac{1}{2}}^{1} \intl_{D} \intl_{D} \frac{\abs{\nabla v(x+t(y-x))}^2}{\abs{x-y}^{\gamma-2}}\ dy\ dx\ dt.
\end{ma}
\]
Using the convexity of $D$, more precisely using the fact that the transformation $x \mapsto x+t(y-x)$ maps $D$ into a subset of $D$,
\[
\begin{ma}
&\leq& \intl_{t=0}^{\frac{1}{2}} \intl_{D} \intl_{D} \frac{\abs{\nabla v(z)}^2}{(1-t)^{2-\gamma}\abs{z-y}^{\gamma-2}}\ (1-t)^{-n}\ dz\ dy\ dt\\
&&\quad + \intl_{t=\frac{1}{2}}^{1} \intl_{D} \intl_{D} \frac{\abs{\nabla v(z)}^2}{t^{2-\gamma}\abs{x-z}^{\gamma-2}}\ t^{-n}\ dz\ dx\ dt\\
&\aleq{}& \intl_{D} \abs{\nabla v(z)}^2 \intl_{D} \abs{z-z_2}^{2-\gamma}\ dz_2\ \ dz\\
&\overset{\gamma < n+2}{\aleq{}}& \intl_{D} \abs{\nabla v(z)}^2 \ dz.\\
\end{ma}
\]
\end{proofP}

An immediate consequence for $\gamma = 0$ is the classic Poincar\'{e} inequality for mean values on convex domains.
\begin{lemma}\label{la:poincCMV}
There is a uniform constant $C > 0$ such that for any $v \in C^\infty(\R^n)$ and for any convex, bounded set $D \subset \R^n$
\[
 \intl_D \abs{v - (v)_D}^2 \leq C\ \brac{\diam(D)}^2\ \Vert \nabla v \Vert_{L^2(D)}^2.
\]
\end{lemma}
In the following two sections we prove in Lemma~\ref{la:poincmv} and Lemma~\ref{la:poincmvAn} higher (fractional) order analogues of this Mean-Value-Poincar\'e-Inequality, on the ball and on the annulus, respectively. More precisely, for $\eta_{r}^k$ from Section \ref{ss:cutoff} we will only show that
\[
 \Vert \laps{s} (\eta_r^k v) \Vert_{L^2(\R^n)} \aleq{} \Vert \laps{s} v \Vert_{L^2(\R^n)},
\]
if $v$ satisfies a mean value condition, similar to the following: For some $N \in \N_0$ and a domain $D \subset \R^n$ (in our example e.g. $D = \supp \eta_{r}^k$ and $N = \lceil s \rceil - 1$)
\begin{equation}\label{eq:meanvalueszero}
\mvintl_{D} \partial^\alpha v = 0, \quad \mbox{for any multiindex $\alpha \in (\N_0)^n$, $\abs{\alpha} \leq N$}.
\end{equation}
The necessary ingredients are not that different from those in the proofs of similar statements as e.g. in \cite{DR09Sphere} or \cite[Proposition 3.6.]{GM05} and can be paraphrased as follows: For any $s > 1$ we can decompose $\laps{s}$ into $\laps{t} \circ T$ for some $t \in (0,1)$ and where $T$ is a classic differential operator possibly plugged behind a Riesz-transform. So, we first focus in Proposition~\ref{pr:mvpoinc} on the case $\laps{s}$ where $s \in (0,1)$. There we first use the integral representation of $\laps{t}$ as in Section \ref{ss:idlaps} and then apply in turns the fundamental theorem of calculus and the mean value condition.
\subsection{On the Ball}
We premise some very easy estimates.
\begin{proposition}\label{pr:intbrxmyest}
For $s \in (0,1)$, there exists a constant $C_s > 0$ such that for any $x \in B_r(x_0)$
\[
\intl_{B_r(x_0)} \frac{1}{\abs{x-y}^{n+2s-2}}\ dy \leq C_s\ r^{2-2s},
\]
and
\[
\intl_{\R^n \backslash B_{2r}(x_0)} \frac{1}{\abs{x-y}^{n+2s}}\ dy \leq C_s\ r^{-2s}.
\]
\end{proposition}
\begin{proofP}{\ref{pr:intbrxmyest}}
We have
\[
\begin{ma}
\intl_{B_r(x_0)} \frac{1}{\abs{x-y}^{n+2s-2}}\ dy &\leq& \intl_{B_{2r}(0)} \frac{1}{\abs{z}^{n+2s-2}}\ dz\\
&\overset{s < 1}{\aeq{s}}& (2r)^{2-2s}
\end{ma}
\]
and
\[
\begin{ma}
\intl_{\R^n \backslash B_{2r}(x_0)} \frac{1}{\abs{x-y}^{n+2s}}\ dy &\leq& 2^{n+2s}\ \intl_{\R^n \backslash B_{2r}(0)} \frac{1}{\abs{z}^{n+2s}}\ dz\\
&\overset{s >0}{\aeq{s}}& (2r)^{-2s}.
\end{ma}
\]
\end{proofP}

\begin{proposition}\label{pr:doubleintvmvest}
Let $\gamma \in [0,n+2)$, $N \in \N$. Then for a constant $C_{N,\gamma}$ and for any $v \in C^\infty(\R^n)$ satisfying \eqref{eq:meanvalueszero} on some $D = B_{r} \subset \R^n$,
\[
\intl_{B_{r}} \intl_{B_{r}} \frac{\abs{v(x)-v(y)}^2}{\abs{x-y}^\gamma}\ dy\ dx \leq 
C_{N,\gamma}\ r^{2N-\gamma} \intl_{B_{r}} \intl_{B_{r}} \abs{\nabla^{N}v(x) - \nabla^{N}v(y)}^2\ dx\ dy.
\]
\end{proposition}
\begin{proofP}{\ref{pr:doubleintvmvest}}
It suffices to prove this proposition for $B_1(0)$ and then scale the estimate. So let $r = 1$. By Proposition~\ref{pr:annulusuxmuy:convex}, 
\[
\begin{ma}
&&\intl_{B_1} \intl_{B_1} \frac{\abs{v(x)-v(y)}^2}{\abs{x-y}^\gamma}\ dy\ dx\\
&\aleq{}& \intl_{B_1} \abs{\nabla v(z)}^2\ dz\\
&\overset{\eqref{eq:meanvalueszero}}{=}&  \intl_{B_1} \abs{\nabla v(z) - (\nabla v)_{B_1}}^2\ dz\\
&\aleq{}& \intl_{B_1}\ \intl_{B_1} \abs{\nabla v(z) - \nabla v(z_2)}^2\ dz\ dz_2\\
\end{ma}
\]
Iterating this procedure $N$ times with repeated use of Proposition~\ref{pr:annulusuxmuy:convex} for $\gamma = 0$, we conclude. 
\end{proofP}

\begin{proposition}\label{pr:mvpoinc}
For any $N \in \N_0$, $s \in [0,1)$ there is a constant $C_{N,s} > 0$ such that the following holds. For any $v \in C^\infty(\R^n)$, $r > 0$, $x_0 \in \R^n$ such that \eqref{eq:meanvalueszero} holds on $D = B_{4r}(x_0)$ we have for all multiindices $\alpha$, $\beta \in (\N_0)^n$, $\abs{\alpha} + \abs{\beta} = N$
\[
\left \Vert \laps{s} \left ((\partial^\alpha \eta_{r,x_0}) (\partial^\beta v) \right ) \right \Vert_{L^2(\R^n)}
\leq C_{N,s}\ [v]_{B_{4r}(x_0),N+s}.
\]
\end{proposition}
\begin{proofP}{\ref{pr:mvpoinc}}
The case $s = 0$ follows by the classic Poincar\'e inequality, so let from now on $s \in (0,1)$. Set
\[
w(y) := (\partial^\alpha \eta_{r}(y)) (\partial^\beta v(y)).
\]
Note that $\supp w \subset B_{2r}$. Moreover, by the definition of $\eta_{r}$, we have
\begin{equation}
\label{eq:loc:partialalphaetaest} 
\abs{w} \leq C_{\alpha}\ r^{-\abs{\alpha}} \abs{\partial^\beta v} \leq C_N r^{\abs{\beta}-N} \abs{\partial^\beta v}.
\end{equation}
By Proposition~\ref{pr:equivlaps} we have to estimate
\[
\begin{ma}
\Vert \laps{s} w \Vert_{L^2}^2 &\aeq{}&\intl_{\R^n}\ \intl_{\R^n} \frac{\abs{w(x)-w(y)}^2}{\abs{x-y}^{n+2s}}\ dx\ dy\\
&=& \intl_{B_{4r}}\ \intl_{B_{4r}} \frac{\abs{w(x)-w(y)}^2}{\abs{x-y}^{n+2s}}\ dx\ dy\\
&&\quad +2\intl_{B_{4r}}\ \intl_{\R^n \backslash B_{4r}} \frac{\abs{w(x)-w(y)}^2}{\abs{x-y}^{n+2s}}\ dx\ dy\\
&&\quad + \intl_{\R^n \backslash B_{4r}}\ \intl_{\R^n \backslash B_{4r}} \frac{\abs{w(x)-w(y)}^2}{\abs{x-y}^{n+2s}}\ dx\ dy\\
&=& \intl_{B_{4r}}\ \intl_{B_{4r}} \frac{\abs{w(x)-w(y)}^2}{\abs{x-y}^{n+2s}}\ dx\ dy\\
&&\quad+ 2\intl_{B_{4r}}\ \abs{w(y)}^2 \intl_{\R^n \backslash B_{4r}} \frac{1}{\abs{x-y}^{n+2s}}\ dx\ dy\\
&=:& I + 2II. 
\end{ma}
\]
To estimate $II$, we use the fact that $\supp w \subset B_{2r}$ and the second part of Proposition~\ref{pr:intbrxmyest} to get
\[
\begin{ma}
\abs{II} &\aleq{s}& r^{-2s} \intl_{B_{4r}} \abs{w(y)}^2 \ dy\\
&\overset{\eqref{eq:loc:partialalphaetaest}}{\aleq{}}& r^{2(\abs{\beta}-N-s)} \intl_{B_{4r}} \abs{\partial^{\beta}v(y)}^2\ dy\\
&\overset{\eqref{eq:meanvalueszero}}{\aleq{}}&r^{2(\abs{\beta}-N-s)} \intl_{B_{4r}} \abs{\partial^{\beta}v(y)- \brac{\partial^{\beta}v}_{B_{4r}}}^2\ dy\\
&\aleq{}& r^{2(\abs{\beta}-N-s)-n} \intl_{B_{4r}} \intl_{B_{4r}} \abs{\partial^{\beta}v(y)-\partial^{\beta}v(x)}^2\ dy\ dx.
\end{ma}
\]
As $\partial^\beta v$ satisfies \eqref{eq:meanvalueszero} for $N-\abs{\beta}$, by Proposition~\ref{pr:doubleintvmvest} for $\gamma = 0$,
\[
\intl_{B_{4r}} \intl_{B_{4r}} \abs{\partial^{\beta}v(y)-\partial^{\beta}v(x)}^2\ dy\ dx \aleq{N} r^{2(N-\abs{\beta})} \intl_{B_{4r}} \intl_{B_{4r}} \abs{\nabla^N v(y)-\nabla^N v(x)}^2\ dx\ dy.
\]
Furthermore, we have for $x,y \in B_{4r}$
\[
r^{-n-2s} \aleq{s} \abs{x-y}^{-n-2s},
\]
which altogether implies that
\[
\abs{II} \aleq{} [v]_{B_{4r},N+s}.
\]
In order to estimate $I$, note that 
\[
\begin{ma}
&&\abs{w(x)-w(y)}\\
 &\leq& \Vert \partial^\alpha \eta_r\Vert_{L^\infty} \abs{\partial^\beta v(x)- \partial^\beta v(y)} + \Vert \nabla \partial^\alpha \eta_r \Vert_{L^\infty}\ \abs{x-y}\ \abs{\partial^\beta v(y)}\\
&\aleq{N}& r^{-\abs{\alpha}} \abs{\partial^\beta v(x)- \partial^\beta v(y)} + r^{-\abs{\alpha}-1}\abs{x-y}\  \abs{\partial^\beta v(y)}.
\end{ma}
\]
Thus, we can decompose $\abs{I} \aleq{} \abs{I_1} + \abs{I_2}$ where
\[
I_1 = r^{2(\abs{\beta}-N)} \intl_{B_{4r}}\ \intl_{B_{4r}} \frac{\abs{\partial^\beta v(x)-\partial^\beta v(y)}^2}{\abs{x-y}^{n+2s}}\ dx\ dy,
\]
and
\[
\begin{ma}
I_2 &=& r^{2(\abs{\beta}-N-1)} \intl_{B_{4r}}\ \intl_{B_{4r}} \frac{\abs{\partial^\beta v(y)}^2}{\abs{x-y}^{n-2+2s}}\ dx\ dy\\
&\overset{\ontop{\sref{P}{pr:intbrxmyest}}{s < 1}}{\aleq{}}& r^{2(\abs{\beta}-N)-2s} \intl_{B_{4r}} \abs{\partial^\beta v(y)}^2\ dy\\
&\overset{\eqref{eq:meanvalueszero}}{\aleq{}}& r^{2(\abs{\beta}-N)-(n+2s)} \intl_{B_{4r}} \intl_{B_{4r}} \abs{\partial^\beta v(y) - \partial^\beta v(z)}^2\ dy\ dz.
\end{ma}
\]
Using again that $\partial^\beta v$ satisfies \eqref{eq:meanvalueszero} for $N-\abs{\beta}$ on $B_{4r}$, by Proposition~\ref{pr:doubleintvmvest} for $\gamma = n+2s$
\[
\begin{ma}
\abs{I_1} &\aleq{}&  r^{-n-2s} \intl_{B_{4r}} \intl_{B_{4r}} \abs{\nabla^N u (x)-\nabla^N u(y)}^2\ dx\ dy\\
&\aleq{}&\intl_{B_{4r}} \intl_{B_{4r}} \frac{\abs{\nabla^N u (x)-\nabla^N u(y)}^2}{\abs{x-y}^{n+2s}}\ dx\ dy,
\end{ma}
\]
and the same for $I_2$. This concludes the case $s > 0$.
\end{proofP}

\begin{lemma}[Poincar\'{e} inequality with mean value condition (Ball)]\label{la:poincmv}
For any $N \in \N_0$, $s \in [0,N+1)$, $t \in [0,N+1-s)$ there is a constant $C_{N,s,t}$ such that the following holds. For any $r > 0$, $x_0 \in \R^n$ and any $v \in C^\infty(\R^n)$ satisfying \eqref{eq:meanvalueszero} for $N$ and on $D = B_{4r}(x_0)$, we have 
\[
\begin{ma}
\Vert \laps{s} \eta_{r,x_0} v \Vert_{L^2(\R^n)} &\leq& C_{s,t}\ r^{t}\ [v]_{B_{4r}(x_0),s+t}\\
&\leq& C_{s,t}\ r^{t} \Vert \laps{s+t} v \Vert_{L^2(\R^n)}.
\end{ma}
\]
\end{lemma}
\begin{proofL}{\ref{la:poincmv}}
We have
\[
 \laps{s} \approx \laps{\gamma} \laps{\delta} \lap^{K}
\]
for \[
\begin{ma}
\gamma &=& s-\lfloor s \rfloor \in [0,1),\\
\delta &=& \lfloor s \rfloor -2 \left \lfloor \frac{\lfloor s \rfloor}{2} \right \rfloor \in \{0,1\},\\
K &=& \left \lfloor \frac{\lfloor s \rfloor}{2} \right \rfloor \in \N_0.
\end{ma}
\]
More precisely, if $\delta = 1$ (cf. Remark \ref{rem:lapsClasLap}),
\[
\laps{s} = c_n \Rz_{i} \laps{\gamma} \partial_i \lap^{K},
\]
and if $\delta = 0$,
\[
\laps{s} = c_n \laps{\gamma} \lap^{K}.
\]
As the Riesz Transform $\Rz_{i}$ is a bounded operator from $L^2$ into $L^2$ we can estimate both cases by
\[
\Vert \laps{s} (\eta_r v) \Vert_{L^2} \aleq{} \suml_{\ontop{\alpha,\beta \in (\N_0)^n}{\abs{\alpha} + \abs{\beta} = 2K+\delta}} \Vert \laps{\gamma} \left ((\partial^\alpha \eta_r) (\partial^\beta v) \right ) \Vert_{L^2}.
\]
This and Proposition~\ref{pr:mvpoinc} imply
\[
 \Vert \laps{s} (\eta_r v) \Vert_{L^2}^2 \aleq{s} \brac{[v]_{B_{4r}(x_0),s}}^2\\
\]
If $t = 0$ this gives the claim. So let now $t > 0$. If $s \in \N$, we have by the mean value property (which holds for $\nabla^s v$ as $s < N+1$, so $s \leq N$)
\[
\begin{ma}
 [v]_{B_{4r}(x_0),s}^2 &\aeq{}& \Vert \nabla^s v \Vert_{L^2}^2\\
&\aleq{}& r^{-n} \intl_{B_{4r}} \intl_{B_{4r}} \brac{\nabla^s u(x) - \nabla^s u(y)}^2\ dx\ dy\\
&\aleq{}& \intl_{B_{4r}} \intl_{B_{4r}} \frac{\brac{\nabla^s u(x) - \nabla^s u(y)}^2}{\abs{x-y}^n}\ dx\ dy.\\ 
\end{ma}
\]
So for every $s > 0$ we have
\[
\begin{ma}
 [v]_{B_{4r}(x_0),s}^2 &\aleq{}& \intl_{B_{4r}} \intl_{B_{4r}} \frac{\brac{\nabla^{\lfloor s\rfloor} u(x) - \nabla^{\lfloor s\rfloor} u(y)}^2}{\abs{x-y}^{n+2(s-\lfloor s \rfloor) }} \ dx\ dy.\\
\end{ma}
\]
If $\lfloor s \rfloor = \lfloor s+t \rfloor$, this implies using $\abs{x-y} \ageq{} r$ for $x,y \in B_{4r}$,
\[
 [v]_{B_{4r}(x_0),s}^2 \aleq{} r^{2t} [v]_{B_{4r}(x_0),s+t}^2.
\]
If $\lfloor s \rfloor < \lfloor s+t \rfloor \leq N$, $\nabla^{\lfloor s \rfloor} v$ satisfies the mean value condition \eqref{eq:meanvalueszero} up to the order $N-\lfloor s \rfloor \geq 1$ as $\lfloor s \rfloor < N$.\\
With this in mind one can see, using Proposition~\ref{pr:doubleintvmvest}, if $s+t > \lfloor s + t \rfloor$ 
\[
 [v]_{B_{4r}(x_0),s} \leq r^{\lfloor s+t\rfloor-s} [v]_{B_{4r}(x_0),\lfloor s + t\rfloor}
\]
or else if $s+t = \lfloor s + t \rfloor$
\[
 [v]_{B_{4r}(x_0),s} \leq r^{t} [v]_{B_{4r}(x_0),s + t}.
\]
In the former case, we can again use that $\abs{x-y} \ageq{} r$ for any $x,y \in B_{4r}$ to conclude.
\end{proofL}
\begin{remark}
By obvious modifications of the proofs, one checks that the result of Lemma~\ref{la:poincmv} is also valid if $v$ satisfies \eqref{eq:meanvalueszero} on a ball $B_{\lambda r}$ for $\lambda \in (0,4)$. The constant then depends also on $\lambda$.
\end{remark}

\subsection{On the Annulus}
In order to get an estimate similar to Proposition~\ref{pr:annulusuxmuy:convex} on the annulus, Proposition~\ref{pr:annulusuxmuy}, we would like to divide the annulus in finitely many convex parts. As this is clearly not possible, we have to enlarge the non-convex part of the annulus.
\begin{proposition}[Convex cover]\label{pr:convexcoverI}
Let $A = B_2 \backslash B_1(0)$ or $B_2 \backslash B_{\frac{1}{2}}(0)$. Then for each $\varepsilon > 0$ there is $\lambda = \lambda_\varepsilon > 0$, $M = M_\varepsilon \in \N$ and a family of open sets $C_j \subset \R^n$, $j \in \{1,\ldots,M\}$ such that the following holds.
\begin{itemize}
 \item For each $j \in \{1,\ldots,M\}$ the set $C_j$ is convex.
 \item The union 
\[
B_2 \backslash B_1 \subset \bigcup\limits_{j=1}^M C_j \subset B_2 \backslash B_{1-\varepsilon}\quad \mbox{or} \quad B_2 \backslash B_{\frac{1}{2}} \subset \bigcup \limits_{j=1}^M C_j \subset B_2 \backslash B_{\frac{1}{2}-\varepsilon},
\]
respectively.
\item For each $i,j \in \{1,\ldots,M\}$ such that $C_i \cap C_j \neq \emptyset$
\[
\conv{C_i \cup C_{j}} \subset B_2 \backslash B_{1-\varepsilon}\quad \mbox{or} \quad \conv{C_i \cup C_{j}} \subset B_2 \backslash B_{\frac{1}{2}-\varepsilon},
\]
respectively, where $\conv{C_i \cup C_{j}}$ denotes the convex hull of $C_i \cup C_{j}$.
\item For each $x,y \in A$, at least one of the following conditions holds
	\begin{itemize}
 		\item[(i)] $\abs{x-y} \geq \lambda$ or
		\item[(ii)] both $x,y \in C_j$ for some $j$.
	\end{itemize}
\end{itemize}
\end{proposition}
\begin{proofP}{\ref{pr:convexcoverI}}
We sketch the case $B_2 \backslash B_1$. Fix $\varepsilon > 0$ and denote by 
\[
 \S := \{x \in \R^n:\ \abs{x} = 1 \} \subset \R^n.
\]
For $r > 0$ and $x \in \S$ we define
\[
S_r(x) := \S \cap B_r(x).
\]
For any $r > 0$ we can pick $(x_k)_{k=1}^M \subset \S$, such that $\{S_r(x_k)\}_{k=1}^M$ covers all of $\S$ where $M = M_r \in \N$ is a finite number. We set $S_k := S_{2r}(x_k)$. If $r = r_\varepsilon > 0$ is chosen small enough, one can also guarantee that the convex hull $\conv{S_k \cup S_l}$ for every $k,l \in \{1,\ldots,M\}$ with $S_k \cap S_l \neq \emptyset$ is a subset of $B_{1} \backslash B_{\frac{1}{2}-\varepsilon}$.\\
The sets $C_j$ are then defined as
\[
 C_j = \conv{\left \{  x \in \R^n:\ \abs{x} < 2,\ x = \alpha y\ \mbox{for $\alpha > 1$ and $y \in S_j$} \right \}}.
\]
They obviously satisfy the first three properties.\\
In order to prove the last property, note that
\[
 \abs{x-y} \geq \abs{\frac{x}{\abs{x}} - \frac{y}{\abs{y}}} \quad \mbox{for all $x,y \in B_2 \backslash B_1$}.
\]
So assume there is $x,y \in B_2 \backslash B_1$ such that $\{x,y\} \not \subset C_j$ for all $j = 1,\ldots,M$. But this in particular implies that for some $k = 1,\ldots,M$, $\frac{x}{\abs{x}} \in S_{r}(x_k)$ but $\frac{y}{\abs{y}} \not \in S_{2r}(x_k)$. In particular, for a constant $\lambda = \lambda_r$ only depending on $r$ and the dimension $n$,
\[
 \abs{\frac{y}{\abs{y}} - \frac{x}{\abs{x}}} \geq \lambda_r.
\]
\end{proofP}

\begin{proposition}\label{pr:annulusuxmuy:gammaeq0}
Let $A = B_2 \backslash B_1(0)$ or $B_2 \backslash B_{\frac{1}{2}}(0)$. Then for any $\varepsilon > 0$, there exists a constant $C_{\varepsilon} > 0$ so that the following holds. For any $v \in C^\infty(\R^n)$
\[
 \intl_A \intl_A \abs{v(x) - v(y)}^2\ dx\ dy \leq C_{\varepsilon}\ \intl_{\tilde{A}}  \abs{\nabla v}^2(z)\ dz,
\]
where $\tilde{A} = B_2 \backslash B_{1-\varepsilon}(0)$ or $B_2 \backslash B_{\frac{1}{2}-\varepsilon}(0)$, respectively.
\end{proposition}
\begin{proofP}{\ref{pr:annulusuxmuy:gammaeq0}}
By Proposition~\ref{pr:convexcoverI} we can estimate
\[
\begin{ma}
 &&\intl_A \intl_A \abs{v(x) - v(y)}^2\ dx\ dy\\
&\leq& \suml_{i,j = 1}^M \intl_{C_i} \intl_{C_j} \abs{v(x) - v(y)}^2\ dx\ dy\\
&=:& \suml_{i,j = 1}^M I_{i,j}.
\end{ma}
\]
If $i=j$ we have by convexity of $C_i$ and Proposition~\ref{pr:annulusuxmuy:convex}
\[
 I_{i,j} \leq C_{C_j}\ \intl_{C_j}  \abs{\nabla v}^2(z)\ dz \leq C_{\varepsilon}\ \intl_{\tilde{A}} \abs{\nabla v}^2(z)\ dz.
\]
If $i$ and $j$ are such that $C_i \cap C_j \neq \emptyset$,
\[
\begin{ma}
 I_{i,j} &\leq& \intl_{\conv{C_i\cup C_j}} \intl_{\conv{C_i\cup C_j}} \abs{v(x) - v(y)}^2\ dx\ dy\\
&\overset{\sref{P}{pr:annulusuxmuy:convex}}{\aleq{}}& \intl_{\conv{C_i\cup C_j}} \abs{\nabla v}^2\\
&\overset{\sref{P}{pr:convexcoverI}}{\aleq{}}& \intl_{\tilde{A}} \abs{\nabla v}^2.\\
\end{ma}
\]
Finally, in any other case for $i$, $j$, there are indices $k_l \in \{1,\ldots,M\}$, $l=1,\ldots,L$, such that $k_1 = i$ and $k_L = j$ and $C_{k_l} \cap C_{k_{l+1}} \neq 0$. Let's abbreviate
\[
 (v)_k := \mvint_{C_k} v.
\]
With this notation,
\[
\begin{ma}
&&I_{i,j}\\
&=& \intl_{C_i} \intl_{C_j} \abs{v(x) - v(y)}^2\ dx\ dy\\
&\leq& C_M \left ( \intl_{C_i} \intl_{C_j} \abs{v(x) - (v)_j}^2 + \suml_{l=1}^{L-1} \abs{(v)_{k_l} - (v)_{k_{l+1}}}^2 + \abs{(v)_i - v(y)}^2\ dx\ dy \right )\\
&\aleq{}&  I_{j,j} + \suml_{l=i}^{L} I_{k_l,l_{l+1}} + I_{i,i} .
\end{ma}
\]
So we can reduce this case for $i$, $j$, to the estimates of the previous cases and conclude.
\end{proofP}

As a consequence we have
\begin{proposition}\label{pr:annulusuxmuy}
Let $A = B_2 \backslash B_1(0)$ or $B_2 \backslash B_{\frac{1}{2}}(0)$. Then for any $\varepsilon > 0$, $\gamma \in [0,n+2)$ there exists a constant $C_{\varepsilon,\gamma} > 0$ so that the following holds. For any $v \in C^\infty(\R^n)$
\[
 \intl_A \intl_A \frac{\abs{v(x) - v(y)}^2}{\abs{x-y}^\gamma}\ dx\ dy \leq C_{\varepsilon,\gamma}\ \intl_{\tilde{A}}  \abs{\nabla v(z)}^2\ dz,
\]
where $\tilde{A} = B_2 \backslash B_{1-\varepsilon}(0)$ or $B_2 \backslash B_{\frac{1}{2}-\varepsilon}(0)$, respectively.
\end{proposition}
\begin{proofP}{\ref{pr:annulusuxmuy}}
By Proposition~\ref{pr:convexcoverI} we can divide
\[
 \begin{ma}
&&\intl_A \intl_A \frac{\abs{v(x) - v(y)}^2}{\abs{x-y}^\gamma}\ dx\ dy\\
&\leq& \sum_{j=1}^M \intl_{C_j} \intl_{C_j} \frac{\abs{v(x) - v(y)}^2}{\abs{x-y}^\gamma}\ dx\ dy + \lambda^{-\gamma} \intl_{A} \intl_{A} \abs{v(x)-v(y)}^2\ dx\ dy.
 \end{ma}
\]
These quantities are estimated by Proposition~\ref{pr:annulusuxmuy:convex} and Proposition~\ref{pr:annulusuxmuy:gammaeq0}, respectively.
\end{proofP}

As a consequence of the last estimate, analogously to the case of a ball, we can prove the following Poincar\'{e}-inequality:
\begin{lemma}[Poincar\'{e}'s Inequality with mean value condition (Annulus)]\label{la:poincmvAn}
For any $N \in \N_0$, $s \in [0,N+1)$, $t \in [0,N+1-s)$ there is a constant $C_{N,s,t}$ such that the following holds. For any $v \in C^\infty(\R^n)$, $x_0 \in \R^n$, $r > 0$ such that $v$ satisfies \eqref{eq:meanvalueszero} for $N$ on $D = A_k = B_{2^{k+1} r}(x_0) \backslash B_{2^{k-1} r}(x_0)$ or $D = A_k = B_{2^{k+1} r}(x_0) \backslash B_{2^{k} r}(x_0)$ we have 
\[
\Vert \laps{s} (\eta^k_{r,x_0} v) \Vert_{L^2(\R^n)} \leq C_{s,t}\ \brac{2^k r}^{t}\ [v]_{\tilde{A}_k,s+t},\\
\]
where
\[
 \tilde{A}_k = B_{2^{k+2} r}(x_0) \backslash B_{2^{k-2} r}(x_0).  
\]
\end{lemma}
\begin{proofL}{\ref{la:poincmvAn}}
The methods used are similar to the case of the ball, cf. in particular the proof of Proposition~\ref{pr:mvpoinc} and Lemma~\ref{la:poincmv}. We only sketch the case $t =0$.\\
One picks an open set $E$, $\supp \eta_{r}^k \subset E \subset \tilde{A}_k$, such that $\dist(\partial E, \supp \eta_{r}^k) \in (0,\varepsilon)$ and $\dist(E, \partial \tilde{A}_k) > 0$, for very small $\varepsilon > 0$. As in the case of a ball, one can reduce the problem to essentially estimate
\[
 \intl_{E} \intl_{E} \frac{ \abs{\partial^\beta v(x)-\partial^\beta v(y)}^2}{\abs{x-y}^{n+2s}}\ dx\ dy,
\]
\[
 \intl_{\supp \eta_{r}^k} \frac{ \abs{\partial^\beta v(x)}^2}{\varepsilon^{n+2s}}\ dx,
\]
for some multiindex $\abs{\beta} \leq N$. Applying the mean value condition \eqref{eq:meanvalueszero} and Proposition~\ref{pr:annulusuxmuy}, these integrals are estimated by
\[
 \intl_{\tilde{E}} \abs{\nabla \partial^\beta v(z)}^2\ dz,
\]
for some $E \subset \tilde{E} \subset \tilde{A}_k$, where $\tilde{E}$ is a bit ``fatter`` than $E$. Iterating this (and in every step thickening the set $E$ by a tiny factor $\varepsilon$) until we reach the highest differentiability, we conclude.
\end{proofL}
\begin{remark}\label{rem:poincmvAn}
Again, one checks that the claim is also satisfied if $v$ satisfies \eqref{eq:meanvalueszero} on a possibly smaller annulus, making the constant depending also on this scaling.
\end{remark}

\subsection{Comparison between Mean Value Polynomials on Different Sets}
For a bounded domain $D \subset \R^n$ and $N \in \N_0$ and for $v \in \Sw(\R^n)$ we define the polynomial $P(v) \equiv P_{D,N}(v)$ to be the unique polynomial of order $N$ such that
\begin{equation}\label{eq:pmeanvaluedef}
 \mvintl_{D} \partial^\alpha (v-P(v)) = 0, \quad \mbox{for every multiindex $\alpha \in (\N_0)^n$, $\abs{\alpha} \leq N$.}
\end{equation}
The goal of this section is to estimate in Proposition~\ref{pr:etarkpbmpkest} and Lemma~\ref{la:mvestbrakShrpr} the difference
\[
 P_{B_r(x),N}(v) - P_{B_{2^k r}(x)\backslash B_{2^{k-1}r}(x),N}(v), \quad \mbox{for $k \in \Z$}   
\]
in terms of $\laps{s} v$. To do so, we adapt the methods applied in the proof of \cite[Lemma 4.2]{DR09Sphere}, the main difference being that we have to extend their argument to polynomials of degree greater than zero.\\
We will need an inductive description of $P(v)$. First, for a multiindex $\alpha = (\alpha_1,\ldots,\alpha_n)$ set
\[
\alpha! := \alpha_1 !\ldots \alpha_n! = \partial^\alpha x^\alpha.
\]
For $i \in  \{0, \ldots, N\}$ set
\begin{equation}\label{eq:definQmeanval}
\begin{ma}
 Q^{i}_{D,N}(v) &:=& Q^{i+1}_{D,N}(v) + \suml_{\abs{\alpha} = i} \frac{1}{\alpha!}\ x^\alpha \mvintl_D \partial^\alpha (v-Q^{i+1}_{D,N}(v)),\\
Q^{N}_{D,N}(v) &:=& \suml_{\abs{\alpha}=N} \frac{1}{\alpha!}\ x^\alpha \mvintl_D \partial^\alpha v.
\end{ma}
\end{equation}
One checks that 
\begin{equation}\label{eq:palphaqieqp}
\partial^\alpha Q^i = \partial^\alpha P,\quad \mbox{whenever $\abs{\alpha} \geq i$,}
\end{equation}
and in particular $Q^0 = P$.\\
Moreover we will introduce the following sets of annuli:
\[
A_j \equiv A_j(r) = B_{2^j r} \backslash B_{2^{j-1}r},\quad \tilde{A}_j \equiv \tilde{A}_j(r) := A_j \cup A_{j+1}.
\]
\begin{proposition}\label{pr:vmqmmvi}
For any $N \in \N$, $s \in (N,N+1]$, $D \subseteq D_2 \subset \R^n$ smoothly bounded domains there is a constant $C_{D_2,D,N,s}$ such that the following holds: Let $v \in C^\infty(\R^n)$. For any multiindex $\alpha \in (\N_0)^n$ such that $\abs{\alpha} = i \leq N-1$,
\[
\begin{ma} 
&&\intl_{D_2} \Babs{\partial^\alpha (v-Q_{D,N}^{i+1}(v)) - \brac{\partial^\alpha (v-Q^{i+1}_{D,N}(v))}_D}\\
&\leq& C_{D_2,D,N,s}\  \left (\frac{\abs{D_2}}{\abs{D}}\right )^{\frac{1}{2}} \diam(D_2)^{\frac{n}{2}+s-N}\ [v]_{D_2,s}
\end{ma}
\]
where $[v]_{D,s}$ is defined as in \eqref{eq:defhsloc}.\\
If $D = r \tilde{D}$, $D_2 = r \tilde{D}_2$, then $C_{D_2,D,N,s} = r^{N-i} C_{\tilde{D}_2,\tilde{D},N,s}$.%
\end{proposition}
\begin{proofP}{\ref{pr:vmqmmvi}}
Let us denote
\[
 I := \intl_{D_2} \Babs{\partial^\alpha (v-Q_{D,N}^{i+1}) - \brac{ \partial^\alpha (v-Q^{i+1}_{D,N}(v))}_D }.
\]
A first application of H\"older's and classic Poincar\'{e}'s inequality yields
\[
 I \leq C_{D,D_2}\ \abs{D_2}^{\frac{1}{2}}\ \Vert \nabla \partial^\alpha (v-Q_{D,N}^{i+1}) \Vert_{L^2(D_2)}.
\]
Next, \eqref{eq:palphaqieqp} and the definition of $P$ in \eqref{eq:pmeanvaluedef} imply that we can apply classic Poincar\'e inequality $N-i-1$ times more, to estimate $I$ by
\[ 
\begin{ma}
&\leq& C_{D_2,D,N}\ \abs{D_2}^{\frac{1}{2}}\ \Vert \nabla^N (v-P_{D,N}(v)) \Vert_{L^2(D_2)}\\
&\overset{\eqref{eq:definQmeanval}}{=}& C_{D_2,D,N}\ \abs{D_2}^{\frac{1}{2}}\  \Vert \nabla^N v- \brac{ \nabla^N v}_D \Vert_{L^2(D_2)}.
\end{ma}
\]
If $s = N+1$, yet another application of Poincar\'{e}'s inequality yields the claim. In the case $s \in (N,N+1)$, we estimate further
\[
\begin{ma}
I &\leq& C_{D_2,D,N}\ \left (\frac{\abs{D_2}}{\abs{D}}\right )^{\frac{1}{2}}\ \left ( \intl_{D_2} \intl_{D_2} \abs{\nabla^N v(x) - \nabla^N v(y)}^2\ dx\ dy \right )^{\frac{1}{2}},
\end{ma}
\]
which is bounded by
\[C_{D_2,D,N}\ \left (\frac{\abs{D_2}}{\abs{D}}\right )^{\frac{1}{2}}\ \diam(D_2)^{\frac{n+2(s-N)}{2}}\ \left ( \intl_{D_2} \intl_{D_2} \frac{\abs{\nabla^N v(x) - \nabla^N v(y)}^2}{\abs{x-y}^{n+2(s-N)}}\ dx\ dy \right )^{\frac{1}{2}}.\\
\]
The scaling factor for $D = r \tilde{D}$ then follows by the according scaling factors of Poincar\'{e}'s inequality.%
\end{proofP}

\begin{proposition}\label{pr:mvestnachbarn}
For any $N \in \N_0$, $s \in (N,N+1]$, there is a constant $C_{N,s} > 0$ such that the following holds: For any $j \in \Z$, any multiindex $\abs{\alpha} \leq i \leq N$ and $v \in C^\infty(\R^n)$ 
 \[
\left \Vert \partial^\alpha \left (Q^{i}_{A_j,N} - Q^{i}_{A_{j+1},N} \right ) \right \Vert_{L^\infty (A_j)} \leq C_{N,s} (2^j r)^{s-\abs{\alpha}-\frac{n}{2}}\ [v]_{\tilde{A}_j,s}.
\]
\end{proposition}
\begin{proofP}{\ref{pr:mvestnachbarn}}
Assume first that $i = N$. Then if $s \in (N,N+1)$,
\[
 \begin{ma}
  &&\Vert \partial^\alpha (Q^N_{A_j} - Q^N_{A_{j+1}}) \Vert_{L^\infty(A_j)}\\
&\overset{\eqref{eq:definQmeanval}}{\aleq{}}& (2^jr)^{N-\abs{\alpha}} \frac{1}{\abs{A_j}^2} \intl_{\tilde{A}_j} \intl_{\tilde{A}_j} \abs{\nabla^N v(x) - \nabla^N v(y)}\ dx\ dy\\
&\aleq{}&  (2^jr)^{N-\abs{\alpha}} \frac{1}{\abs{A_j}} \left (\intl_{\tilde{A}_j} \intl_{\tilde{A}_j} \abs{\nabla^N v(x) - \nabla^N v(y)}^2\ dx\ dy  \right )^{\frac{1}{2}}\\ 
&\aleq{}&  (2^jr)^{-\abs{\alpha}-\frac{n}{2}+s}  [v]_{\tilde{A}_j,s}. 
 \end{ma}
\]
If $s = N+1$ and $i = N$, one uses classic Poincar\'{e} inequality to prove the claim.\\
Now let $i \leq N-1$, $s \in (N,N+1]$, and assume we have proven the claim for $i+1$. By \eqref{eq:definQmeanval},
\[
\begin{ma}
 &&Q^i_{A_j}-Q^i_{A_{j+1}}\\
&=& Q^{i+1}_{A_j}-Q^{i+1}_{A_{j+1}}\\
&&\quad + \suml_{\abs{\beta} = i} \frac{1}{\beta!}\ x^\beta \left (\mvintl_{A_j} \partial^\beta (v-Q^{i+1}_{A_{j+1}}) - \mvintl_{A_{j+1}} \partial^\beta (v-Q^{i+1}_{A_{j+1}})\right )\\
&&\quad + \suml_{\abs{\beta} = i} \frac{1}{\beta!}\ x^\beta \left (\mvintl_{A_j} \partial^\beta (Q^{i+1}_{A_{j+1}} - Q^{i+1}_{A_j}) \right ).
\end{ma}
\]
Consequently,
\[
\begin{ma}
 &&\Vert \partial^\alpha (Q^i_{A_j}-Q^i_{A_{j+1}}) \Vert_{L^\infty(A_j)}\\
&\aleq{}& \Vert \partial^\alpha (Q^{i+1}_{A_j}-Q^{i+1}_{A_{j+1}}) \Vert_{L^\infty(A_j)}\\
&&\quad +  (2^j r)^{i-\abs{\alpha}} \suml_{\abs{\beta} = i} \mvintl_{A_j} \abs{\partial^\beta (v-Q^{i+1}_{A_{j+1}}) - \mvintl_{A_{j+1}} \partial^\beta (v-Q^{i+1}_{A_{j+1}}) }\\
&&\quad + (2^j r)^{i-\abs{\alpha}} \suml_{\abs{\beta} = i} \Vert \partial^\beta (Q^{i+1}_{A_{j+1}} - Q^{i+1}_{A_j}) \Vert_{L^\infty(A_j)}.
\end{ma}
\]
Then the claim for $i+1$ and Proposition~\ref{pr:vmqmmvi} conclude the proof.
\end{proofP}

\begin{proposition}\label{pr:mvestbrak}
For any $N \in \N_0$, $s \in (N,N+1]$ there is a constant $C_{N,s}$ such that the following holds. For any multiindex $\alpha \in (\N_0)^n$, $\abs{\alpha} \leq i \leq N$, for any $r > 0$, $k \in \Z$ and any $v \in \Sw(\R^n)$ if $s-\frac{n}{2} \not \in \{i,\ldots,N\}$,
\[
 \Vert \partial^\alpha (Q^i_{B_r} - Q^i_{A_k}) \Vert_{L^\infty(\tilde{A}_k)} \leq C_{N,s}\ r^{s-\abs{\alpha}-\frac{n}{2}} \left (2^{k(s-\abs{\alpha}-\frac{n}{2})} + 2^{k(i-\abs{\alpha})} \right )\ [v]_{\R^n,s},
\]
and if $s-\frac{n}{2} \in  \{i,\ldots,N\}$,
\[
\begin{ma}
 &&\Vert \partial^\alpha (Q^i_{B_r} - Q^i_{A_k}) \Vert_{L^\infty(\tilde{A}_k)}\\
 &&\leq C_{N,s}\ r^{s-\abs{\alpha}-\frac{n}{2}}\ 2^{k(i-\abs{\alpha})} \left (\abs{k} +1+2^{k(s-i-\frac{n}{2})} \right ) \ [v]_{\R^n,s}.
 \end{ma}
\]
Here as before, $A_k = B_{2^{k}r}(x) \backslash B_{2^{k-1}r}(x)$ and $\tilde{A}_k = B_{2^{k+1}r}(x) \backslash B_{2^{k-1}r}(x)$.
\end{proposition}
\begin{proofP}{\ref{pr:mvestbrak}}
For the sake of shortness of presentation, let us abbreviate
\[
 d^{i,\alpha}_k := \Vert \partial^\alpha (Q^i_{B_r} - Q^i_{A_k}) \Vert_{L^\infty(\tilde{A}_k)}.
\]
Assume first $i = N$.
\[
 \begin{ma}
  d^{N,\alpha}_k &\overset{\eqref{eq:definQmeanval}}{\aleq{}}& \left \Vert \suml_{\abs{\beta} = N} \frac{\partial^\alpha x^\beta}{\beta!} \left ( \mvintl_{B_r} \partial^\beta v - \mvintl_{A_k} \partial^\beta v \right ) \right \Vert_{L^\infty(\tilde{A}_k)}\\ 
&\aleq{}& (2^k r)^{N-\abs{\alpha}} \abs{\mvintl_{B_r} \nabla^N v - \mvintl_{A_k} \nabla^N v}\\
&\aeq{}& (2^k r)^{N-\abs{\alpha}} \abs{\suml_{l=-\infty}^0 \frac{\abs{A_l}}{\abs{B_r}} \mvintl_{A_l} \nabla^N v - \mvintl_{A_k} \nabla^N v}.
 \end{ma}
\]
As $\frac{\abs{A_l}}{\abs{B_r}} = 2^{ln} (1-2^{-n})$ and thus $\suml_{l=-\infty}^0 \frac{\abs{A_l}}{\abs{B_r}} = 1$, for $k > 0$ we estimate further
\[
 \begin{ma}
  &&d^{N,\alpha}_k\\
&\aleq{}& (2^k r)^{N-\abs{\alpha}} \suml_{l=-\infty}^0 2^{ln} \abs{\mvintl_{A_l} \nabla^N v - \mvintl_{A_k} \nabla^N v}\\
&\aleq{}& (2^k r)^{N-\abs{\alpha}} \suml_{l=-\infty}^0 2^{ln} \suml_{j=l}^{k-1} \abs{\mvintl_{A_j} \nabla^N v - \mvintl_{A_{j+1}} \nabla^N v}\\
&\overset{(\bigstar)}{\aleq{}}& (2^k r)^{N-\abs{\alpha}} \suml_{l=-\infty}^0 2^{ln} \suml_{j=l}^{k-1} (2^{j}r)^{-n} \left (\intl_{\tilde{A}_j}\ \intl_{\tilde{A}_j} \abs{\nabla^N v(x) - \nabla^N v(y)}^2 \ dx\ dy \right )^{\frac{1}{2}}\\
&\aleq{}& (2^k r)^{N-\abs{\alpha}} \suml_{l=-\infty}^0 2^{ln} \suml_{j=l}^{k-1} (2^{j}r)^{-\frac{n}{2}+s-N}\ [v]_{\tilde{A}_j,s}.\\
\end{ma}
\]
Of course, if $s = N+1$, one replaces the estimate in $(\bigstar)$ and uses instead Poincar\'{e}'s inequality. If $k \leq 0$ one has by virtually the same computation,
\[
\begin{ma}
 d^{N,\alpha}_k &\aleq{}& (2^k)^{N-\abs{\alpha}}r^{s-\frac{n}{2}-\abs{\alpha}}\ \Big ( \suml_{l=-\infty}^{k-1} 2^{ln} \suml_{j=l}^{k-1} 2^{j(-\frac{n}{2}+s-N)}\ [v]_{\tilde{A}_j,s}\\
 &&+ \suml_{l=k}^{0} 2^{ln} \suml_{j=k}^{l-1} 2^{j(-\frac{n}{2}+s-N)}\ [v]_{\tilde{A}_j,s} \Big ).\\
\end{ma}
\]
Now we have to take care, whether $s-\frac{n}{2}-N = 0$ or not. Let
\[
 a_k := \begin{cases}
         2^{k(s-\frac{n}{2}-N)}, \quad &\mbox{if $s-\frac{n}{2} - N \neq 0$,}\\
	 \abs{k}, \quad &\mbox{if $s-\frac{n}{2} - N = 0$,}\\
        \end{cases}
\]
and respectively,
\[
 b_l := \begin{cases}
         2^{l(s-\frac{n}{2}-N)}, \quad &\mbox{if $s-\frac{n}{2} - N \neq 0$,}\\
	 \abs{l}, \quad &\mbox{if $s-\frac{n}{2} - N = 0$.}\\
        \end{cases}
\]
With this notation, applying H\"older's inequality for series, $d^{N,\alpha}_k$ is estimated independently of whether $k > 0$ or not, by
\[
(2^k)^{N-\abs{\alpha}} r^{s-\abs{\alpha}-\frac{n}{2}} \suml_{l=-\infty}^0 2^{ln} \left ( a_k + b_l \right )\ \left (\suml_{j=-\infty}^{\infty} [v]_{\tilde{A}_j,s}^2  \right )^{\frac{1}{2}}\\
\]
\[
\begin{ma}
&\aleq{}& r^{s-\frac{n}{2}-\abs{\alpha}} \brac{2^{k(N-\abs{\alpha})}a_k + (2^k)^{N-\abs{\alpha}} \suml_{l=-\infty}^0 2^{ln}b_l} [v]_{\R^n,s}\\
&\aleq{}& r^{s-\frac{n}{2}-\abs{\alpha}}\ [v]_{\R^n,s} \left ( 2^{k(N- \abs{\alpha})} a_k + 2^{k(N-\abs{\alpha})} \right ).
\end{ma}
\]
This concludes the case $i=N$. Next, let $i < N$ and assume the claim is proven for $i+1$.
\[
\begin{ma}
d^{i,\alpha}_{k} &=&\Vert \partial^\alpha (Q_{B_r}^{i} - Q_{A_{k}}^i) \Vert_{L^\infty(\tilde{A}_{k})}\\ 

&\overset{\eqref{eq:definQmeanval}}{\aleq{}}& 
d^{i+1,\alpha}_{k}
+ \suml_{\abs{\beta} = i} \left (2^{k}r \right )^{i-\abs{\alpha}} \abs{ \mvintl_{B_r} \partial^\beta (v-Q_{B_r}^{i+1}) -\mvintl_{A_{k}} \partial^\beta (v-Q_{A_{k}}^{i+1})}\\
&\aleq{}&
d^{i+1,\alpha}_{k}\\
&&+ \suml_{\abs{\beta} = i} \left (2^{k}r \right )^{i-\abs{\alpha}} c_n \suml_{l=-\infty}^0 2^{ln} \abs{\mvintl_{A_l} \partial^\beta (v-Q_{B_r}^{i+1}) -\mvintl_{A_{k}} \partial^\beta (v-Q_{A_{k}}^{i+1})},
 \end{ma}
\]
where $c_n 2^{ln} = \frac{\abs{A_l}}{\abs{B_r}}$, so $\suml_{l = - \infty}^0 c_n 2^{ln} = 1$ as we have done in the case $i=N$ above. We estimate further,
\[
d_{k}^{i,\alpha} \aleq{} d^{i+1,\alpha}_{k} +
\]
\[
+ \suml_{\abs{\beta} = i} \left (2^{k}r \right )^{i-\abs{\alpha}} \suml_{l=-\infty}^0 2^{ln} \left (d^{i+1,\beta}_{l} + \abs{\mvintl_{A_l} \partial^\beta (v-Q_{A_l}^{i+1}) -\mvintl_{A_{k}} \partial^\beta (v-Q_{A_{k}}^{i+1})}\right ).
 \]
As above in the case $i=N$ we use a telescoping series to write
\[
\begin{ma}
 &&\abs{\mvintl_{A_l} \partial^\beta (v-Q_{A_l}^{i+1}) -\mvintl_{A_{k}} \partial^\beta (v-Q_{A_{k}}^{i+1})}\\
&\leq& \suml_{j=l}^{k-1} \abs{\mvintl_{A_j} \partial^\beta (v-Q_{A_j}^{i+1}) -\mvintl_{A_{j+1}} \partial^\beta (v-Q_{A_{j+1}}^{i+1})}\\
&\aleq{}& \suml_{j=l}^{k-1} \left \Vert \partial^\beta (Q^{i+1}_{A_j} - Q^{i+1}_{A_{j+1}} ) \right \Vert_{L^\infty (A_j)}\\
&&\quad + \mvintl_{\tilde{A}_j} \abs{\partial^\beta (v-Q_{A_{j+1}}^{i+1}) -\mvintl_{A_{j+1}} \partial^\beta (v-Q_{A_{j+1}}^{i+1})}\\
&=:& \suml_{j=l}^{k-1} (I_j + II_j).
\end{ma}
\]
Again we should have taken care of whether $l < k-1$ or $k-1 \leq l$, but as in the case $i=N$ both cases are treated the same way. The term $I_j$ is estimated by Proposition~\ref{pr:mvestnachbarn},
\[
 I_j \aleq{} \left (2^j r\right )^{s-\abs{\beta}-\frac{n}{2}}\ [v]_{\tilde{A}_j,s} = \left (2^j r\right )^{s-i-\frac{n}{2}} [v]_{\tilde{A}_j,s}.
\]
And by Proposition~\ref{pr:vmqmmvi},
\[
 II_j \aleq{} (2^j r)^{-n+\frac{n}{2}+s -i}\ [v]_{\tilde{A}_j,s} = (2^j r)^{s-i-\frac{n}{2}}\ [v]_{\tilde{A}_j,s}.
\]
Hence,
\[
\begin{ma}
 &&\abs{\mvintl_{A_l} \partial^\beta (v-Q_{A_l}^{i+1}) -\mvintl_{A_{k}} \partial^\beta (v-Q_{A_{k}}^{i+1})}\\
&\aleq{}& r^{s-i-\frac{n}{2}} \suml_{j=l}^{k-1} (2^j)^{s-i-\frac{n}{2}}\ [v]_{\tilde{A}_j,s}\\
&\aleq{}& r^{s-i-\frac{n}{2}} \left ( a_k + b_l\right ) \left (\suml_{j=l}^{k-1} [v]_{\tilde{A}_j,s}^2 \right )^{\frac{1}{2}},\\
\end{ma}
\]
for $a_k$ and $b_k$ similar to the case $i=N$ above defined as
\[
 a_k := \begin{cases}
         2^{k(s-\frac{n}{2}-i)}, \quad &\mbox{if $s-\frac{n}{2} - i \neq 0$,}\\
	 \abs{k}, \quad &\mbox{if $s-\frac{n}{2} - i = 0$,}\\
        \end{cases}
\]
and respectively,
\[
 b_l := \begin{cases}
         2^{l(s-\frac{n}{2}-i)}, &\quad \mbox{if $s-\frac{n}{2} - i \neq 0$,}\\
	 \abs{l}, \quad &\mbox{if $s-\frac{n}{2} - i = 0$.}\\
        \end{cases}
\]
Plugging all these estimates in, we have achieved the following estimate
\[
 \begin{ma}
  && d^{i,\alpha}_k\\
&\aleq{}& d^{i+1,\alpha}_k + \suml_{\abs{\beta} = i} \left (2^k r \right )^{i-\abs{\alpha}}\ \suml_{l=-\infty}^0 2^{ln} d_l^{i+1,\beta}\\
&&\quad + r^{s-\abs{\alpha}-\frac{n}{2}}\ 2^{k(i-\abs{\alpha})}\ (a_k + 1)\ [v]_{\R^n,s}.
 \end{ma}
\]
In either case, whether $s-\frac{n}{2}-\tilde{i} = 0$ for some $\tilde{i} \geq i$ or not, using the claim for $i+1$ we have
\[
 \begin{ma}
  \suml_{\abs{\beta} = i} \left (2^k r \right )^{i-\abs{\alpha}}\ \suml_{l=-\infty}^0 2^{ln} d_l^{i+1,\beta} \aleq{} C_{N,s}\ r^{s-\frac{n}{2}-\abs{\alpha}} [v]_{\R^n,s},
 \end{ma}
\]
and thus can conclude.
\end{proofP}

As an immediate consequence of Proposition~\ref{pr:mvestbrak} for $i=0$, $\abs{\alpha} = 0$, and $s = \frac{n}{2}$, we get the following two results.
\begin{proposition}\label{pr:etarkpbmpkest}
For a uniform constant $C > 0$, for any $v \in \Sw(\R^n)$, $r > 0$, $k \in \N$
\[
 \Vert \eta_r^k (P_{B_r,\lceil \frac{n}{2} \rceil -1}(v) - P_{A_k,\lceil \frac{n}{2} \rceil -1}(v)) \Vert_{L^\infty(\R^n)} \leq C\  (1+\abs{k}) \Vert \lapn v \Vert_{L^2(\R^n)}.
\]
Here, $A_k = B_{2^{k+1}r}(x) \backslash B_{2^{k}r}(x)$ and $\tilde{A}_k = B_{2^{k+1}r}(x) \backslash B_{2^{k-1}r}(x)$.
\end{proposition}%
\begin{proposition}\label{pr:estetarkvmp}
There exists a constant $C > 0$ such that for any $r > 0$, $x_0 \in \R^n$, $k \in \N_0$, $v \in \Sw(\R^n)$ we have
\[
 \Vert \eta_{r,x_0}^k (v-P) \Vert_{L^2(\R^n)} \leq C \left (2^k r \right )^{\frac{n}{2}}\ (1+\abs{k})\ \Vert \lapn v \Vert_{L^2(\R^n)},
\]
where $P$ is the polynomial of order $N := \left \lceil \frac{n}{2} \right \rceil-1$ such that $v-P$ satisfies the mean value condition \eqref{eq:meanvalueszero} in $D := B_{2r}$.
Here, in a slight abuse of notation for $k = 0$, $\eta_r^k \equiv \eta_{r}-\eta_{\frac{1}{2}r}$ for $\eta$ from Section \ref{ss:cutoff}.
\end{proposition}
\begin{proofP}{\ref{pr:estetarkvmp}}
Let $P_k$ be the polynomial of order $N = \left \lceil \frac{n}{2} \right \rceil-1$ such that $v$ satisfies the mean value condition \eqref{eq:meanvalueszero} in $B_{2^k r} \backslash B_{2^{k-1} r}$. We then have,
\[
 \Vert \eta_r^k (v-P) \Vert_{L^2(\R^n)} \aleq{}  \Vert \eta_r^k (v-P_k) \Vert_{L^2(\R^n)} + \left (2^k r\right )^{\frac{n}{2}} \Vert \eta_r^k (P-P_k) \Vert_{L^\infty}.
\]
As Proposition~\ref{pr:etarkpbmpkest} estimates the second part of the last estimate, we are left to estimate
\[
 \Vert \eta_r^k (v-P_k) \Vert_{L^2(\R^n)} \leq C \left (2^k r \right )^{\frac{n}{2}}\ \Vert \lapn v \Vert_{L^2(\R^n)}.
\]
But this is rather easy and can be proven by similar arguments as used in the proof of Lemma~\ref{la:poincmvAn}, see also Remark \ref{rem:poincmvAn}: as by classic Poincar\'e inequality and the fact that by choice of $P_k$ the mean values over $B_{2^{k+1}r} \backslash B_{2^{k}r}$ of all derivatives up to order $\lfloor \frac{n}{2} \rfloor$ of $v-P_k$ are zero, so
\[
 \Vert \eta_r^k (v-P_k) \Vert_{L^2(\R^n)} \aleq{} \left (2^k r \right )^{\lfloor \frac{n}{2} \rfloor}\ \Vert \nabla^{\lfloor \frac{n}{2} \rfloor} (v-P_k) \Vert_{L^2(B_{2^{k+1}r} \backslash B_{2^{k-1}r})}.
\]
If $n$ is an even number, this proves the claim. If $n$ is odd, we use again the mean value condition to see
\[
\begin{ma}
 &&\Vert \nabla^N (v-P_k) \Vert_{L^2(B_{2^{k+1}r} \backslash B_{2^{k-1}r})}^2\\
&\aleq{}& \mvintl_{B_{2^{k+1}r} \backslash B_{2^{k}r}} \intl_{B_{2^{k+1}r} \backslash B_{2^{k-1}r}} \abs{\nabla^N v(x) - \nabla^N v(y)}^2\ dx\ dy\\
&\aleq{}& \left (2^k r \right )^{n-2 \lfloor \frac{n}{2} \rfloor} \intl_{B_{2^{k+1}r} \backslash B_{2^{k-1}r}}\intl_{B_{2^{k+1}r} \backslash B_{2^{k-1}r}} \frac{\abs{\nabla^N v(x) - \nabla^N v(y)}^2}{\abs{x-y}^{2n-2\lfloor \frac{n}{2} \rfloor}}\ dx\ dy \\ 
&\aleq{}& \left (2^k r \right )^{n-2\lfloor \frac{n}{2} \rfloor}\ \Vert \lapn v \Vert_{L^2(\R^n)}^2.
\end{ma}
\]
Taking the square root of the last estimate, one concludes.
\end{proofP}
We will need the following a little bit sharper version of Proposition~\ref{pr:etarkpbmpkest}, too.
\begin{lemma}\label{la:mvestbrakShrpr}(compare \cite[Lemma 4.2]{DR09Sphere}))\\
Let $N := \lceil \frac{n}{2} \rceil-1$ and $\gamma > N$. Then for $\tilde{\gamma} = -N + \min (n,\gamma)$ and for any $v \in \Sw(\R^n)$, $B_r(x_0) \subset \R^n$, $r > 0$,
\[
\suml_{k=1}^\infty 2^{-\gamma k} \left \Vert (P_{B_r,N}(v) - P_{A_k,N}(v)) \right \Vert_{L^{\infty}(\tilde{A}_k)} \leq C_\gamma \ \suml_{j=-\infty}^\infty 2^{-\abs{j}\tilde{\gamma}}\ [v]_{\tilde{A}_j,\frac{n}{2}}.
\]
Here, $A_k = B_{2^{k+1}r}(x) \backslash B_{2^{k}r}(x)$ and $\tilde{A}_k = B_{2^{k+1}r}(x) \backslash B_{2^{k-1}r}(x)$.
\end{lemma}
\begin{remark}
More precisely, we will prove for $i \in \{0,\ldots,N\}$, that whenever $\gamma > N$, $\abs{\alpha} \leq i$, for $\tilde{\gamma} := \min (n-N,\gamma -N)$
\begin{equation}\label{eq:mvpoincSiagammaClaim}
\sum_{k=-\infty}^\infty 2^{-\gamma \abs{k}} \Vert \partial^\alpha (Q^i_{B_r} - Q^i_{A_k}) \Vert_{L^\infty(\tilde{A}_k)} \leq C_{\gamma,N} \left (r^{-\abs{\alpha}}\ \suml_{j=-\infty}^\infty 2^{-\abs{j} \tilde{\gamma}}\ [v]_{\tilde{A}_j,\frac{n}{2}} \right ).
\end{equation}
This more precise statement will be used in the estimates for the homogeneous norm $[\cdot]_s$, Lemma~\ref{la:comps01}.
\end{remark}
\begin{proofL}{\ref{la:mvestbrakShrpr}}
As in the proof of Proposition~\ref{pr:mvestbrak}, set
\[
 d^{i,\alpha}_k := \Vert \partial^\alpha (Q^i_{B_r} - Q^i_{A_k}) \Vert_{L^\infty(\tilde{A}_k)}.
\]
Moreover, we set
\[
  S^{i,\alpha}_\gamma := \suml_{k=1}^\infty 2^{-\gamma k}\ d^{i,\alpha}_k
\]
and
\[
  S^{i,\alpha}_{-\gamma} := \suml_{k=-\infty}^0 2^{\gamma k}\ d^{i,\alpha}_k.
\]
Then, by the computations in the proof of Proposition~\ref{pr:mvestbrak}, for any $\abs{\alpha} \leq N$,
\[
\begin{ma}
 S^{N,\alpha}_\gamma &\aleq{}& r^{-\abs{\alpha}}\ \suml_{k=1}^\infty \suml_{l=-\infty}^0\   \suml_{j=l}^{k-1} 2^{-jN+ln-\gamma k+kN-k\abs{\alpha}}\ [v]_{\tilde{A}_j,\frac{n}{2}}\\
&=& r^{-\abs{\alpha}}\ \suml_{j=-\infty}^0 2^{-jN}\ [v]_{\tilde{A}_j,\frac{n}{2}}\ \suml_{l=-\infty}^j\ \suml_{k=1}^\infty 2^{ln}\ 2^{k(N-\gamma-\abs{\alpha})}\\
&&+\ r^{-\abs{\alpha}}\ \suml_{j=1}^\infty 2^{-jN}\ [v]_{\tilde{A}_j,\frac{n}{2}}\ \suml_{l=-\infty}^0\ \suml_{k=j+1}^\infty 2^{ln}\ 2^{k(N-\gamma-\abs{\alpha})}\\
&\overset{\gamma > N}{\aleq{}}&
r^{-\abs{\alpha}}\ \suml_{j=-\infty}^0 2^{j(n-N)}\ [v]_{\tilde{A}_j,\frac{n}{2}}\\
&&+\ r^{-\abs{\alpha}}\ \suml_{j=1}^\infty 2^{j(-\gamma-\abs{\alpha})}\ [v]_{\tilde{A}_j,\frac{n}{2}}.\\
\end{ma}
\]
Similarly,
\[
\begin{ma}
 S^{N,\alpha}_{-\gamma} &\aleq{}& r^{-\abs{\alpha}}\ \suml_{k=-\infty}^0 \suml_{l=-\infty}^{k-1}\   \suml_{j=l}^{k-1} 2^{-jN+ln+\gamma k+kN-k\abs{\alpha}}\ [v]_{\tilde{A}_j,\frac{n}{2}}\\
&& +\ r^{-\abs{\alpha}}\ \suml_{k=-\infty}^0 \suml_{l=k}^0\   \suml_{j=k}^{l-1} 2^{-jN+ln+\gamma k+kN-k\abs{\alpha}}\ [v]_{\tilde{A}_j,\frac{n}{2}} \\
&\aleq{}& r^{-\abs{\alpha}}\ 
 \sum_{j=-\infty}^0 2^{-jN} [v]_{\tilde{A}_j,\frac{n}{2}}\ \sum_{k = j+1}^{0} \sum_{l=-\infty}^{j} 2^{ln} 2^{k(\gamma+N-\abs{\alpha})}\\
 && +\ r^{-\abs{\alpha}}\ \suml_{j=-\infty}^0 2^{-jN} [v]_{\tilde{A}_j,\frac{n}{2}} \sum_{k=-\infty}^{j} \sum_{l=j+1}^{0} 2^{ln} 2^{k(\gamma+N-\abs{\alpha})}\\
&\overset{\abs{\alpha}\leq N}{\aleq{}}& r^{-\abs{\alpha}}\ 
 \sum_{j=-\infty}^0 2^{j(n-N)} [v]_{\tilde{A}_j,\frac{n}{2}}\\
 && +\ r^{-\abs{\alpha}}\ \suml_{j=-\infty}^0 2^{j(\gamma-\abs{\alpha})} [v]_{\tilde{A}_j,\frac{n}{2}}.
\end{ma}
\]
For $0 \leq i \leq N-1$, again using the computations done for the proof of Proposition~\ref{pr:mvestbrak},
\[
 \begin{ma}
  S^{i,\alpha}_\gamma &\aleq{}& S^{i+1,\alpha}_\gamma\\
&&+\ r^{i-\abs{\alpha}} \sum_{\abs{\beta} = i}\ \suml_{k=1}^\infty 2^{k (i-\abs{\alpha}-\gamma)} S_{-n}^{i+1,\beta} \\
&&+\ r^{-\abs{\alpha}} \suml_{k=1}^\infty 2^{k(i-\abs{\alpha}-\gamma)} \suml_{l=-\infty}^0 2^{ln} \suml_{j=l}^{k-1} 2^{-ji}\ [v]_{\tilde{A}_j,\frac{n}{2}}\\
&\overset{\gamma > i}{\aleq{}}& S^{i+1,\alpha}_\gamma\\
&&+\ r^{i-\abs{\alpha}} \sum_{\abs{\beta} = i}\ S^{i+1,\beta}_{-n}\\
&&+\ r^{-\abs{\alpha}}\ \suml_{j=-\infty}^0 2^{j(n-i)}\ [v]_{\tilde{A}_j,\frac{n}{2}}\\
&&+\ r^{-\abs{\alpha}}\ \suml_{j=1}^\infty 2^{j(-\gamma-\abs{\alpha})}\ [v]_{\tilde{A}_j,\frac{n}{2}}\\
&\overset{i\leq N}{\aleq{}}& S^{i+1,\alpha}_\gamma\\
&&+\ r^{i-\abs{\alpha}} \sum_{\abs{\beta} = i}\ S^{i+1,\beta}_{-n}\\
&&+\ r^{-\abs{\alpha}}\ \suml_{j=-\infty}^0 2^{j(n-N)}\ [v]_{\tilde{A}_j,\frac{n}{2}}\\
&&+\ r^{-\abs{\alpha}}\ \suml_{j=1}^\infty 2^{j(-\gamma-\abs{\alpha})}\ [v]_{\tilde{A}_j,\frac{n}{2}}.\\
\end{ma}
\]
And
\[
 \begin{ma}
  S^{i,\alpha}_{-\gamma} 

&\aleq{}& S^{i+1,\alpha}_{-\gamma}\\
&&+\ r^{i-\abs{\alpha}} \sum_{\abs{\beta} = i}\ \suml_{k=-\infty}^0 2^{k (i-\abs{\alpha}+\gamma)} S_{-n}^{i+1,\beta} \\
&&+\ r^{-\abs{\alpha}} \suml_{k=-\infty}^0 2^{k(i-\abs{\alpha}+\gamma)} \suml_{l=-\infty}^{k-1} 2^{ln} \suml_{j=l}^{k-1} 2^{-ji}\ [v]_{\tilde{A}_j,\frac{n}{2}}\\
&&+\ r^{-\abs{\alpha}} \suml_{k=-\infty}^0 2^{k(i-\abs{\alpha}+\gamma)} \suml_{l=k}^{0} 2^{ln} \suml_{j=k-1}^{l} 2^{-ji}\ [v]_{\tilde{A}_j,\frac{n}{2}}\\

&\aleq{}& S^{i+1,\alpha}_{-\gamma}\\
&&+\ r^{i-\abs{\alpha}} \sum_{\abs{\beta} = i}\ S_{-n}^{i+1,\beta} \\
&&+\ r^{-\abs{\alpha}} \suml_{j=-\infty}^0 2^{-ji} [v]_{\tilde{A}_j,\frac{n}{2}} \suml_{l=-\infty}^{j} \suml_{k=j+1}^{0} 2^{ln} 2^{k(i-\abs{\alpha}+\gamma)}\\
&&+\ r^{-\abs{\alpha}} \sum_{j=-\infty}^0 2^{-ji} [v]_{\tilde{A}_j,\frac{n}{2}} \sum_{k=-\infty}^{j} \sum_{l=j}^0 2^{ln} 2^{k(i-\abs{\alpha}+\gamma)}\\
&\aleq{}& S^{i+1,\alpha}_{-\gamma}\\
&&+\ r^{i-\abs{\alpha}} \sum_{\abs{\beta} = i}\ S_{-n}^{i+1,\beta} \\
&&+\ r^{-\abs{\alpha}} \suml_{j=-\infty}^0 2^{j(n-i)} [v]_{\tilde{A}_j,\frac{n}{2}}\\
&&+\ r^{-\abs{\alpha}} \sum_{j=-\infty}^0 2^{j(\gamma - \abs{\alpha})} [v]_{\tilde{A}_j,\frac{n}{2}}\\

&\overset{i \leq N}{\aleq{}}& S^{i+1,\alpha}_{-\gamma}\\
&&+\ r^{i-\abs{\alpha}} \sum_{\abs{\beta} = i}\ S_{-n}^{i+1,\beta} \\
&&+\ r^{-\abs{\alpha}} \suml_{j=-\infty}^0 2^{j(n-N)} [v]_{\tilde{A}_j,\frac{n}{2}}\\
&&+\ r^{-\abs{\alpha}} \sum_{j=-\infty}^0 2^{j(\gamma - \abs{\alpha})} [v]_{\tilde{A}_j,\frac{n}{2}}\\
\end{ma}
\]
Consequently, one can prove by induction for $i \in \{0,\ldots,N\}$, that \eqref{eq:mvpoincSiagammaClaim} holds whenever $\gamma > N$, $\abs{\alpha} \leq i$, for $\tilde{\gamma} := \min (n-N,\gamma -N)$, i.e.
\[
S^{i,\alpha}_\gamma + S^{i,\alpha}_{-\gamma} \leq C_{\gamma,N} \left (r^{-\abs{\alpha}}\ \suml_{j=-\infty}^\infty 2^{-\abs{j} \tilde{\gamma}}\ [v]_{\tilde{A}_j,\frac{n}{2}} \right ),
\]
Taking $i = 0$, $\alpha = 0$, we conclude.
\end{proofL}

\section{Integrability and Compensation Phenomena}\label{sec:tart}
We will frequently use the following operator
\begin{equation}\label{eq:Hdef}
 H(u,v) := \lapn (u v) - (\lapn u) v - u \lapn v,\quad \mbox{for }u,v \in \Sw(\R^n).
\end{equation}
In general there is no product rule making $H(u,v) \equiv 0$, or $H(u,v)$ an operator of lower order, as would happen if $n \in 4\N$. But in some way this quantity still acts \emph{like} an operator of lower order, as Lemma~\ref{la:tart:prphuvwedgeest} shows.\\
This was observed in \cite{DR09Sphere}. As remarked there, the compensation phenomena that appear are very similar to the ones in Wente's inequality (see the introduction of \cite{DR09Sphere} for more on that). In fact, in this note we would like to stress that even an argument very similar Tartar's proof in \cite{Tartar85} still works.\\
\\
In this section we present a rather simple estimate which somehow models the compensation phenomenon: More specifically, for $p > 0$ we are going to treat in Corollary \ref{co:esttriangle2} the quantity
\[
\abs{\abs{x-y}^p-\abs{y}^p - \abs{x}^p}.
\]
%
%
\begin{proposition}\label{pr:esttriangle1}
For any $x,y \in \R^n$ and any $p > 0$ we have
\[
 \abs{ \abs{x-y}^p - \abs{y}^p} \leq C_p\ \begin{cases}
				      \abs{x}^p \quad &\mbox{if $p \in (0,1)$,}\\
                                      \abs{x}^p + \abs{x} \abs{y}^{p-1} \quad &\mbox{if $p \geq 1$}.
                                     \end{cases}
\]
\end{proposition}
\begin{proofP}{\ref{pr:esttriangle1}}
The inequality is obviously true if $\abs{y} \leq 2  \abs{x}$ or $x = 0$. So assume $x \neq 0$ and $2\abs{x} < \abs{y}$, in particular,
\begin{equation}\label{eq:tart:asuxleq12y}
\abs{y-tx} \geq \abs{y}-t\abs{x} > \brac{1-\frac{t}{2}} \abs{y} > \abs{x}, \quad \mbox{for any $t \in (0,1)$}. 
\end{equation}
We use Taylor expansion for $f(t) = \abs{y-tx}^p$ to write
\[
\abs{\abs{x-y}^p - \abs{y}^p} \aleq{} \sum_{k=1}^{\lfloor p \rfloor} \abs{\frac{d^k}{dt^k} \Big \vert_{t=0}  \abs{y-tx}^p} + \sup_{t \in (0,1)} \abs{\frac{d^{\lfloor p \rfloor + 1}}{dt^{\lfloor p \rfloor + 1}} \abs{y-tx}^p}. 
\]
For $k \geq 1$,
\[
\abs{\frac{d^{k}}{dt^{k}} \abs{y-tx}^p} \aleq{} \abs{y-tx}^{p-k} \abs{x}^k.
\]
So for $1 \leq k \leq \lfloor p \rfloor$,
\[
\abs{\frac{d^k}{dt^k} \Big \vert_{t=0}  \abs{y-tx}^p} \aleq{} \abs{y}^{p-k}\ \abs{x}^k \overset{\abs{x}<\abs{y}}{\aleq{}} \abs{x} \abs{y}^{p-1}.
\]
For $k = \lfloor p \rfloor + 1 > p$, $s \in (0,1)$,
\[
\abs{\frac{d^{k}}{ds^{k}} \abs{y-sx}^p} \aleq{} \abs{y-sx}^{p-k} \abs{x}^k \overset{\eqref{eq:tart:asuxleq12y}}{\aleq{}} \abs{x}^p.
\]
\end{proofP}

%
%
%
Proposition~\ref{pr:esttriangle1} has the following consequence
\begin{corollary}\label{co:esttriangle2}
For any $x,y \in \R^n$ and any $p > 0$, $\theta \in [0,1]$ we have for a uniform constant $C_p > 0$
\[
 \abs{ \abs{x-y}^p - \abs{y}^p - \abs{x}^p} \leq C_p\ \begin{cases}
				      \abs{x}^{p\theta}\ \abs{y}^{p(1-\theta)} \quad &\mbox{if $p \in (0,1]$,}\\
                                      \abs{x}^{p-1}\abs{y} + \abs{x} \abs{y}^{p-1} \quad &\mbox{if $p > 1$}.
                                     \end{cases}
\]
\end{corollary}
\begin{proofC}{\ref{co:esttriangle2}}
We only prove the case $p > 1$, the case $p \in (0,1]$ is similar. By Proposition~\ref{pr:esttriangle1}, 
\[
\begin{ma}
&&\abs{ \abs{x-y}^p - \abs{y}^p - \abs{x}^p}\\
&\aleq{}& \min \left \{\abs{x}^p, \abs{y}^p \right \} + \abs{x}^{p-1}\abs{y} + \abs{y}^{p-1}\abs{x}\\
&\leq& 2\abs{x}^{p-1}\abs{y} + \abs{y}^{p-1}\abs{x}.
\end{ma}
\]
\end{proofC}

\begin{lemma}\label{la:tart:prphuvwedgeest}
For any $u, v \in \Sw(\R^n)$ we have in the case $n = 1,2$
\[ 
\abs{H(u,v)^\wedge} \leq C\ \abs{(\lap^{\frac{n}{8}} u)^\wedge}\ast \abs{(\lap^{\frac{n}{8}} v)^\wedge}(\xi),
\]
and in the case $n \geq 3$
\[
 \abs{(H(u,v))^\wedge} \leq C\ \abs{(\lap^{\frac{n-2}{4}} u)^\wedge} \ast \abs{(\laps{1} v)^\wedge} +  C \abs{(\laps{1} u)^\wedge}  \ast   \abs{(\lap^{\frac{n-2}{4}} v)^\wedge}.
\]
\end{lemma}
\begin{proofL}{\ref{la:tart:prphuvwedgeest}}
As $u,v \in \Sw (\R^n)$ one checks that $H(u,v) \in L^2(\R^n)$ and thus its Fourier -Transform is well defined. Consequently,
\[
\begin{ma}
 (H(u,v))^\wedge(\xi) &=& \abs{\xi}^{\frac{n}{2}} u^\wedge\ast v^\wedge(\xi)  - v^\wedge \ast (\abs{\cdot}^{{\frac{n}{2}}} u^\wedge)(\xi) - u^\wedge \ast (\abs{\cdot}^{{\frac{n}{2}}} v^\wedge) (\xi)\\
&=& \intl_{\R^n} u^\wedge(\xi-y)\ v^\wedge (y)\ 
\left ( 
\abs{\xi}^{{\frac{n}{2}}} - \abs{\xi-y}^{\frac{n}{2}} - \abs{y}^{{\frac{n}{2}}}
\right )\ dy.
\end{ma}
\]
If $n = 1,2$, Corollary \ref{co:esttriangle2} (for $p := \frac{n}{2}$) gives
\[
 \abs{\abs{\xi}^{{\frac{n}{2}}} - \abs{y}^{{\frac{n}{2}}} - \abs{\xi - y}^{{\frac{n}{2}}} } \leq C\ \abs{y}^{\frac{n}{4}}\ \abs{\xi - y}^{\frac{n}{4}},
\]
in the case $n \geq 3$ we have
\[
 \abs{\abs{\xi}^{{\frac{n}{2}}} - \abs{y}^{{\frac{n}{2}}} - \abs{\xi - y}^{{\frac{n}{2}}} } \leq C\ (\abs{y}^{\frac{n-2}{2}}\ \abs{\xi - y} + \abs{\xi - y}^{\frac{n-2}{2}}\ \abs{y}).
\]
This gives the claim.
\end{proofL}

\begin{theorem}\label{th:integrability}
(Cf. \cite{Tartar85}, \cite[Theorem 1.2, Theorem 1.3]{DR09Sphere})\\
Let $u,v \in \mathcal{S}(\R^n)$ and set
\[
 H(u,v) := \lapn  (uv) - v\lapn u - u\lapn v.
\]
Then, 
\[
 \Vert H(u,v)^\wedge \Vert_{L^{2,1}(\R^n)} \leq C_n\ \Vert \lapn  u \Vert_{L^2(\R^n)}\ \Vert \lapn v \Vert_{L^2(\R^n)}.
\]
and
\[
 \Vert H(u,v) \Vert_{L^2(\R^n)} \leq C_n\ \Vert (\lapn u)^\wedge \Vert_{L^{2,\infty}(\R^n)}\ \Vert \lapn v \Vert_{L^2(\R^n)}.
\]
In particular,
\[
 \Vert H(u,v) \Vert_{L^2(\R^n)} \leq C_n\ \Vert \lapn u \Vert_{L^{2}(\R^n)}\ \Vert \lapn v \Vert_{L^2(\R^n)}.
\]
\end{theorem}
\begin{proofT}{\ref{th:integrability}}
Lemma~\ref{la:tart:prphuvwedgeest} implies, in the case $n = 1,2$
\[
 \abs{(H(u,v))^\wedge} \leq C \left ( \abs{\cdot}^{-\frac{n}{4}}  \abs{(\lapn  u)^\wedge}  \right )\ast \left ( \abs{\cdot}^{-\frac{n}{4}}  \abs{(\lapn  v)^\wedge} \right )
\]
and in the case $n \geq 3$
\[
\begin{ma}
 \abs{(H(u,v))^\wedge} &\leq& C \left ( \abs{\cdot}^{-1}  \abs{(\lapn  u)^\wedge}  \right )\ast \left ( \abs{\cdot}^{-\frac{n-2}{2}}  \abs{(\lapn  v)^\wedge} \right )\\
&& +  C \left ( \abs{\cdot}^{-\frac{n-2}{2}}  \abs{(\lapn  u)^\wedge}  \right )\ast \left ( \abs{\cdot}^{-1}  \abs{(\lapn  v)^\wedge} \right ).
\end{ma}
\]
Now we use H\"older's inequality: By Proposition~\ref{pr:dl:lso} we have that
\[
\begin{array}{lll}
\abs{\cdot}^{-\frac{n}{4}} \in L^{4,\infty}(\R^n), \quad & L^{2} \cdot L^{4,\infty} \subset L^{\frac{4}{3},2}, \quad &  L^{2,\infty} \cdot L^{4,\infty} \subset L^{\frac{4}{3},\infty},\\
\abs{\cdot}^{-1} \in L^{n,\infty}(\R^n), \quad & L^{2} \cdot L^{n,\infty} \subset L^{\frac{2n}{n+2},2}, \quad &  L^{2,\infty} \cdot L^{n,\infty} \subset L^{\frac{2n}{n+2},\infty},\\
\abs{\cdot}^{-\frac{n-2}{2}} \in L^{\frac{2n}{n-2},\infty}(\R^n), \quad & L^{2} \cdot L^{\frac{2n}{n-2},\infty} \subset L^{\frac{n}{n-1},2}, \quad &  L^{2,\infty} \cdot L^{\frac{2n}{n-2},\infty} \subset L^{\frac{n}{n-1},\infty}.
\end{array}
\]
Moreover, convolution acts as follows
\[
\begin{array}{lll}
L^{\frac{4}{3},2} \ast L^{\frac{4}{3},2} \subset L^{2,1}, \quad & L^{\frac{4}{3},\infty} \ast L^{\frac{4}{3},2} \subset L^{2},\\
L^{\frac{2n}{n+2},2} \ast L^{\frac{n}{n-1},2} \subset L^{2,1}, \quad & L^{\frac{2n}{n+2},2} \ast L^{\frac{n}{n-1},\infty} + L^{\frac{2n}{n+2},\infty} \ast L^{\frac{n}{n-1},2} \subset L^{2}.
\end{array}
\]
We can conclude.
\end{proofT}
\section{Localization Results for the Fractional Laplacian}\label{sec:locEf}
Even though $\lap^s$ is a nonlocal operator, its ``differentiating force'' concentrates around the point evaluated. Thus, to estimate $\laps{s}$ at a given point $x$ one has to look ``only around'' $x$. In this spirit the following results hold.

\subsection{Multiplication with disjoint support}
In \cite{DR09Sphere} a special case of the following Lemma is used many times. As a consequence of lower order effects appearing when dealing with dimensions and orders greater than one, we will need it in a more general setting, namely for arbitrary homogeneous multiplier operators.
\begin{lemma} \label{la:bs:disjsuppGen}
Let $M$ be an operator with Fourier multiplier $m \in \Sws(\R^n,\Cc)$, $m \in C^\infty(\R^n \backslash \{0\},\Cc)$, i.e.
\[
 Mv := (m v^\wedge)^\vee \quad \mbox{for any $v \in \Sw$}.
\]
If $m$ is homogeneous of order $\delta > -n$, for any $a, b \in \Sw(\R^n,\Cc)$ such that for some $\gamma, d > 0$, $x \in \R^n$, $\supp a \subset B_\gamma(x)$ and $\supp b \subset \R^n \backslash B_{d+\gamma}(x)$,
\[
 \abs{\intl_{\R^n} a\ Mb} \leq C_M\ d^{-n-\delta} \ \Vert a \Vert_{L^1(\R^n)}\ \Vert b \Vert_{L^1(\R^n)}.
\]
\end{lemma}
An immediate consequence, taking $m := \abs{\cdot}^{s+t}$, is
\begin{corollary} \label{co:bs:disjsupp}
Let $s, t > -n$, $s+t > -n$. Then, for all $a,b \in \Sw(\R^n,\Cc)$, such that for some $d, \gamma > 0$, $\supp a \subset B_\gamma(x)$ and $\supp b \subset \R^n \backslash B_{d+\gamma}(x)$,
\[
 \abs{\intl_{\R^n} \laps{s}  a\ \laps{t}  b} \leq C_{n,s,t}\ d^{-(n+s+t)}\ \Vert a \Vert_{L^1}\ \Vert b \Vert_{L^1}.
\]
\end{corollary}
Lemma~\ref{la:bs:disjsuppGen} follows from the following proposition, as the commutation of translations and multiplier operators allows us to assume $\supp a \subset B_{\gamma}(0)$ and $\supp b \subset \R^n \backslash B_{\gamma + d}(0)$.
\begin{proposition} \label{pr:bs:disjsuppGen}
Let $m \in C^\infty(\R^n \backslash \{0\},\Cc)\cap \Sws$. If for some $\delta > -n$ we have that $m(\lambda x) = \lambda^\delta m(x)$ for any $x \in \R^n \backslash \{0\}$ and any $\lambda > 0$,
\[
 \abs{\intl_{\R^n} m\ \varphi^\wedge} \leq C_m\ d^{-n-\delta} \ \Vert \varphi \Vert_{L^1(\R^n)}, \quad \mbox{for any $\varphi \in C_0^\infty(\R^n \backslash \overline{B_d(0)},\Cc)$, $d > 0$}.
\]
\end{proposition}

Proposition~\ref{pr:bs:disjsuppGen} again follows from some general facts about the Fourier Transform on tempered distributions:
\begin{proposition}[Smoothness takes over to Fourier Transform]\label{pr:bs:disjsuppp1}${}$\\
Let $f \in \Sws(\R^n,\Cc)$ and $f \in C^\infty(\R^n \backslash \{ 0 \},\Cc)$. If moreover $f$ is weakly homogeneous of order $\delta \in \R$, i.e.
\[
 f[\varphi(\lambda \cdot )] = \lambda^{-n-\delta} f[\varphi], \quad \mbox{for all $\varphi \in \mathcal{S}(\R^n,\Cc)$,}
\]
then $f^\wedge, f^\vee \in \Sws(\R^n,\Cc)$ also belong to $C^\infty(\R^n \backslash \{0\},\Cc)$.
\end{proposition}
\begin{proofP}{\ref{pr:bs:disjsuppp1}}
We refer to \cite[Proposition 2.4.8]{GrafC08}.
\end{proofP}

\begin{proposition}[Homogeneity takes over to Fourier Transform]\label{pr:bs:disjsuppp2}
Let $f \in \Sws(\R^n,\Cc)$. If $f$ is weakly homogeneous of order $\delta \in \R$, then $g = f^\wedge \in \Sws(\R^n,\Cc)$ and $h = f^\vee \in \Sws(\R^n,\Cc)$ are weakly homogeneous of order $\gamma = -n - \delta$.
\end{proposition}
\begin{proofP}{\ref{pr:bs:disjsuppp2}}
This follows just by the definition of Fourier transform on tempered distributions,
\[
 f^\wedge [\varphi(\lambda \cdot)] = f [\varphi(\lambda \cdot)^\wedge] = \lambda^{-n} f [\varphi^\wedge (\frac{1}{\lambda} \cdot)] = \lambda^{-n} \lambda^{- (-n-\delta)} f [\varphi^\wedge].
\]
The case $f^\vee$ is done the same way.
\end{proofP}

\begin{proposition}[Weak Homogeneity and Strong Homogeneity]\label{pr:bs:disjsuppp3}${}$\\
Let $g \in \Sws(\R^n,\Cc)$, $g \in C^\infty(\R^n \backslash \{0\},\Cc)$. If $g$ is weakly homogeneous of order $\gamma$, then also pointwise
\[
 g(\lambda x) = \lambda^\gamma g(x), \quad \mbox{for every $x \in \R^n \backslash \{0\}$, $\lambda > 0$.}
\]
\end{proposition}
\begin{proofP}{\ref{pr:bs:disjsuppp3}}
We have for any $\varphi \in \Sw(\R^n,\Cc)$, and any $\lambda > 0$
\[
 g[\varphi(\lambda^{-1} \cdot)] = \intl g (x)\ \varphi(\lambda^{-1} x)\ dx = \lambda^n \intl g(\lambda z)\ \varphi(z) dz
\]
and by weak homogeneity
\[
 \lambda^{n+\gamma} g[\varphi] = g[\varphi(\lambda^{-1} \cdot)].
\]
Thus,
\[
 \intl_{\R^n} (\lambda^\gamma \tilde{g}(x) - \tilde{g}(\lambda x)) \varphi(x) = 0, \quad \mbox{for any $\varphi \in \Sw$}
\]
which implies $\lambda^\gamma \tilde{g}(x) = \tilde{g}(\lambda x)$ for any $x \neq 0$.
\end{proofP}

\begin{proposition}[Strong Homogeneity]\label{pr:bs:disjsuppp4}
Let $g \in \Sws(\R^n,\Cc)$, $g \in C^\infty(\R^n \backslash \{ 0\},\Cc)$. If there is $\gamma \leq 0$ such that
\[
 g(\lambda x) = \lambda^\gamma g(x) \quad \mbox{for every $x \in \R^n \backslash \{0\}$, $\lambda > 0$}
\]
then
\[
 \abs{\intl g\ \varphi} \leq d^\gamma \Vert g \Vert_{L^\infty(\S^{n-1})}\ \Vert \varphi \Vert_{L^1(\R^n)},\quad \mbox{for every $\varphi \in C_0^\infty(\R^n\backslash \overline{B_d(0)})$, $d > 0$.}
\]
\end{proposition}
\begin{proofP}{\ref{pr:bs:disjsuppp4}}
For every $\varphi \in C_0^\infty(\R^n\backslash \overline{B_d(0)})$, $d > 0$, we have
\[
 \abs{\intl g(x)\ \varphi(x)\ dx} = \abs{\intl \abs{x}^\gamma\ g\brac{\frac{x}{\abs{x}}}\ \varphi(x)\ dx} \overset{\gamma \leq 0}{\leq} d^{\gamma}\ \Vert g \Vert_{L^\infty(\S^{n-1})}\ \Vert \varphi \Vert_{L^1(\R^n)}.
\]
\end{proofP}
Proposition~\ref{pr:bs:disjsuppp1} - Proposition~\ref{pr:bs:disjsuppp4} imply Proposition~\ref{pr:bs:disjsuppGen}.
\subsection{Equations with disjoint support localize}
As a consequence of Corollary \ref{co:bs:disjsupp} we can \emph{de facto} localize our equations, i.e. replace multiplications of nonlocal operators applied to mappings with disjoint support (which would be zero in the case of local operators) by an operator of order zero:
\begin{lemma}[Localizing]\label{la:bs:localizing}
Let $b \in \Hf(\R^n)$. Assume there is $d,\gamma > 0$, $x \in \R^n$ such that for $E := B_{\gamma+d}(x)$, $\supp b \subset \R^n \backslash E$. Then there is a function $a \in L^2(\R^n)$ such that for $D := B_{\gamma}(x)$
\[
 \intl_{\R^n} \lapn b\ \lapn \varphi = \intl_{\R^n} a\ \varphi, \quad \mbox{for every $\varphi \in C_0^\infty(D)$}
\]
and
\[
 \Vert a \Vert_{L^2(\R^n)} \leq C_{D,E}\Vert b \Vert_{L^2(\R^n)}.
\]
\end{lemma}
\begin{proofL}{\ref{la:bs:localizing}}
We are going to show that
\begin{equation}\label{eq:bs:fstarbdd}
 \abs{f(\varphi)} := \abs{\intl_{\R^n} \lapn b\ \lapn \varphi} \leq  C_{D,E}\Vert \varphi \Vert_{L^2(\R^n)} \quad \mbox{for every $\varphi \in C_0^\infty(D)$.}
\end{equation}
Then $f(\cdot)$ is a linear and bounded operator on the dense subspace $C_0^\infty(D) \subset L^2(D)$. Hence, it is extendable to all of $L^2(D)$. Being a linear functional, by Riesz' representation theorem there exists $a \in L^2(D)$ such that $f(\varphi) = \langle a, \varphi \rangle_{L^2(D)}$ for every $\varphi \in L^2(D)$.\\
It remains to prove \eqref{eq:bs:fstarbdd}, which is done as in the proofs of \cite{DR09Sphere}. Set $r := \frac{1}{2} (\gamma+d)$, so that $E = B_{2r}(x) \supset D$. Applying Corollary \ref{co:bs:disjsupp}
\[
\begin{ma} 
\intl_{\R^n} \lapn b\ \lapn \varphi &=& \suml_{k=1}^\infty\ \intl_{\R^n} \lapn (\eta^k_{r,x} b)\ \lapn \varphi\\
&=:& \suml_{k=1}^\infty\ I_k.\\
\end{ma}
\]
If $k \geq 3$, using that the support of $\eta_r^k$ and $\varphi$ are disjoint, more precisely by Corollary \ref{co:bs:disjsupp},
\[
\begin{ma}
II_k &\overset{\sref{C}{co:bs:disjsupp}}{\aleq{}}& 2^{-2kn} \Vert \eta^k_r b \Vert_{L^1(\R^n)} \Vert \varphi \Vert_{L^1(\R^n)}\\
&\aleq{}& 2^{-\frac{3}{2}kn} \Vert \eta^k_r b \Vert_{L^2(\R^n)} \Vert \varphi \Vert_{L^1(\R^n)}\\
&\aleq{}& 2^{-\frac{3}{2}kn} \Vert b \Vert_{L^2(\R^n)} \Vert \varphi \Vert_{L^2(D)}.\\
\end{ma}
\]
For $1 \leq k \leq 3$ we use that the support of $a$ and $\varphi$ are disjoint, to get also by Corollary \ref{co:bs:disjsupp}
\[
 II_k \aleq{} d^{-{3}{2}n} \Vert b \Vert_{L^2(\R^n)} \Vert \varphi \Vert_{L^2(D)}.
\]
Consequently,
\[
 \sum_{k=1}^\infty II_k \leq C_{D,E} \Vert b \Vert_{L^2(\R^n)}\ \Vert \varphi \Vert_{L^2(D)}.
\]
\end{proofL}

\subsection{Hodge decomposition: Local estimates of s-harmonic functions}\label{ss:hodge}
If for an integrable function $h$ we have weakly $\lap h = 0$ in a, say, big ball, we can estimate
\[
\Vert h \Vert_{L^2(B_r)} \leq C \left (\frac{r}{\rho}\right )^2 \Vert h \Vert_{L^2(B_\rho)}, \quad \mbox{for $0 < r < \rho$}.
\]
The goal of this subsection is to prove in Lemma~\ref{la:estharmonic} a similar estimate, for the nonlocal operator $\lapn$.\\
\begin{proposition}\label{pr:estconvolutlplmp}
Let $s \in (0,\frac{n}{2})$. Then for any $x \in \R^n$, $r > 0$ and $v \in \Sw$, such that $\supp v \subset B_r(x)$, and any $k \in \N_0$,
\[
\Vert \abs{(\laps{s} \eta_{r,x}^k)^\wedge} \ast \abs{(\lapms{s} v)^\wedge} \Vert_{L^2(\R^n)} \leq C_s 2^{-ks} \Vert v \Vert_{L^2(\R^n)}.
\]
\end{proposition}
\begin{proofP}{\ref{pr:estconvolutlplmp}}
By convolution rule and
\[
\frac{1}{1} + \frac{1}{2} = 1 + \frac{1}{2}
\]
we have
\begin{equation}\label{eq:convol1}
\Vert \abs{(\laps{s} \eta_{r,x}^k)^\wedge} \ast \abs{(\lapms{s} v)^\wedge} \Vert_{L^2(\R^n)} \aleq{} \Vert (\laps{s} \eta_{r,x}^k)^\wedge \Vert_{L^1(\R^n)}\ \Vert (\lapms{s} v)^\wedge \Vert_{L^2(\R^n)}.
\end{equation}
By Lemma~\ref{la:lapmsest2},
\begin{equation}\label{eq:convol2}
\Vert (\lapms{s} v)^\wedge \Vert_{L^2(\R^n)} = \Vert \lapms{s} v \Vert_{L^2(\R^n)} \leq C_s r^s\Vert v \Vert_{L^2(\R^n)}.
\end{equation}
Furthermore, Proposition~\ref{pr:etarkgoodest} implies
\begin{equation}\label{eq:convol3}
\Vert (\laps{s} \eta_{r,x}^k)^\wedge \Vert_{L^1(\R^n)} \leq C_s (2^kr)^{-s}.
\end{equation}
Together, \eqref{eq:convol1}, \eqref{eq:convol2} and \eqref{eq:convol3} give the claim.
\end{proofP}

As a consequence we have
\begin{proposition}\label{pr:estlapnetlapmnvL2}
There is a uniform constant $C > 0$ such that for any $r > 0$, $x \in \R^n$, $v \in \Sw$, such that $\supp v \subset B_r(x)$, and for any $k \in \N_0$
\[
\Vert \lapn (\eta_{r,x}^k \lapmn v) \Vert_{L^2(\R^n)} \leq C\ 2^{-k \frac{1}{4}} \Vert v \Vert_{L^2(\R^n)}.
\]
\end{proposition}
\begin{proofP}{\ref{pr:estlapnetlapmnvL2}}
We have according to \eqref{eq:Hdef}
\[
\lapn (\eta_{r,x}^k \lapmn v) = (\lapn \eta_{r,x}^k) \lapmn v + \eta_{r,x}^k v + H(\eta_{r,x}^k,\lapmn v).
\]
By the support condition on $v$,
\[
\eta_{r,x}^k v = 0, \quad \mbox{if $k \geq 1$},
\]
so trivially for any $k \in \N_0$,
\[
\Vert \eta_{r,x}^k v \Vert_{L^2(\R^n)} \leq 2^{\frac{n}{2}}\ 2^{-k \frac{n}{4}} \Vert v \Vert_{L^2(\R^n)}.
\]
Next, applying Proposition~\ref{pr:etarkgoodest} for $s = \frac{n}{2}$ and $p = 4$ and Lemma~\ref{la:lapmsest2} for $s = \frac{n}{2}$ and $p'=4$, we have
\[
\Vert (\lapn \eta_{r,x}^k) \lapmn v \Vert_{L^2(\R^n)} \leq
\Vert (\lapn \eta_{r,x}^k) \Vert_{L^4}\ \Vert \lapmn v \Vert_{L^4} \aleq{}
 2^{-k\frac{n}{4}}r^{-\frac{n}{4}} r^{\frac{n}{4}}\ \Vert v \Vert_{L^2}.
\]
Thus, we have shown that
\begin{equation}\label{eq:hetalpmvleft}
\Vert \lapn (\eta_{r,x}^k \lapmn v) \Vert_{L^2(\R^n)} \aleq{} 2^{-k\frac{n}{4}} \Vert v \Vert_{L^2(\R^n)} + \Vert H(\eta_{r,x}^k,\lapmn v) \Vert_{L^2(\R^n)}.
\end{equation}
By Lemma~\ref{la:tart:prphuvwedgeest} we have that in the case $n=1,2$
\[
\Vert H(\eta_{r,x}^k,\lapmn v) \Vert_{L^2(\R^n)} \aleq{} \Vert \abs{(\lap^{\frac{n}{8}} \eta_{r,x}^k)^\wedge}\ast \abs{(\lap^{-\frac{n}{8}} v)^\wedge} \Vert_{L^2(\R^n)},
\]
and in the case $n \geq 3$
\[
\begin{ma}
&&\Vert H(\eta_{r,x}^k,\lapmn v) \Vert_{L^2(\R^n)}\\
&\aleq{}& \Vert \abs{(\lap^{\frac{n-2}{4}} \eta_{r,x}^k)^\wedge} \ast \abs{(\lap^{\frac{2-n}{4}} v)^\wedge} \Vert_{L^2} +  \Vert \abs{(\laps{1} \eta_{r,x}^k)^\wedge}  \ast   \abs{(\lapms{1} v)^\wedge} \Vert_{L^2}.
\end{ma}
\]
That is, in order to prove the claim we need the estimate
\begin{equation}\label{eq:estconvol.hopefull}
\Vert \abs{(\laps{s} \eta_{r,x}^k)^\wedge} \ast \abs{(\lapms{s} v)^\wedge} \Vert_{L^2} \leq C_s\ 2^{-ks} \Vert v \Vert_{L^2}
\end{equation}
where $s = \frac{n}{4}$ in the case $n= 1,2$ and $s = \frac{n-2}{2}$ or $s =1$ in the case $n \geq 3$. In all three cases we have that $0 < s < \frac{n}{2}$ and Proposition~\ref{pr:estconvolutlplmp} implies \eqref{eq:estconvol.hopefull}.
Plugging these estimates into \eqref{eq:hetalpmvleft} we conclude.
\end{proofP}

\begin{lemma}[Estimate of the Harmonic Term]\label{la:estharmonic}
Let $h \in L^2(\R^n)$, such that 
\begin{equation}\label{eq:lapnhWeq0}
\intl_{\R^n} h\ \lapn \varphi = 0 \quad \mbox{for any $\varphi \in C_0^\infty(B_{\Lambda r}(x))$.}
\end{equation}
for some $\Lambda > 0$. Then, for a uniform constant $C > 0$
\[
\Vert h \Vert_{L^2(B_{r}(x))} \leq C\ \Lambda^{-\frac{1}{4}} \Vert h \Vert_{L^2(\R^n)}.
\]
\end{lemma}
\begin{proofL}{\ref{la:estharmonic}}
It suffices to prove the claim for large $\Lambda$, say $\Lambda \geq 8$. Let $k_0 \in \N$, $k_0 \geq 3$, such that $\Lambda < 2^{k_0} \leq 2\Lambda$. Approximate $h$ by functions $h_\varepsilon \in C_0^\infty(\R^n)$ such that for any $\varepsilon > 0$ the distance $\Vert h - h_\varepsilon \Vert_{L^2(\R^n)} \leq \varepsilon$ and $\Vert h_\varepsilon \Vert_{L^2(\R^n)} \leq 2 \Vert h \Vert_{L^2(\R^n)}$. By Riesz' representation theorem,
\[
\Vert h_\varepsilon \Vert_{L^2(B_r(x))} = \sup_{\ontop{v \in C_0^\infty(B_r(x))}{\Vert v \Vert_{L^2} \leq 1}} \intl h_\varepsilon v.\\
\] 
For such a $v$, note that by Proposition~\ref{la:lapmsest2}, $\lapmn v \in L^p(\R^n)$ for any $p > 2$, and thus $\eta_{r,x}^k \lapmn v \in L^2(\R^n)$ for any $k \in \N_0$. Moreover, by Proposition~\ref{pr:estlapnetlapmnvL2}
\begin{equation}\label{eq:estharm:lapnetalapmnv}
 \Vert \lapn (\eta_{r,x}^k \lapmn v) \Vert_{L^2(\R^n)} \leq C\ 2^{-\frac{k}{4}}.
\end{equation}
In order to apply \eqref{eq:lapnhWeq0}, we rewrite\footnote{Note that $\lapn h_\varepsilon \in L^p(\R^n)$ for any $p \in (1,2)$. In fact, for all large $k \in \N$ the $L^p$-Norm on Annuli $A_k = B_{2^{k+1}(0)} \backslash B_{2^k(0)}$, $\Vert \lapn h_\varepsilon \Vert_{L^p(A_k)} \leq C_{h_\varepsilon} 2^{-kn (\frac{3}{2}-\frac{1}{p}} \Vert h_\varepsilon \Vert_{L^2(\R^n)}$ as can be shown using Corollary \ref{co:bs:disjsupp}. Thus, $(\lapn h_\varepsilon) (\lapmn v) \in L^1(\R^n)$.}
\[
\begin{ma}
\intl h_\varepsilon\ v &=& \intl (\lapn h_\varepsilon) (\lapmn v)\\
&=& \suml_{k=0}^\infty \intl (\lapn h_\varepsilon)\ \eta_{r,x}^k\ \lapmn v\\
&=&\suml_{k=k_0-1}^\infty \intl h_\varepsilon\ \lapn (\eta_{r,x}^k \lapmn v) + \suml_{k=0}^{k_0-2} \intl h_\varepsilon\ \lapn (\eta_{r,x}^k \lapmn v)\\
&=:& I + II.
\end{ma}
\]
The second term $II$ goes to zero as $\varepsilon \to 0$. In fact, for $k \leq k_0-2$ we have that $\supp \eta_{r,x}^k \subset B_{\Lambda r}(x)$ and thus 
\[
\begin{ma}
\intl_{\R^n} h_\varepsilon\ \lapn (\eta_{r,x}^k \lapmn v) &\overset{\eqref{eq:lapnhWeq0}}{=}& \intl (h_\varepsilon - h)\ \lapn (\eta_{r,x}^k \lapmn v)\\
&\leq& \Vert h_\varepsilon - h\Vert_{L^2(\R^n)}\ \Vert \lapn (\eta_{r,x}^k \lapmn v) \Vert_{L^2(\R^n)}\\
&\leq& \varepsilon\ \Vert \lapn (\eta_{r,x}^k \lapmn v) \Vert_{L^2(\R^n)}\\
&\overset{\eqref{eq:estharm:lapnetalapmnv}}{\leq}& C_\Lambda\ \varepsilon.
\end{ma}
\]
For the remaining term $I$ we have, using again Proposition~\ref{pr:estlapnetlapmnvL2}, 
\[
\begin{ma}
I &=& \suml_{k=k_0-1}^\infty \intl h_\varepsilon\ \lapn (\eta_{r,x}^k \lapmn v)\\
&\leq& \suml_{k=k_0-1}^\infty \Vert \lapn (\eta_{r,x}^k \lapmn v) \Vert_{L^2(\R^n)}\ \Vert h_\varepsilon \Vert_{L^2(\R^n)}.\\
&\overset{\eqref{eq:estharm:lapnetalapmnv}}{\leq}& \Vert h_\varepsilon \Vert_{L^2(\R^n)}\ \suml_{k=k_0-1}^\infty 2^{-\frac{k}{4}}\\
&\aleq{}& \Vert h \Vert_{L^2(\R^n)}\ \suml_{k=k_0-1}^\infty 2^{-\frac{k}{4}}.
\end{ma}
\]
Because of
\[
\suml_{k=k_0-1}^\infty 2^{-\frac{k}{4}} \leq C 2^{-k_0\frac{1}{4}} \leq C \Lambda^{-\frac{1}{4}},
\]
we arrive at
\[
\intl h_\varepsilon\ v \leq C \varepsilon + C \Lambda^{-\frac{1}{4}}\ \Vert h \Vert_{L^2(\R^n)}.
\]
Consequently, for all $\varepsilon > 0$,
\[
 \Vert h \Vert_{L^2(B_r(x))} \leq (C+1) \varepsilon + C \Lambda^{-\frac{1}{4}}\ \Vert h \Vert_{L^2(\R^n)}.
\]
Letting $\varepsilon \to 0$, we conclude.
\end{proofL}

The following theorem proves Theorem~\ref{th:hodge}.
\begin{theorem}\label{th:localest}
There are uniform constants $\Lambda, C > 0$ such that the following holds: For any $x \in \R^n$ and any $r > 0$ we have for every $v \in L^2(\R^n)$, $\supp v \subset B_r(x)$ 
\[
\Vert v \Vert_{L^2(B_r(x))} \leq C \sup_{\varphi \in C_0^\infty(B_{\Lambda r}(x))} \frac{1}{\Vert \lapn \varphi \Vert_{L^2(\R^n)}}\ \intl_{\R^n} v \lapn \varphi.
\]
\end{theorem}
\begin{proofT}{\ref{th:localest}}
We have,
\[
\Vert v \Vert_{L^2(B_r(x))} = \sup_{\ontop{f \in L^2(\R^n)}{\Vert f \Vert_{L^2} \leq 1}} \intl v\ f.
\]
By Lemma~\ref{la:hodge} and Lemma~\ref{la:estharmonic}, we decompose $f = \lapn \varphi + h$, $\varphi \in \Hf(\R^n)$ and $\supp \varphi \subset B_{\Lambda r}(x)$, $\Vert h \Vert_{L^2(B_r(x))} \leq C\ \Lambda^{-\frac{1}{4}}$ for arbitrarily large $\Lambda > 0$. Thus, by the support condition on $v$,
\[
\Vert v \Vert_{L^2(B_r(x))} \leq C \sup_{\ontop{\varphi \in C_0^\infty(B_{\Lambda r}(x))}{\Vert \lapn \varphi \Vert_{L^2(\R^n)} \leq 1}} \intl v \lapn \varphi + C\Lambda^{-\frac{1}{4}}\ \Vert v \Vert_{L^2(B_r(x))}.
\]
Taking $\Lambda$ large enough, we can absorb and conclude.
\end{proofT}

%
\subsection{Products of lower order operators localize well}\label{ss:loolocwell}
The goal of this subsection are Lemma~\ref{la:lowerorderlocalest} and Lemma~\ref{la:lowerorderlocalest2}, which essentially state that terms of the form
\[
 \laps{s} a\ \lap^{\frac{n}{4}-\frac{s}{2}}b
\]
``localize alright'', if $s$ is neither of the extremal values $0$ nor $\frac{n}{2}$. 

%
\begin{proposition}[Lower Order Operators and $L^2$]\label{pr:lowerorderest}
For any $s \in (0,\frac{n}{2})$, $M_1$, $M_2$ zero multiplier operators there exists a constant $C_{M_1,M_2,s} > 0$ such that for any $u,v \in \Sw$,
\[
 \Vert M_1\lap^{\frac{2s-n}{4}} u\ M_2\lap^{-\frac{s}{2}} v \Vert_{L^2(\R^n)} \leq C_{M_1,M_2,s} \Vert u \Vert_{L^2(\R^n)}\ \Vert v \Vert_{L^2(\R^n)}.
\]
\end{proposition}
\begin{proofP}{\ref{pr:lowerorderest}}
Set $p := \frac{n}{s}$ and $q := \frac{2n}{n-2s}$. As $2 < p,q < \infty$ (using also H\"ormander's multiplier theorem, \cite{Hoermander60}),
\[
\begin{ma}
 &&\Vert M_1 \lap^{\frac{2s-n}{4}} u\ M_2 \lap^{-\frac{s}{2}} v \Vert_{L^2}\\
&\leq& \Vert M_1 \lap^{\frac{2s-n}{4}} u\Vert_{L^p} \ \Vert M_2 \lap^{-\frac{s}{2}} v \Vert_{L^q}\\
&\overset{p,q \in (1,\infty)}{\aleq{}}& \Vert \lap^{\frac{2s-n}{4}} u\Vert_{L^p} \ \Vert \lap^{-\frac{s}{2}} v \Vert_{L^q}\\
&\overset{\ontop{p,q \in [2,\infty)}{\sref{P}{pr:fourierlpest}}}{\aleq{s}}& \Vert \abs{\cdot}^{\frac{2s-n}{2}} u^\wedge\Vert_{L^{p',p}} \ \Vert \abs{\cdot}^{-s} v^\wedge \Vert_{L^{q',q}}\\
&\overset{p,q \geq 2}{\aleq{s}}& \Vert \abs{\cdot}^{\frac{2s-n}{2}} u^\wedge\Vert_{L^{p',2}} \ \Vert \abs{\cdot}^{-s}\ v^\wedge \Vert_{L^{q',2}}\\
&\overset{\sref{P}{pr:dl:lso}}{\aleq{s}}& \Vert u^\wedge\Vert_{L^{2,2}} \ \Vert v^\wedge \Vert_{L^{2,2}}\\
&=&\Vert u \Vert_{L^2} \ \Vert v \Vert_{L^{2}}.\\
\end{ma}
\]
\end{proofP}
\begin{lemma}\label{la:lowerorderlocalest}
Let $s \in (0,\frac{n}{2})$ and $M_1, M_2$ zero multiplier operators. Then there is a constant $C_{M_1,M_2,s} > 0$ such that the following holds. For any $u,v \in \Sw$ and any $\Lambda > 2$,
\[
\begin{ma}
&&\Vert M_1\laps{s}  u\  M_2 \lap^{\frac{n}{4}-\frac{s}{2}}v \Vert_{L^2(B_r(x))}\\
&\leq& C_{M_1,M_2,s}\ \left (\Vert \lapn u \Vert_{L^2(B_{2\Lambda r}(x))} + \Lambda^{-s} \suml_{k=1}^\infty 2^{-ks} \Vert \eta_{\Lambda r,x}^k \lapn u \Vert_{L^2} \right ) \Vert \lapn v \Vert_{L^2}.
\end{ma}
\]
\end{lemma}
\begin{proofL}{\ref{la:lowerorderlocalest}}
As usual
\[
 \Vert \laps{s} M_1 u\ \lap^{\frac{n}{4}-\frac{s}{2}}\ M_2 v \Vert_{L^2(B_r(x))} = \sup_{\ontop{\varphi \in C_0^\infty(B_r(x),\Cc)}{\Vert \varphi \Vert_{L^2} \leq 1}} \abs{\intl M_1 \laps{s}u \ M_2 \lap^{\frac{n}{4}-\frac{s}{2}}\ v\ \varphi}.
\]
For such a $\varphi$ we then decompose $\laps{s} u$ into the part which is close to $B_r(x)$ and the far-off part:
\[
\begin{ma}
 &&\intl M_1 \laps{s}u \ M_2 \lap^{\frac{n}{4}-\frac{s}{2}}\ v\ \varphi\\
&=& \intl M_1 \lap^{\frac{s}{2}-\frac{n}{4}} (\eta_{\Lambda r} \lapn u) \ M_2\lap^{\frac{n}{4}-\frac{s}{2}}\ v\ \varphi\\
&&+ \sum_{k=1}^\infty \intl M_1 \lap^{\frac{s}{2}-\frac{n}{4}} (\eta^k_{\Lambda r} \lapn u) \ M_2 \lapms{s}\lapn\ v\ \varphi \\
&=:& I + \suml_{k=1}^\infty II_k.
\end{ma}
\]
We first estimate the $I$ by Proposition~\ref{pr:lowerorderest}
\[
 \abs{I} \aleq{s} \Vert \eta_{\Lambda r} \lapn u \Vert_{L^2}\ \Vert \lapn v \Vert_{L^2}.
\]
In order to estimate $II_k$, observe that for any $\varphi \in C_0^\infty(B_r(x),\Cc)$, $\Vert \varphi \Vert_{L^2} \leq 1$, $s \in (0,\frac{n}{2})$, if we set $p := \frac{2n}{n+2s} \in (1,2)$ 
\begin{equation}\label{eq:lol:vplapmsnv}
\begin{ma}
&&\Vert \varphi\ M_2 \lapms{s} \lapn v \Vert_{L^1}\\
&\leq& \Vert \varphi \Vert_{L^p(\R^n)}\ \Vert M_2 \lapms{s} \lapn v \Vert_{L^{p'}(\R^n)}\\
&\aleq{}& r^s\ \Vert \lapms{s} \lapn v \Vert_{L^{p'}(\R^n)}\\
&\overset{p' \geq 2}{\aleq{}}& r^s\ \Vert \abs{\cdot}^{-s} \brac{\lapn v}^\wedge \Vert_{L^{p,p'}(\R^n)}\\
&\overset{p' \geq 2}{\aleq{}}& r^s\ \Vert \abs{\cdot}^{-s} \brac{\lapn v}^\wedge \Vert_{L^{p,2}(\R^n)}\\
&\aleq{}& r^s\ \Vert \abs{\cdot}^{-s} \Vert_{L^{\frac{n}{s},\infty}}\ \Vert \brac{\lapn v}^\wedge \Vert_{L^2}\\
&\aleq{}& r^s\ \Vert \lapn v \Vert_{L^2}.
\end{ma}
\end{equation}
Hence, as for any $k \geq 1$ we have $\dist (\supp \varphi,\supp \eta_{\Lambda r}^k) \ageq{} 2^{k}\Lambda r$,
\[
\begin{ma} 
&&\abs{\intl M_1 \lap^{\frac{s}{2}-\frac{n}{4}} (\eta^k_{\Lambda r} \lapn u) \ M_2 \lap^{\frac{n}{4}-\frac{s}{2}}\ v\ \varphi}\\
&\overset{\sref{L}{la:bs:disjsuppGen}}{\aleq{s,M}}& (2^k \Lambda r)^{-n-s+\frac{n}{2}} \Vert \eta^k_{\Lambda r} \lapn u \Vert_{L^1}\ \Vert M_2 \lap^{\frac{n}{4}-\frac{s}{2}}\ v\  \varphi\Vert_{L^1}\\
&\overset{\eqref{eq:lol:vplapmsnv}}{\aleq{s,M}}& (2^k \Lambda r)^{-\frac{n}{2}-s} \Vert \eta^k_{\Lambda r} \lapn u \Vert_{L^1}\ r^s\ \Vert \lapn v \Vert_{L^2}\\
&\aleq{}& (2^k \Lambda r)^{-s} \Vert \eta^k_{\Lambda r} \lapn u \Vert_{L^2}\ r^{s}\ \Vert \lapn\ v\Vert_{L^2}\ \\
&\aeq{}& 2^{-ks} \Lambda^{-s} \Vert \eta^k_{\Lambda r} \lapn u \Vert_{L^2}\ \Vert \lapn v\Vert_{L^2}.
\end{ma}
\]
\end{proofL}
A different version of the same effect is the following Lemma.
\begin{lemma}\label{la:lowerorderlocalest2}
Let $s \in (0,\frac{n}{2})$ and $M_1, M_2$ be zero-multiplier operators. Then there is a constant $C_{M_1,M_2,s} > 0$ such that the following holds. For any $u,v \in \Sw$ and for any $\Lambda > 2$, $r > 0$, $B_r \equiv B_r(x) \subset \R^n$,
\[
\begin{ma} 
&&\Vert M_1 \laps{s} u\ M_2\lap^{\frac{n}{4}-\frac{s}{2}}\ v \Vert_{L^2(B_r(x))}\\
&\leq& C_{M_1,M_2,s}\ \Vert \eta_{\Lambda r,x}\lapn u \Vert_{L^2}\ \Vert \eta_{\Lambda r,x}\lapn v \Vert_{L^2}\\
&&+ C_{M_1,M_2,s}\ \Lambda^{-s}\ \Vert \eta_{\Lambda r,x} \lapn v \Vert_{L^2}\ \sum_{k=1}^\infty 2^{-sk} \Vert \eta_{\Lambda r,x}^k \lapn u \Vert_{L^2}\\
&&+ C_{M_1,M_2,s}\ \Lambda^{s-\frac{n}{2}}\ \Vert \eta_{\Lambda r,x} \lapn u \Vert_{L^2}\ \sum_{l=1}^\infty 2^{(s-\frac{n}{2})l} \Vert \eta_{\Lambda r,x}^l \lapn v \Vert_{L^2}\\
&& + C_{M_1,M_2,s}\ \Lambda^{-\frac{n}{2}}\ \sum_{k,l = 1}^\infty 2^{-(ks + l (\frac{n}{2}-s))} \Vert \eta^k_{\Lambda r,x} \lapn u \Vert_{L^2} \ \Vert \eta^l_{\Lambda r,x} \lapn v \Vert_{L^2}.
\end{ma}
\]
\end{lemma} 
\begin{proofL}{\ref{la:lowerorderlocalest2}}
We have
\[
\begin{ma}
 &&M_1 \laps{s} u\ M_2\lap^{\frac{n}{4}-\frac{s}{2}}\ v \\
&=& M_1\lap^{\frac{s}{2}-\frac{n}{4}} \brac{\eta_{\Lambda r} \lapn u}\ M_2\lap^{-\frac{s}{2}}\ \brac{\eta_{\Lambda r} \lapn v} \\
&&+ \sum_{k=1}^\infty M_1 \lap^{\frac{s}{2}-\frac{n}{4}} \brac{\eta^k_{\Lambda r} \lapn u}\ M_2 \lap^{-\frac{s}{2}}\ \brac{\eta_{\Lambda r} \lapn v} \\
&&+ \sum_{l=1}^\infty M_1\lap^{\frac{s}{2}-\frac{n}{4}} \brac{\eta_{\Lambda r} \lapn u}\ M_2 \lap^{-\frac{s}{2}}\ \brac{\eta_{\Lambda r}^l \lapn v} \\
&&+ \sum_{k,l=1}^\infty M_1\lap^{\frac{s}{2}-\frac{n}{4}} \brac{\eta^k_{\Lambda r} \lapn u}\ M_2 \lap^{-\frac{s}{2}}\ \brac{\eta_{\Lambda r}^l \lapn v} \\
&=& I + \sum_{k=1}^\infty II_k + \sum_{l=1}^\infty III_k + \sum_{k,l=1}^\infty IV_{k,l}.
\end{ma}
\]
By Proposition~\ref{pr:lowerorderest},
\[
 \Vert I \Vert_{L^2} \aleq{} \Vert \eta_{\Lambda r}\lapn u \Vert_{L^2}\ \Vert \eta_{\Lambda r} \lapn v \Vert_{L^2}.
\]
As in the proof of Lemma~\ref{la:lowerorderlocalest},
\[
 \Vert II_k \Vert_{L^2(B_r)} \aleq{} 2^{-sk} \Lambda^{-s} \Vert \eta_{\Lambda r}^k \lapn u \Vert_{L^2}\ \Vert \eta_{\Lambda r} \lapn v \Vert_{L^2},
\]
and
\[
 \Vert III_l \Vert_{L^2(B_r)} \aleq{} 2^{(s-\frac{n}{2})l} \Lambda^{s-\frac{n}{2}} \Vert \eta_{\Lambda r} \lapn u \Vert_{L^2}\ \Vert \eta_{\Lambda r}^l \lapn v \Vert_{L^2}.
\]
Finally,
\[
\begin{ma}
  \Vert IV_{k,l} \Vert_{L^2(B_r)} 
&\aleq{}& \brac{2^k \Lambda r}^{-s}\ \Vert \eta^k_{\Lambda r} \lapn u \Vert_{L^2} \ \Vert \lapms{s} \brac{\eta^l_{\Lambda r} \lapn v} \Vert_{L^2(B_r)}\\
  &\aleq{}& \brac{2^k \Lambda r}^{-s}\ \brac{2^l \Lambda r}^{s-\frac{n}{2}}\ r^{\frac{n}{2}}\ \Vert \eta^k_{\Lambda r} \lapn u \Vert_{L^2} \ \Vert \eta^l_{\Lambda r} \lapn v \Vert_{L^2} \\
&\aleq{}& \Lambda^{-\frac{n}{2}}\ 2^{- (ks + l (\frac{n}{2}-s))}\ \Vert \eta^k_{\Lambda r} \lapn u \Vert_{L^2} \ \Vert \eta^l_{\Lambda r} \lapn v \Vert_{L^2}.
\end{ma}
\]
\end{proofL}

\subsection{Fractional Product Rules for Polynomials}
It is obvious, that for any constant $c \in \R$ and any $\varphi \in \Sw$, $s > 0$,
\[
 \laps{s} (c\varphi ) = c\laps{s} \varphi.
\]
In this section, we are going to extend this kind of product rule to polynomials of degree greater than zero, which in our application will be mean value polynomials as in \eqref{eq:meanvalueszero}. As we have to deal with dimensions greater than one, our mean value polynomials will be also of order greater than zero, making such product rules important.

\begin{proposition}[Product Rule for Polynomials]\label{pr:lapsmonprod2}
Let $N \in \N_0$, $s \geq N$. Then for any multiplier operator $M$ defined by
\[
(M v)^\wedge = m v^\wedge, \quad \mbox{for any $v \in \Sw$},
\]
for $m \in C^\infty(\R^n \backslash \{0\},\Cc)$ and homogeneous of order zero, there exists for every multiindex $\beta \in \brac{ \N_0}^n$, $\abs{\beta} \leq N$, a multiplier operator $M_\beta \equiv M_{\beta,s,N}$, $M_\beta = M$ if $\abs{\beta} = 0$, with multiplier $m_\beta \in C^\infty(\R^n \backslash \{0\},\Cc)$ also homogeneous of order zero such that the following holds.
Let $Q = x^\alpha$ for some multiindex $\alpha \in \brac{\N_0}^n$, $\abs{\alpha} \leq N$. Then
\begin{equation}\label{eq:lapsqvarphiclaim}
M\laps{s} (Q\varphi) = \suml_{\abs{\beta} \leq \abs{\alpha}} \partial^\beta Q\ M_\beta \laps{s-\abs{\beta}} \varphi \quad \mbox{for any $\varphi \in \Sw$}.
\end{equation}
Consequently, for any polynomial $P = \suml_{\abs{\alpha} \leq N} c_\alpha x^\alpha$,
\[
M\laps{s} (P\varphi) = \suml_{\abs{\beta} \leq N} \partial^\beta P\  M_\beta \laps{s-\abs{\beta}} \varphi \quad \mbox{for any $\varphi \in \Sw$}.
\]
\end{proposition}
\begin{proofP}{\ref{pr:lapsmonprod2}}
The claim for $P$ follows immediately from the claim about $Q$ as left- and right-hand side are linear in the space of polynomials.\\
We will prove the claim for $Q$ by induction on $N$, but first we make some preperatory observations. For an operator $M$ with multiplier $m$ as requested, for $\alpha \in (\N_0)^n$ a multiindex and $s \in \R$ set
\[
m_{\alpha,s}(\xi) := \frac{1}{\brac{2\pi \im}^{\abs{\alpha}}}\abs{\xi}^{\abs{\alpha}-s}\ \partial^\alpha \brac{\abs{\xi}^s\ m(\xi)}, \quad \xi \in \R^n \backslash \{0\}, 
\]
and let $M_{\alpha,s}$ be the according operator with $m_{\alpha,s}$ as Fourier multiplier. In a slight abuse of this notation, for multiindices with only one entry we will write
\[
 M_{k,s} \equiv M_{\alpha_k,s} \quad \mbox{for $k \in (1,\ldots,n)$}, 
\]
where $\alpha_k = (0,\ldots,0,1,0\ldots,0)$ and the $1$ is exactly at the $k$th entry of $\alpha_k$.\\
Note that $m_{\alpha,s}(\cdot)$ is homogeneous of order zero. In fact, this is true as the derivative of a function of zero homogeneity has homogeneity $-1$, a fact which itself follows from taking the limit $h \to 0$ in the following equation which is valid for any $i \in \{1,\ldots,n\}$, $\xi \in \R^n \backslash \{0\}$, $\lambda > 0$, $0 \neq h \in (-\abs{\xi},\abs{\xi})$
\[
 \frac{m(\lambda (\xi + h e_i)) - m(\lambda \xi)}{\lambda h} = \lambda^{-1} \frac{m(\xi + h e_i) - m(\xi)}{h}.
\]
Also, we have the following relation for any $s \in \R$,
\begin{equation}\label{eq:lapsmon:malphsrel}
\brac{M_{\alpha,s}}_{\beta,s-\abs{\alpha}} = M_{\alpha+\beta,s}.
\end{equation}
Observe furthermore that
\[
x_1 v(x) = -\frac{1}{2\pi \im} \brac{\partial_1 v^\wedge}^\vee(x),
\]
so for $s \geq 1$
\[
\begin{ma}
&&\brac{M\laps{s} \brac{(\cdot)_1 v}}^\wedge(\xi)\\
&=& -\frac{1}{2\pi \im}\  m(\xi)\ \abs{\xi}^s \partial_1 v^\wedge(\xi)\\
&=& -\frac{1}{2\pi \im}\  \partial_1 (M\laps{s} v)^\wedge(\xi) + 
\frac{1}{2\pi \im} \partial_1 (m(\xi) \abs{\xi}^s)\ v^\wedge(\xi)\\
&=& -\frac{1}{2\pi \im}\  \partial_1 (M\laps{s} v)^\wedge(\xi) + \brac{M_{1,s} \laps{s-1} v}^\wedge(\xi),
\end{ma}
\]
that is
\begin{equation}\label{eq:MlapsQphi:it}
M\laps{s} \brac{(\cdot)_1 v}(x) = x_1 M\laps{s} v + M_{1,s} \laps{s-1} v.
\end{equation}
So one could suspect that for $Q = x^\alpha$ for some multiindex $\alpha$, $\abs{\alpha} \leq s$,
\begin{equation}\label{eq:MlapsQ:IV}
M\laps{s} (Q\varphi) = \sum_{\abs{\beta} \leq s} \partial^\beta Q\ \frac{1}{\beta!} M_{\beta,s}\ \laps{s-\abs{\beta}} \varphi.
\end{equation}
where
\[
\beta! := \beta_1!\ldots\beta_n!.
\]
This is of course true if $Q \equiv 1$. As induction hypothesis, fix $N > 0$ and assume \eqref{eq:MlapsQ:IV} to be true for any monomial $\tilde{Q}$ of degree at most $\tilde{N} < N$ whenever $s \geq \tilde{N}$ and $M$ is an operator with the desired properties. Let then $Q$ be a monomial of degree at most $N$, and assume $s \geq N$. We decompose w.l.o.g. $Q = x_1 \tilde{Q}$ for some monomial $\tilde{Q}$ of degree at most $N-1$. Then,
\begin{equation}\label{eq:MlapsQphi:indfirststep}
M\laps{s} (Q\varphi) \overset{\eqref{eq:MlapsQphi:it}}{=} x_1 M\laps{s} \brac{\tilde{Q}\varphi} + M_{1,s} \laps{s-1} \brac{\tilde{Q}\varphi}.
\end{equation}
For a multiindex $\beta = (\beta_1,\ldots,\beta_n) \in \brac{\N_0}^n$ let us set
\[
\tau_1 (\beta) := (\beta_1+1,\beta_2,\ldots,\beta_n) \quad \mbox{and}\quad 
\tau_{-1} (\beta) := (\beta_1-1,\beta_2,\ldots,\beta_n).
\]
Observe that
\begin{equation}\label{eq:MlapsQphi:pbxq}
\partial^\beta (x_1 Q) = \beta_1 \partial^{\tau_{-1}(\beta)} Q + x_1 \partial^\beta Q.
\end{equation}
Applying now in \eqref{eq:MlapsQphi:indfirststep} the induction hypothesis \eqref{eq:MlapsQ:IV} on $M\laps{s}$ and $M_{1,s} \laps{s-1}$, we have
\[
\begin{ma}
M\laps{s}  (Q\varphi) &=& x_1  \sum_{\abs{\beta} \leq s} \partial^\beta \tilde{Q}\ \frac{1}{\beta!} M_{\beta,s}\ \laps{s-\abs{\beta}} \varphi\\
&& +  \sum_{\abs{\tilde{\beta}} \leq s-1} \partial^{\tilde{\beta}} \tilde{Q}\ \frac{1}{\tilde{\beta}!} \brac{M_{1,s}}_{\tilde{\beta},s-1}\ \laps{s-(\abs{\tilde{\beta}}+1)} \varphi\\ 

&\overset{\eqref{eq:lapsmon:malphsrel}}{=}& \sum_{\abs{\beta} \leq s} x_1 \partial^\beta \tilde{Q}\ \frac{1}{\beta!} M_{\beta,s}\ \laps{s-\abs{\beta}} \varphi\\
&& + \sum_{\abs{\tilde{\beta}} \leq s-1} \partial^{\tilde{\beta}} \tilde{Q}\ \frac{1}{\tilde{\beta}!} \brac{M_{\tau_1(\tilde{\beta}),s}}\ \laps{s-\abs{\tau_1(\tilde{\beta})}} \varphi.\\
 \end{ma}
 \]
Next, by \eqref{eq:MlapsQphi:pbxq}
 \[
 \begin{ma}
 &=& \sum_{\abs{\beta} \leq s} \partial^\beta \brac{x_1 \tilde{Q}} \ \frac{1}{\beta!} M_{\beta,s}\ \laps{s-\abs{\beta}} \varphi\\
 
 &&-\sum_{\ontop{\abs{\beta} \leq s}{\beta_1 \geq 1}} \partial^{\tau_{-1} (\beta)} \tilde{Q} \ \frac{\beta_1}{\beta!} M_{\beta,s}\ \laps{s-\abs{\beta}} \varphi\\
  
 &&+ \sum_{\abs{\tilde{\beta}} \leq s-1} \partial^{\tilde{\beta}} \tilde{Q}\ \frac{1}{\tilde{\beta}!}\ M_{\tau_1(\tilde{\beta}),s}\ \laps{s-\abs{\tau_1(\tilde{\beta})}} \varphi\\
 
 &=& \sum_{\abs{\beta} \leq s} \partial^\beta \brac{x_1 \tilde{Q}} \ \frac{1}{\beta!} M_{\beta,s}\ \laps{s-\abs{\beta}} \varphi\\
 
 &&-\sum_{\ontop{\abs{\beta} \leq s}{\beta_1 \geq 1}} \partial^{\tau_{-1} (\beta)} \tilde{Q} \ \frac{1}{\tau_{-1}(\beta)!} M_{\beta,s}\ \laps{s-\abs{\beta}} \varphi\\
  &&+ \sum_{\abs{\tilde{\beta}} \leq s-1} \partial^{\tilde{\beta}} \tilde{Q}\ \frac{1}{\tilde{\beta}!}\ M_{\tau_1(\tilde{\beta}),s}\ \laps{s-\abs{\tau_1(\tilde{\beta})}} \varphi\\
  
  &=& \sum_{\abs{\beta} \leq s} \partial^\beta \brac{x_1 \tilde{Q}} \ \frac{1}{\beta!} M_{\beta,s}\ \laps{s-\abs{\beta}} \varphi.
 \end{ma}
\]

\end{proofP}

\begin{proposition}\label{pr:pvarphiest}
There is a uniform constant $C > 0$ such that the following holds: Let $u \in \Sw$ and $P$ any polynomial of degree at most $N := \lceil \frac{n}{2} \rceil-1$. Then for any $\Lambda > 2$, $B_r(x_0) \subset \R^n$, $\varphi \in C_0^\infty(B_{r}(x_0))$, $\Vert \lapn \varphi \Vert_{L^2(\R^n)} \leq 1$,
\[
\begin{ma}
&&\Vert \lapn (P \varphi) - P \lapn \varphi \Vert_{L^2(B_{r}(x_0))}\\ 
&\leq& C\ \brac{\Vert \lapn \brac{\eta_{\Lambda r,x_0}(u-P)} \Vert_{L^2(\R^n)} + \Vert \lapn u \Vert_{L^2(B_{2\Lambda r}(x_0))}}\\
&& + C\ \Lambda^{-1} \suml_{k=1}^\infty 2^{-k} \Vert \eta_{\Lambda r,x_0}^k \lapn u \Vert_{L^2(\R^n)}.
\end{ma}
\]
\end{proposition}
\begin{proofP}{\ref{pr:pvarphiest}}
By Proposition~\ref{pr:lapsmonprod2} (where we take $M$ the identity and $s = \frac{n}{2}$) 
\[
\lapn (P \varphi) - P \lapn \varphi = \suml_{1 \leq \abs{\beta} \leq N}  \partial^\beta P\ M_\beta \lap^{\frac{n-2\abs{\beta}}{4}} \varphi.
\]
As we estimate the $L^2$-norm on $B_{r}$ and there $\eta_{\Lambda r} \equiv 1$, we will further rewrite
\[
\begin{ma}
&=& -\suml_{1 \leq \abs{\beta} \leq N}  \partial^\beta (\eta_{\Lambda r}(u-P)) M_\beta \lap^{\frac{n-2\abs{\beta}}{4}} \varphi\\
&&+ \suml_{1 \leq \abs{\beta} \leq N}  \partial^\beta u\ M_\beta \lap^{\frac{n-2\abs{\beta}}{4}} \varphi\\
&=:& \suml_{1 \leq \abs{\beta} \leq N} (I_\beta + II_\beta) \qquad \mbox{on $B_r(x_0)$}.
\end{ma}
\]
As $1 \leq \abs{\beta} \leq N < \frac{n}{2}$, we have by Lemma~\ref{la:lowerorderlocalest} for $v = \varphi$
\[
\begin{ma}
 \Vert II_\beta \Vert_{L^2(B_{r})} &\aleq{}& \Vert \lapn u \Vert_{L^2(B_{2\Lambda r})} + \Lambda^{-\abs{\beta}}\suml_{k=1}^\infty 2^{-k\abs{\beta}} \Vert \eta_{\Lambda r}^k \lapn u \Vert_{L^2}\\
&\leq&\Vert \lapn u \Vert_{L^2(B_{2\Lambda r})} + \Lambda^{-1}\suml_{k=1}^\infty 2^{-k} \Vert \eta_{\Lambda r}^k \lapn u \Vert_{L^2}.
\end{ma}
\]
We can write
\[
 I_\beta = M_\beta \lap^{\frac{2\abs{\beta}-n}{4}}\ \lapn (\eta_{\Lambda r} (v-P))\ M_{\beta} \lapms{\abs{\beta}} \lapn \varphi
\]
and by Proposition~\ref{pr:lowerorderest} applied to $\lapn (\eta_{\Lambda r} (u-P))$ and $\lapn \varphi$ for $s = \abs{\beta}$
\[
 \Vert I_\beta \Vert_{L^2(\R^n)} \aleq{} \Vert \lapn (\eta_{\Lambda r} (u-P)) \Vert_{L^2(\R^n)}.
\] 
\end{proofP}
\section{Local Estimates and Compensation Phenomena: Proof of Theorem~\ref{th:localcomp}}\label{sec:localTart}
Theorem~\ref{th:localcomp} is essentially a consequence of the following two results.

\begin{lemma}\label{la:hvphilocest}
There is a uniform constant $C > 0$ such that for any ball $B_r(x_0) \subset \R^n$, $\varphi \in C_0^\infty(B_{r}(x_0))$, $\Vert \lapn \varphi \Vert_{L^2} \leq 1$, and $\Lambda > 4$ as well as for any $v \in \Hf(\R^n)$,
\[
\Vert H(v,\varphi) \Vert_{L^2(B_{r}(x_0))}
\] 
\[
\leq 
C\ \left ([v]_{B_{4\Lambda r}(x_0),\frac{n}{4}} + \Vert \lapn v \Vert_{B_{2\Lambda r}(x_0)}+ \Lambda^{-\frac{1}{2}} \Vert \lapn v \Vert_{L^2(\R^n)} \right ).
\]
\end{lemma}
\begin{proofL}{\ref{la:hvphilocest}}
We have for almost every point in $B_r \equiv B_r(x_0)$,
\[
\begin{ma}
H(v,\varphi) &=& \lapn (v \varphi) - v \lapn \varphi - \varphi \lapn v \\
&=& \lapn (\eta_{\Lambda r} v \varphi) - \eta_{\Lambda r} v \lapn \varphi - \varphi \lapn \left (\eta_{\Lambda r} v + (1-\eta_{\Lambda r}) v\right )\\
&=&  I - II - III.
\end{ma}
\]
Then we rewrite for a polynomial $P$ of order $\lceil \frac{n}{2} \rceil-1$ which we will choose below, using again that the support of $\varphi$ lies in $B_r$, so $\varphi \eta_{\Lambda r} = \varphi$ on $\R^n$,
\[
I = \lapn (\eta_{\Lambda r} (v-P) \varphi) + \lapn \brac{P \varphi}, 
\]
\[
II = \eta_{\Lambda  r} (v-P) \lapn \varphi + P \lapn \varphi,
\]
\[
III = \varphi \lapn (\eta_{\Lambda  r} (v-P)) + \varphi \lapn (\eta_{\Lambda  r} P) +   \varphi \lapn ((1-\eta_{\Lambda  r}) v ).
\]
Thus,
\[
 I-II-III = \widetilde{I} + \widetilde{II} - \widetilde{III},
\]
where
\[
\begin{ma}
\widetilde{I} &=& H(\eta_{\Lambda r} (v-P),\varphi),\\
\widetilde{II} &=& \lapn (P \varphi) - P \lapn \varphi,\\
\widetilde{III} &=& \varphi \lapn (P+(1-\eta_{\Lambda r}) (v-P) ).\\
\end{ma}
\]
Theorem~\ref{th:integrability} implies
\[
 \Vert \widetilde{I} \Vert_{L^2(\R^n)} \aleq{} \Vert \lapn (\eta_{\Lambda r} (v-P) ) \Vert_{L^2},
\] 
Proposition~\ref{pr:pvarphiest} states for $u = v$ and $s = \frac{n}{2}$ that
\[
\begin{ma}
&& \Vert \widetilde{II} \Vert_{L^2(B_r)}\\
&\aleq{}& \Vert \lapn \eta_{\Lambda   r}(v-P) \Vert_{L^2(\R^n)} + \Vert \lapn v \Vert_{L^2(B_{2\Lambda   r})} + \Lambda^{-1} \suml_{k=1}^\infty 2^{-k} \Vert \eta_{\Lambda r}^k \lapn v \Vert_{L^2(\R^n)}\\
&\aleq{}&  \Vert \lapn \eta_{\Lambda   r}(v-P) \Vert_{L^2(\R^n)} + \Vert \lapn v \Vert_{L^2(B_{2\Lambda   r})} + \Lambda^{-1} \Vert \lapn v \Vert_{L^2(\R^n)}.\\
\end{ma}
\]
It remains to estimate $\widetilde{III}$. Choose $P$ to be the polynomial such that $v-P$ satisfies the mean value condition \eqref{eq:meanvalueszero} for $N = \lceil \frac{n}{2} \rceil - 1$ and in $B_{2\Lambda r}(x_0)$.\\
We have to estimate for $\psi \in C_0^\infty(B_r)$, $\Vert \psi \Vert_{L^2} \leq 1$,
\[
\intl \widetilde{III} \psi = \intl \psi \varphi\ \lapn (P+(1-\eta_{\Lambda r})(v-P)).
\]
Note that
\[
 P+(1-\eta_{\Lambda r})(v-P) = \eta_{\Lambda r} P + (1-\eta_{\Lambda r}v) \in \Sw(\R^n),
\]
so we can write
\[
\begin{ma}
\intl \widetilde{III} \psi &=& \intl \lapn (\psi \varphi)\ P+(1-\eta_{\Lambda r})(v-P)\\
&=& \lim_{R \to \infty} \intl \lapn (\psi \varphi)\eta_R P + \intl \lapn (\psi \varphi) (1-\eta_{\Lambda r})(v-P).
\end{ma}
\]
By Remark \ref{rem:cutoffPolbdd} we have
\[
 \intl \lapn (\psi \varphi)\eta_R P = o(1) \quad \mbox{for $R \to \infty$},
\]
so in fact we only have to estimate for any $R > 1$
\[
\begin{ma}
 &&\suml_{k=1}^\infty \intl \psi\ \varphi\ \lapn (\eta_R \eta_{\Lambda r}^k (v-P))\\
&\overset{\sref{L}{la:bs:disjsuppGen}}{\aleq{}}& \suml_{k=1}^\infty (2^k \Lambda r)^{-\frac{3}{2} n}\ \Vert \varphi \Vert_{L^2}\ \Vert \eta_{\Lambda r}^k (v-P) \Vert_{L^1}\\
&\overset{\sref{L}{la:poinc}}{\aleq{}}& \suml_{k=1}^\infty (2^k \Lambda)^{-n} r^{-\frac{n}{2}}\ \Vert \eta_{\Lambda r}^k (v-P) \Vert_{L^2}\\
&=& \Lambda^{-\frac{n}{2}}\ \suml_{k=1}^\infty 2^{-\frac{n}{2} k}\ \brac{2^k \Lambda r}^{-\frac{n}{2}}\ \Vert \eta_{\Lambda r}^k (v-P) \Vert_{L^2}\\
&\overset{\sref{P}{pr:estetarkvmp}}{\aleq{}}& \Lambda^{-\frac{n}{2}}\ \suml_{k=1}^\infty 2^{-k\frac{n}{2}}(1+k)\ \Vert \lapn v \Vert_{L^2(\R^n)}\\
&\aleq{}& \Lambda^{-\frac{n}{2}} \Vert \lapn v \Vert_{L^2(\R^n)}\\
&\leq& \Lambda^{-\frac{1}{2}} \Vert \lapn v \Vert_{L^2(\R^n)}.\\
\end{ma}
\]
In order to finish the whole proof it is then only necessary to apply Lemma~\ref{la:poincmv} which implies that
\[
 \Vert \lapn \brac{\eta_{\Lambda r} (v-P)} \Vert_{L^2} \aleq{} [v-P]_{B_{4\Lambda r},\frac{n}{2}} = [v]_{B_{4\Lambda r},\frac{n}{2}}.
\]
\end{proofL}

\begin{lemma}\label{la:hwwlocest}
For any $v \in H^{\frac{n}{2}}(\R^n)$, $\varepsilon \in (0,1)$, there exists $\Lambda > 0$, $R > 0$, $\gamma > 0$ such that for all $x_0 \in \R^n$, $r < R$
\[
\begin{ma}
 &&\Vert H(v,v) \Vert_{L^2(B_r(x_0))}\\
&\leq& \varepsilon \brac{ [v]_{B_{4\Lambda r},\frac{n}{2}} + \Vert \lapn v \Vert_{L^2(B_{4\Lambda r})} }\\
&&+ C\ \Lambda^{\frac{1}{2}} \brac{ \sum_{k=1}^\infty 2^{-\gamma k} \Vert \lapn v \Vert_{L^2(A_k)}
+ \sum_{k=-\infty}^\infty 2^{-\gamma \abs{k}} [v]_{A_k,\frac{n}{2}}.
}
\end{ma}
\]
Here we set $A_k := B_{2^{k+4}4\Lambda r} \backslash B_{2^{k-1} r}$.
\end{lemma}
\begin{proof}
Let $\delta = \varepsilon \tilde{\delta} > 0 \in (0,1)$, where $\tilde{\delta}$ is a uniform constant whose value will be chosen later. Pick $\Lambda > 10$ depending on $\delta$ and $v$ such that
\begin{equation}\label{eq:loc:Lambdalapnvsmall}
 \Lambda^{-\frac{1}{2}}\ \Vert \lapn v \Vert_{L^2(\R^n)} \leq \delta.
\end{equation}
Depending on $\delta$ and $\Lambda$ choose $R > 0$ so small such that
\begin{equation}\label{eq:loc:vbLambdasmall} 
 [v]_{B_{10\Lambda r}(x_0),\frac{n}{2}} + \Vert \lapn v \Vert_{L^2(B_{10\Lambda r}(x_0))} \leq \delta,\quad \mbox{for all $x_0 \in \R^n$, $r < R$}.
\end{equation}
We can assume that $v \in C_0^\infty(\R^n)$. In fact, by Lemma~\ref{la:Tartar07:Lemma15.10} we can approximate $v$ in $\Hf(\R^n)$ by $v_k \in C_0^\infty(\R^n)$ such that \eqref{eq:loc:Lambdalapnvsmall} and \eqref{eq:loc:vbLambdasmall} are fulfilled for any $v_k$ with $2\delta$ instead of $\delta$ which is a uniform constant \emph{only} depending on $\varepsilon$.
Here one uses that
\[
 [v_k-v]_{\R^n,\frac{n}{2}} \overset{\sref{P}{pr:equivlaps}}{=} \Vert \lapn (v_k - v) \Vert_{L^2} \xrightarrow{k \to \infty} 0.
\]
By Theorem~\ref{th:integrability} and the bilinearity of $H(\cdot,\cdot)$,
\[
 \Vert H(v_k,v_k) - H(v,v)\Vert_{L^2(\R^n)} \xrightarrow{k \to \infty} 0.
\]
So both sides of the claim for $v_k$ converge to the respective sides of the claim for $v$, whereas the constants stay the same.\\
From now on let $r \in (0,R)$ and $x_0 \in \R^n$ be arbitrarily fixed and denote $B_r \equiv B_r(x_0)$. Set $P \equiv P_\Lambda \equiv P_{B_{2\Lambda r}}(v)$ the polynomial of degree $N := \lceil \frac{n}{2} \rceil - 1$ such that the mean value condition \eqref{eq:meanvalueszero} holds on $B_{2\Lambda r}(x_0)$. We denote $\eta_{\Lambda r} \equiv \eta_{\Lambda r, x_0}$ and $\tilde{\eta}_\rho := \eta_{\rho,0}$.\\
As $P$ is not a function in $\Sw(\R^n)$, we ``approximate'' it by $P^\rho := \tilde{\eta}_{\rho} P$, $\rho > \rho_0$ where we choose $\rho_0 > 2\max\{2\Lambda r+\abs{x_0},1\}$ such that $B_{\frac{1}{2}\rho_0}(0) \supset \supp v$. Note that in particular, we only work with $\rho > 0$ such that
\[
 \tilde{\eta}_\rho \equiv 1 \quad \mbox{on $\supp \eta_{2\Lambda r,x_0} \cup \supp v$, for all $\rho > \rho_0$}.
\]
Then,
\begin{equation} \label{eq:hvvest:v1}
 v = \tilde{\eta}_\rho v = \eta_{\Lambda r} (v- P) + \tilde{\eta}_\rho(1-\eta_{\Lambda r}) (v-P) + P^\rho =: v_\Lambda + v^\rho_{-\Lambda} + P^\rho. 
\end{equation}
Observe that all three terms on the right-hand side are functions of $\Sw(\R^n)$. %
%
We have
\begin{equation}\label{eq:hvvest:v2}
 v^2 = (v_\Lambda)^2 + (v_{-\Lambda}^\rho)^2 + \left (P^\rho \right )^2 + 2 v_\Lambda\ v^\rho_{-\Lambda} + 2 \left (v_\Lambda+ v^\rho_{-\Lambda}\right )\ P^\rho.
\end{equation}
As we want to estimate $H(v,v)$ on $B_r \equiv B_r(x_0)$, we are going to rewrite $H(v,v)\varphi$ for an arbitrary $\varphi \in C_0^\infty(B_r)$, such that $\Vert \varphi \Vert_{L^2(\R^n)} \leq 1$. For any $\rho > \rho_0$ (with the goal of letting $\rho \to \infty$ in the end), we will use the following facts
\[
 \varphi P^\rho = \varphi P,\quad v_\Lambda P^\rho = v_\Lambda P, \quad \varphi v^\rho_{-\Lambda} = 0.
\]
Now we start the rewriting process:
\[
\begin{ma}
 &&H(v,v) \varphi\\
&=& \left ( \lapn \brac{v^2} - 2 v \lapn v \right ) \varphi\\

&\overset{\eqref{eq:hvvest:v2}}{=}& \Big ( \lapn (v_\Lambda)^2 + \lapn (v^\rho_{-\Lambda})^2 + \lapn \left (P^\rho \right )^2 \\
&&+ 2 \lapn \left (v_\Lambda\ v^\rho_{-\Lambda}\right ) + 2 \lapn \left (\left (v_\Lambda+ v^\rho_{-\Lambda}\right )\ P^\rho\right )\\
&& - 2 v_\Lambda \lapn v_\Lambda - 2 v_\Lambda \lapn v^\rho_{-\Lambda} - 2 v_\Lambda \lapn P^\rho\\
&& - 2 P^\rho \lapn \brac{v_\Lambda + v^\rho_{-\Lambda}} - 2 P^\rho \lapn P^\rho \Big ) \varphi\\

&=& 
H(v_\Lambda,v_\Lambda) \varphi\\
&&+ 2 \left (\lapn \left (\left (v_\Lambda+ v^\rho_{-\Lambda}\right )\ P^\rho\right ) - P\ \lapn \left (v_\Lambda + v^\rho_{-\Lambda} \right ) \right ) \varphi\\
&&+ \brac{\lapn \left (P^\rho \right )^2 }\varphi\\
&&+ \left (\lapn (v_{-\Lambda}^\rho)^2 + 2 \lapn \left (v_\Lambda\ v_{-\Lambda}^\rho\right ) - 2 v_\Lambda \lapn v_{-\Lambda}^\rho\right )\varphi\\
&& - 2\left ( P\ \lapn P^\rho + v_{\Lambda} \lapn P^\rho \right )\varphi.\\
\end{ma}
\]
Now we add and substract terms, that vanish for $\rho \to \infty$, and arrive at
\[
\begin{ma}
 &&H(v,v) \varphi\\
&=& 
H(v_\Lambda,v_\Lambda) \varphi\\
&&+\ 2\ \left (\lapn \left (\left (v_\Lambda+ v^\rho_{-\Lambda}\right )\ P\right ) - P\ \lapn \left (v_\Lambda + v^\rho_{-\Lambda} \right ) \right ) \varphi\\
&&+\brac{\lapn \brac{\brac{\tilde{\eta}_\rho}^2 P P} - P \lapn \brac{\brac{\tilde{\eta}_\rho}^2 P} }\varphi\\
&&+ \left (\lapn (v_{-\Lambda}^\rho)^2 + 2 \lapn \left (v_\Lambda\ v_{-\Lambda}^\rho\right ) - 2 v_\Lambda \lapn v_{-\Lambda}^\rho\right )\varphi\\
&& +\left (P\ \lapn \brac{\brac{\tilde{\eta}_\rho}^2P} - 2\ P\ \lapn P^\rho - 2\ v_{\Lambda} \lapn P^\rho \right ) \varphi\\
&&  +\ 2\ \lapn \brac{v^\rho_{-\Lambda}(\tilde{\eta}_\rho -1)P}\ \varphi\\
&=:& \left (I + II + III + IV + V + VI\right ) \varphi. 
\end{ma}
\]
First we treat the terms $V$ and $VI$ which will be the parts vanishing for $\rho \to \infty$. \underline{As for $V$}, we have by Remark \ref{rem:cutoffPolbdd} using also that $\rho > 1$,
\[
 \Vert \lapn P^\rho \Vert_{L^\infty(\R^n)} \leq C_{r,\Lambda,v,x_0}\ \rho^{N-\frac{n}{2}} \leq C_{r,\Lambda,v,x_0} \rho^{-\frac{1}{2}},
\]
and by an analogous method one can see that the following holds, too:
\[
\Vert \lapn \brac{\brac{\tilde{\eta}_\rho}^2P} \Vert_{L^\infty(\R^n)} \leq C_{r,\Lambda,v,x_0}\ \rho^{N-\frac{n}{2}} \leq C_{r,\Lambda,v,x_0} \rho^{-\frac{1}{2}}.
\]
Consequently,
\[
 \Vert V \Vert_{L^2(B_r)} \leq C_{r,x_0,v,\Lambda} \rho^{-\frac{1}{2}}.
\]
Next, \underline{as for $VI$}, the product rule for polynomials, Proposition~\ref{pr:lapsmonprod2} for $M = Id$, $\varphi = v_{-\Lambda}^\rho (\tilde{\eta}_\rho -1) \in \Sw(\R^n)$, implies that for some zero-multiplier operator $M_\beta$,
\[
 \lapn \brac{v^\rho_{-\Lambda}(\tilde{\eta}_\rho -1)P} = \sum_{\abs{\beta} \leq N} \partial^\beta P\ M_{\beta} \lap^{\frac{n-2\abs{\beta}}{4}} \brac{v^\rho_{-\Lambda}(\tilde{\eta}_\rho -1)}.
\]
As a consequence, using that $P$ is a polynomial with coefficients depending on $\Lambda, r, v, x_0$,
\[
 \Vert VI \Vert_{L^2(B_r)} \leq C_{v,r,x_0,\Lambda} \sum_{\abs{\beta} \leq N} \Vert M_\beta \lap^{\frac{n-2\abs{\beta}}{4}} \brac{v^\rho_{-\Lambda}(\tilde{\eta}_\rho -1)} \Vert_{L^2(B_r)}.
\]
Now we use the disjoint support lemma, Lemma~\ref{la:bs:disjsuppGen}, to estimate for some $k_0 = k_0(\rho,x_0,\Lambda) \geq 1$ tending to $\infty$ as $\rho \to \infty$,
\[
\begin{ma}
 &&\Vert M_\beta \lap^{\frac{n-2\abs{\beta}}{4}} \brac{v^\rho_{-\Lambda}(\tilde{\eta}_\rho -1)} \Vert_{L^2(B_r)}\\
&\leq& \sum_{k=k_0}^\infty  \Vert M_\beta \lap^{\frac{n-2\abs{\beta}}{4}} \brac{\eta_{\Lambda r,x_0}^k (v-P) \brac{\tilde{\eta}_\rho (1-\tilde{\eta}_\rho)}} \Vert_{L^2(B_r)}\\
&\overset{\sref{L}{la:bs:disjsuppGen}}{\leq}& C_{r,\Lambda} \sum_{k=k_0}^\infty 2^{-k(n - \abs{\beta})} \Vert \brac{\eta_{\Lambda r,x_0}^k (v-P)} \Vert_{L^2(\R^n)}\\
&\overset{\sref{P}{pr:estetarkvmp}}{\leq}& C_{r,\Lambda} \sum_{k=k_0}^\infty 2^{-k(\frac{n}{2} - N)} (1+\abs{k})\ \Vert \lapn v \Vert_{L^2(\R^n)}.
\end{ma}
\]
As $N < \frac{n}{2}$, we have proven that
\[
 \Vert V \Vert_{L^2(B_r(x_0))} + \Vert VI \Vert_{L^2(B_r(x_0))} = o(1) \quad \mbox{for $\rho \to \infty$}.
\]
Next, \underline{we treat $I$}. By Theorem~\ref{th:integrability} and Lemma~\ref{la:poincmv} we have
\[
 \Vert I \Vert_{L^2(B_r)} \aleq{} \Vert \lapn v_\Lambda \Vert_{L^2(\R^n)}^2 \aleq{} \brac{[v]_{B_{4\Lambda r},\frac{n}{2}}}^2 \overset{\eqref{eq:loc:vbLambdasmall}}{\aleq{}} \delta\ [v]_{B_{4\Lambda r},\frac{n}{2}}.
\]
\underline{As for $II$}, by Proposition~\ref{pr:lapsmonprod2}, for any $w \in \Sw(\R^n)$ 
\[
 \begin{ma} 
&&\varphi \brac{\lapn (w\ P) - P \lapn w}\\
&=& \varphi \ \sum_{1 \leq \abs{\beta} \leq N} \partial^\beta P\ M_\beta \lap^{\frac{n-2\abs{\beta}}{4}} w\\
&\overset{\supp \varphi}{=}& \varphi \sum_{1 \leq \abs{\beta} \leq N}  \brac{\partial^\beta \brac{\eta_{\Lambda r}(P-v)}\ M_\beta \lap^{\frac{n-2\abs{\beta}}{4}} w + \partial^\beta v\ M_\beta \lap^{\frac{n-2\abs{\beta}}{4}} w},\\
\end{ma}
\]
so
\[
 \Vert II \Vert_{L^2(B_r)} \leq \sum_{1 \leq \abs{\beta} \leq N} II^\beta_{1,\Lambda} + II^\beta_{2,\Lambda} + II^\beta_{1,-\Lambda} + II^\beta_{2,-\Lambda},
\]
where 
\[
\begin{ma}
 II^\beta_{1,\Lambda} &=& \Vert \partial^\beta \brac{\eta_{\Lambda r}(P-v)}\  M_\beta \lap^{\frac{n-2\abs{\beta}}{4}} v_\Lambda \Vert_{L^2(B_r)}\\
&=& \Vert \partial^\beta v_\Lambda\  M_\beta \lap^{\frac{n-2\abs{\beta}}{4}} v_\Lambda \Vert_{L^2(B_r)},\\
 II^\beta_{2,\Lambda} &=& \Vert \partial^\beta v\  M_\beta \lap^{\frac{n-2\abs{\beta}}{4}} v_\Lambda \Vert_{L^2(B_r)},\\
 II^\beta_{1,-\Lambda} &=& \Vert \partial^\beta \brac{\eta_{\Lambda r}(P-v)}\  M_\beta \lap^{\frac{n-2\abs{\beta}}{4}} v^\rho_{-\Lambda} \Vert_{L^2(B_r)}\\
 &=&\Vert \partial^\beta v_\Lambda\  M_\beta \lap^{\frac{n-2\abs{\beta}}{4}} v^\rho_{-\Lambda} \Vert_{L^2(B_r)},\\
II^\beta_{2,-\Lambda} &=& \Vert \partial^\beta v\  M_\beta \lap^{\frac{n-2\abs{\beta}}{4}} v^\rho_{-\Lambda} \Vert_{L^2(B_r)}.
\end{ma}
\]
Observe that all the operators involved are of order strictly between $(0,\frac{n}{2})$. Consequently, by Proposition~\ref{pr:lowerorderest} and Poincar\'e's inequality, Lemma~\ref{la:poincmv},
\[
\begin{ma}
 II^\beta_{1,\Lambda} &\aleq{}& \Vert \lapn \brac{\eta_{\Lambda r}(P-v)} \Vert_{L^2(\R^n)}\ \Vert \lapn v_\Lambda \Vert_{L^2(\R^n)}\\
&\aleq{}& \brac{[v]_{B_{4\Lambda r},\frac{n}{2}}}^2\\
&\overset{\eqref{eq:loc:vbLambdasmall}}{\aleq{}}& \delta\ [v]_{B_{4\Lambda r},\frac{n}{2}}.
\end{ma}
\]
By Lemma~\ref{la:lowerorderlocalest} and Poincar\'e's inequality, Lemma~\ref{la:poincmv},
\[
\begin{ma}
 II^\beta_{2,\Lambda} &\aleq{}& \Vert \lapn v_\Lambda \Vert_{L^2} \brac{\Vert \lapn v \Vert_{L^2(B_{2\Lambda r})} + \Lambda^{\frac{n}{2}-\abs{\beta}}\sum_{k=1}^\infty 2^{-\frac{k}{2}} \Vert \eta_{\Lambda r}^{k} \lapn v \Vert_{L^2}}\\
&\aleq{}& [v]_{B_{4\Lambda r},\frac{n}{2}}\ \brac{\Vert \lapn v \Vert_{L^2(B_{4\Lambda r})} + \Lambda^{-\frac{1}{2}} \Vert \lapn v \Vert_{L^2}}\\
&\overset{\ontop{\eqref{eq:loc:vbLambdasmall}}{\eqref{eq:loc:Lambdalapnvsmall}}}{\aleq{}}& \delta\ \brac{\Vert \lapn v \Vert_{L^2(B_{4\Lambda r})} + [v]_{B_{4\Lambda r},\frac{n}{2}}}.
\end{ma}
\]
As for $II^\beta_{2,-\Lambda}$ and $II^\beta_{1,-\Lambda}$, we estimate for any $w \in \Sw(\R^n)$, 
\[
\begin{ma}
&&\Vert \partial^\beta w\  M_\beta \lap^{\frac{n-2\abs{\beta}}{4}} v^\rho_{-\Lambda} \Vert_{L^2(B_r)}\\
&\aleq{}&
\sum_{k=1}^\infty \Vert \partial^\beta\lapmn \brac{\eta_{4 r} \lapn w}\  M_\beta \lap^{\frac{n-2\abs{\beta}}{4}} \eta_{\Lambda r}^k (v-P) \tilde{\eta}_\rho \Vert_{L^2(B_r)}\\
&&+ \sum_{l,k=1}^\infty \Vert \partial^\beta\lapmn \brac{\eta_{4 r}^l \lapn w}\  M_\beta \lap^{\frac{n-2\abs{\beta}}{4}} \eta_{\Lambda r}^k (v-P) \tilde{\eta}_\rho \Vert_{L^2(B_r)}\\
&=:& \Sigma_1+ \Sigma_2.
\end{ma}
\]
We first concentrate on $\Sigma_1$. As before, by Lemma~\ref{la:bs:disjsuppGen} and using that $1 \leq \abs{\beta} < \frac{n}{2}$,
\[
\begin{ma}
 &&\Vert \partial^\beta\lapmn \brac{\eta_{4 r} \lapn w}\  M_\beta \lap^{\frac{n-2\abs{\beta}}{4}} \eta_{\Lambda r}^k (v-P) \tilde{\eta}_\rho \Vert_{L^2(B_r)}\\ 
&\aleq{}& \brac{2^k \Lambda r}^{-\frac{3}{2}n+\abs{\beta}} \Vert \partial^\beta\lapmn \brac{\eta_{4 r} \lapn w}\Vert_{L^2}\ \Vert \eta_{\Lambda r}^k (v-P) \Vert_{L^1}\\
&\overset{\sref{L}{la:lapmsest2}}{\aleq{}}& \brac{2^k \Lambda  r}^{-n+\abs{\beta}}  (4 r)^{\frac{n}{2}-\abs{\beta}} \Vert \eta_{4 r} \lapn w\Vert_{L^2}\ \Vert \eta_{\Lambda r}^k (v-P) \Vert_{L^2}\\
&\aeq{}& \Lambda^{\abs{\beta}-\frac{n}{2}}\ \Vert \eta_{4 r} \lapn w \Vert_{L^2}\ 2^{(\abs{\beta}-n)k}\ (\Lambda r)^{-\frac{n}{2}}  \Vert \eta_{\Lambda r}^k (v-P) \Vert_{L^2}.
\end{ma}
\]
Thus, by Proposition~\ref{pr:estetarkvmp} and as $\abs{\beta} < \frac{n}{2}$ (making $\sum_{k > 0} k\  2^{-k(\frac{n}{2}-\abs{\beta})}$ convergent),
\[
\begin{ma} 
\Sigma_1&\aleq{}& \Lambda^{\abs{\beta}-\frac{n}{2}} \Vert \eta_{4 r} \lapn w\Vert_{L^2}\ \Vert \lapn v \Vert_{L^2}\\
&\aleq{}& \Lambda^{-\frac{1}{2}} \Vert \lapn w \Vert_{L^2(B_{4\Lambda r})}\ \Vert \lapn v \Vert_{L^2(\R^n)}\\
&\overset{\eqref{eq:loc:Lambdalapnvsmall}}{\aleq{}}& \delta\ \Vert \lapn w \Vert_{L^2(B_{4\Lambda r})}.
\end{ma}
\]
For the estimate of $\Sigma_2$ we observe
\[
\begin{ma} 
&&\Vert \partial^\beta\lapmn \brac{\eta_{4 r}^l \lapn w}\  M_\beta \lap^{\frac{n-2\abs{\beta}}{4}} \eta_{\Lambda r}^k (v-P) \tilde{\eta}_\rho \Vert_{L^2(B_r)}\\
&\overset{\sref{L}{la:bs:disjsuppGen}}{\aleq{}}& (2^l r)^{-\frac{n}{2}-\abs{\beta}}\ \Vert \brac{\eta_{4 r}^l \lapn w} \Vert_{L^1}\ \Vert M_\beta \lap^{\frac{n-2\abs{\beta}}{4}} \eta_{\Lambda r}^k (v-P) \tilde{\eta}_\rho\Vert_{L^2(B_r)}\\
&\overset{\sref{L}{la:bs:disjsuppGen}}{\aleq{}}& (2^l  r)^{-\frac{n}{2}-\abs{\beta}}\  \Vert \brac{\eta_{4 r}^l \lapn w} \Vert_{L^1}\ \brac{2^k \Lambda r}^{-\frac{3}{2}n+\abs{\beta}} \Vert \eta_{\Lambda r}^k (v-P) \Vert_{L^1}\ r^{\frac{n}{2}}\\
&\aleq{}& r^{-\frac{n}{2}}\ 2^{-\abs{\beta}l}\ \Vert \brac{\eta_{4 r}^l \lapn w} \Vert_{L^2}\ \brac{2^k \Lambda}^{-n+\abs{\beta}} \Vert \eta_{\Lambda r}^k (v-P) \Vert_{L^2}.
\end{ma}
\]
Summing first over $k$ and then over $l$, using again Proposition~\ref{pr:estetarkvmp} and that $\abs{\beta} \in [1,N]$,
\[
\begin{ma} 
\Sigma_2 &\aleq{}& \Lambda^{-\frac{n}{2}+N}\ \sum_{l=1}^\infty 2^{-l} \Vert \eta_{4 r}^l \lapn w \Vert_{L^2}\ \Vert \lapn v \Vert_{L^2}\\
&\overset{\eqref{eq:loc:Lambdalapnvsmall}}{\aleq{}}& \delta\ \sum_{l=1}^\infty 2^{-l} \Vert \eta_{4 r}^l \lapn w \Vert_{L^2}. 
\end{ma}
\]
So we have shown that
\[
\begin{ma}
&&\Vert \partial^\beta w\  M_\beta \lap^{\frac{n-2\abs{\beta}}{4}} v_{-\Lambda}^\rho \Vert_{L^2(B_r)}\\
&\aleq{}&\delta\ \sum_{l=1}^\infty 2^{-l} \Vert \eta_{4 r}^l \lapn w \Vert_{L^2} + \delta \Vert \lapn w \Vert_{L^2(B_{4\Lambda r})}\\
&\aleq{}&\delta \Vert \lapn w \Vert_{L^2(\R^n)}.
\end{ma}
\]
Setting $w = v$ in the case of $II^\beta_{2,-\Lambda}$ and $w = v_\Lambda$ in the case of $II^\beta_{1,-\Lambda}$, this implies
\[
II^\beta_{1,-\Lambda} \aleq{} \delta \Vert \lapn v_\Lambda \Vert_{L^2} \aleq{} \delta\ [v]_{B_{4\Lambda r},\frac{n}{2}},
\]
and
\[
II^\beta_{2,-\Lambda} \aleq{} \sum_{l=1}^\infty 2^{-l} \Vert \lapn v \Vert_{L^2(A_l)} + \delta \Vert \lapn v \Vert_{L^2(B_{4\Lambda r)}}.
\]
\underline{As for $III$}, using yet again \eqref{eq:hvvest:v1}, we have
\[
 P_\rho \tilde{\eta}_\rho = v - v_{\Lambda} - v_{-\Lambda}^\rho \tilde{\eta}_\rho.
\]
As a consequence, we can rewrite
\[
\begin{ma}
 III &=& \brac{\lapn \brac{\brac{\tilde{\eta}_\rho}^2 P P} - P \lapn \brac{\brac{\tilde{\eta}_\rho}^2 P} }\varphi\\
&=& \brac{\lapn \brac{\brac{v - v_{\Lambda} - v_{-\Lambda}^\rho \tilde{\eta}_\rho}P} - P \lapn \brac{v - v_{\Lambda} - v_{-\Lambda}^\rho \tilde{\eta}_\rho}}\varphi.
\end{ma}
\]
Thus, the only part we have not estimated already in $II$ (or which is estimated exactly as in $II$, as the term containing $v^\rho_{-\Lambda} \tilde{\eta}_\rho$) is
\[
 \lapn \brac{v P} - P\lapn v.
\]
Again by Proposition~\ref{pr:lapsmonprod2}, this is decomposed into terms of the following form (for $1 \leq \abs{\beta} \leq N$)
\[
\begin{ma}
 &&\partial^\beta P\ M_\beta \lap^{\frac{n-2\abs{\beta}}{4}} v\\
&=& -\partial^\beta \brac{(v-P)(1-\eta_{\Lambda r})}\ M_\beta \lap^{\frac{n-2\abs{\beta}}{4}} v\\
&&- \partial^\beta \brac{(v-P)\eta_{\Lambda r}}\ M_\beta \lap^{\frac{n-2\abs{\beta}}{4}} v\\
&&+ \partial^\beta v\ M_\beta \lap^{\frac{n-2\abs{\beta}}{4}} v\\
&=:& III_1 + III_2 + III_3.
\end{ma}
\]
Of course,
\[
 \Vert III_1 \Vert_{L^2(B_r)} = 0.
\]
By Lemma~\ref{la:lowerorderlocalest},
\[
\begin{ma}
&&\Vert III_2 \Vert_{L^2(B_r)}\\
&\aleq{}& \Vert \lapn (v-P)\eta_{\Lambda r} \Vert_{L^2} \brac{ \Vert \lapn v \Vert_{L^2(B_{2\Lambda r})} + \Lambda^{-\frac{1}{2}} \sum_{k=1}^\infty 2^{-\frac{k}{2}} \Vert \lapn v \Vert_{L^2(A_k)}}\\
&\overset{\sref{L}{la:poincmv}}{\aleq{}}& [v]_{\frac{n}{2},4\Lambda r} \brac{ \Vert \lapn v \Vert_{L^2(B_{2\Lambda r})} + \sum_{k=1}^\infty 2^{-\frac{k}{2}} \Vert \lapn v \Vert_{L^2(A_k)}}\\
&\overset{\eqref{eq:loc:vbLambdasmall}}{\aleq{}}& \delta [v]_{\frac{n}{2},4\Lambda r} + \delta \sum_{k=1}^\infty 2^{-\frac{k}{2}} \Vert \lapn v \Vert_{L^2(A_k)}.
\end{ma}
\]
And by Lemma~\ref{la:lowerorderlocalest2} and \eqref{eq:loc:vbLambdasmall},
\[
 \Vert III_3 \Vert_{L^2(B_r)} \aleq{} \delta \Vert \lapn v \Vert_{L^2(B_{4\Lambda r})} + \sum_{k=1}^\infty 2^{-\frac{k}{2}} \Vert \lapn v \Vert_{L^2(A_k)}.
\]
Finally, \underline{we have to estimate $IV$}. Set
\[
 \tilde{A}_k := B_{2^{k+4}\Lambda r} \backslash B_{2^{k-4}\Lambda r}.
\]
Using Lemma~\ref{la:bs:disjsuppGen} the first term is done as follows (setting $P_k$ to be the polynomial of order $N$ where $v-P_k$ satisfies \eqref{eq:meanvalueszero} on $B_{2^{k+1} \Lambda r} \backslash B_{2^{k-1} \Lambda r}$)
\[
 \begin{ma} 
&&\Vert \lapn \brac{\eta_{\Lambda r}^k (1-\eta_{\Lambda r})\brac{\tilde{\eta}_{\rho}}^2 (v-P)^2} \Vert_{L^2(B_r)}\\
&\aleq{}& 2^{-k\frac{3}{2}n} \Lambda^{-\frac{3}{2}n} r^{-n}  \Vert \sqrt{\eta_{\Lambda r}^k} (v-P) \Vert_{L^2}^2\\
&\aleq{}& 2^{-k\frac{3}{2}n} \Lambda^{-\frac{3}{2}n} r^{-n}  \brac{\Vert \sqrt{\eta_{\Lambda r}^k} (v-P_k) \Vert_{L^2}^2 + 2^{nk} (\Lambda r)^n \Vert \sqrt{\eta_{\Lambda r}^k} (P-P_k) \Vert_{L^\infty}^2}\\
&\overset{\sref{L}{la:poincmvAn}}{\aleq{}}& 2^{-k\frac{3}{2}n} \Lambda^{-\frac{3}{2}n} r^{-n}  \brac{(2^k\Lambda r)^n \brac{[v]_{\tilde{A}_k,\frac{n}{2}}}^2  + 2^{nk} (\Lambda r)^n \Vert \sqrt{\eta_{\Lambda r}^k} (P-P_k) \Vert_{L^\infty}^2}\\  
&\overset{\sref{P}{pr:etarkpbmpkest}}{\aleq{}}& \Lambda^{-\frac{n}{2}}\ 2^{-k\frac{n}{2}}\ \brac{\brac{[v]_{\tilde{A}_k,\frac{n}{2}}}^2  + k\Vert \sqrt{\eta_{\Lambda r}^k} (P-P_k) \Vert_{L^\infty}\ \Vert \lapn v \Vert_{L^2}}\\  
&\aleq{}& \Lambda^{-\frac{n}{2}}\ 2^{-k\frac{n-\frac{1}{4}}{2}}\ \brac{\brac{[v]_{\tilde{A}_k,\frac{n}{2}}}^2  + \Vert \sqrt{\eta_{\Lambda r}^k} (P-P_k) \Vert_{L^\infty} \Vert \lapn v \Vert_{L^2}}.  
 \end{ma}
\]
Note that as $\frac{n}{2}-\frac{1}{8} > \lceil \frac{n}{2} \rceil -1$, on the one hand Lemma~\ref{la:mvestbrakShrpr} is applicable and on the other hand we have by Proposition~\ref{pr:equivlaps}
\[
 \sum_{k=1}^\infty 2^{-k \frac{n-\frac{1}{4}}{2}} \brac{[v]_{\tilde{A}_k,\frac{n}{2}}}^2 \aleq{} \Vert \lapn v \Vert_{L^2(\R^n)} \sum_{k=1}^\infty 2^{-k \frac{n-\frac{1}{4}}{2}} [v]_{\tilde{A}_k,\frac{n}{2}}.
\]
Consequently, we have for some $\gamma > 0$
\[
\begin{ma}
 \Vert \lapn (v^\rho_{-\Lambda})^2 \Vert_{L^2(B_r)} &\aleq{}& \brac{1+\Vert \lapn v \Vert_{L^2}}\ \sum_{k=-\infty}^\infty 2^{-\gamma \abs{k}} [v]_{\tilde{A}_k,\frac{n}{2}}\\
 &\overset{\eqref{eq:loc:Lambdalapnvsmall}}{\aleq{}}& \Lambda^{\frac{1}{2}}\ \sum_{k=-\infty}^\infty 2^{-\gamma \abs{k}} [v]_{\tilde{A}_k,\frac{n}{2}}. 
\end{ma}
\]
For the next term in $IV$, using the disjoint support as well as Poincar\'{e}'s inequality, Lemma~\ref{la:poinc} and Lemma~\ref{la:poincmv}, and the estimate on mean value polynomials, Proposition~\ref{pr:estetarkvmp}, and as 
\[
v_\Lambda v_{-\Lambda}^\rho = \sum_{k=1}^3 v_\Lambda \brac{\eta_{\Lambda r}^k \tilde{\eta}_\rho\ (v-P)},
\]
we can estimate
\[
\begin{ma}
&&\Vert \lapn \brac{v_\Lambda\ v_{-\Lambda}^\rho }\Vert_{L^2(B_r)}\\
&\overset{\sref{L}{la:bs:disjsuppGen}}{\leq}&
\sum_{k=1}^3\ \brac{2^k \Lambda r}^{-\frac{3}{2} n} \Vert v_\Lambda \Vert_{L^2}\ \Vert \eta_{\Lambda r}^k (v-P)\Vert_{L^2}\ r^{\frac{n}{2}}\\ 
&\overset{\sref{L}{la:poinc}}{\aleq{}}&
\sum_{k=1}^3\ \brac{2^k \Lambda r}^{-\frac{3}{2} n}\ \brac{\Lambda r}^{\frac{n}{2}} \Vert \lapn v_\Lambda \Vert_{L^2}\ \Vert \eta_{\Lambda r}^k (v-P)\Vert_{L^2}\ r^{\frac{n}{2}}\\ 
&\overset{\ontop{\sref{L}{la:poincmv}}{\sref{P}{pr:estetarkvmp}}}{\aleq{}}&\Lambda^{-\frac{n}{2}} [v]_{B_{4\Lambda r},\frac{n}{2}}\ \Vert \lapn v \Vert_{L^2(\R^n)}\\
&\overset{\eqref{eq:loc:Lambdalapnvsmall}}{\aleq{}}& \delta\ [v]_{B_{4\Lambda r},\frac{n}{2}}. 
\end{ma}
\]
Last but not least, 
\[
\begin{ma}
 &&\Vert v_\Lambda \lapn \eta_{\Lambda r}^k (v-P)\tilde{\eta}_\rho \Vert_{L^2(B_r)}\\
 &\overset{\sref{L}{la:bs:disjsuppGen}}{\aleq{}}& (2^k \Lambda r)^{-n} \Vert v_\Lambda \Vert_{L^2}\ \Vert \eta_{\Lambda r}^k (v-P) \Vert_{L^2}\\
 &\overset{\ontop{\sref{L}{la:poinc}}{\sref{L}{la:poincmv}}}{\aleq{}}& 2^{-nk} \brac{\Lambda r}^{-\frac{n}{2}} [v]_{B_{4\Lambda r},\frac{n}{2}}\ \Vert \eta_{\Lambda r}^k (v-P) \Vert_{L^2}\\
 &\overset{\eqref{eq:loc:vbLambdasmall}}{\aleq{}}& 2^{-k\frac{n}{2}} \delta \brac{ \brac{2^{k} \Lambda r}^{-\frac{n}{2}} \Vert \eta_{\Lambda r}^k (v-P_k) \Vert_{L^2} +  \Vert \eta_{\Lambda r}^k (P - P_k) \Vert_{L^\infty}}\\
 &\overset{\sref{L}{la:poincmvAn}}{\aleq{}}& \delta \brac{2^{-\frac{n}{2}k}\ [v]_{A_k,\frac{n}{2}} + 2^{-\frac{n}{2}k} \Vert \eta_{\Lambda r}^k \brac{ P-P_k} \Vert_{L^\infty} }.
 \end{ma}
\]
Again, as $\frac{n}{2} > N$, Lemma~\ref{la:mvestbrakShrpr} implies that for some $\gamma > 0$.
\[
 \Vert v_\Lambda \lapn v_{-\Lambda} \Vert_{L^2(B_r)} \aleq{} \sum_{k=-\infty}^\infty 2^{-\gamma\abs{k}} [v]_{A_k,\frac{n}{2}}.
\]
We conclude by taking $\delta = \tilde{\delta} \varepsilon$ for a uniformly small $\tilde{\delta} > 0$ which does \emph{not} depend on $\Lambda$ or $\Vert \lapn v \Vert_{L^2}$.
\end{proof}

\begin{small}\end{small}\section{Euler-Lagrange Equations}\label{sec:eleqn}
As in \cite{DR09Sphere} we will have two equations controlling the behavior of a critical point of $E_n$. First of all, we are going to use a different structure equation: Obviously, for any $u \in H^{\frac{n}{2}}(\R^n,\R^m)$ with $u(x) \in \S^{m-1}$ almost everywhere on a domain $D \subset \R^n$, we have for $w := \eta u$, $\eta \in C_0^\infty(D)$,
\[
 \sum_{i=1}^m w^i \lapn w^i = - \frac{1}{2} \sum_{i=1}^m H(w^i,w^i) + \lapn \eta,
\]
or in the contracted form
\begin{equation}\label{eq:structureeq}
 w \cdot \lapn w = -\frac{1}{2} H(w,w) + \lapn \eta.
\end{equation}

%
%
The Euler-Lagrange Equations are computed similar as in \cite{DR09Sphere}, \cite{Hel02}.
\begin{proposition}[Localized Euler-Lagrange Equation]\label{pr:eleq}
Let $\eta \in C_0^\infty(D)$ and $\eta \equiv 1$ on an open neighborhood of some ball $\tilde{D} \subset D$.\\
Let $u \in \Hf(\R^n,\R^m)$ be a critical point of $E_n(\cdot)$ on $D$, cf. Definition \ref{def:critpt}. Then $w := \eta u$ satisfies for every $\psi_{ij} \in C_0^\infty(\tilde{D})$, such that $\psi_{ij} = - \psi_{ji}$,
\begin{equation}\label{eq:ELeq}
 -\intl_{\R^n} w^i\ \lapn w^j\ \lapn \psi_{ij} = - \intl_{\R^n} a_{ij} \psi_{ij} + \intl_{\R^n} \lapn w^j\ H(w^i,\psi_{ij}).
\end{equation}
Here $a_{ij} \in L^2(\R^n)$ depends on the choice of $\eta$.
\end{proposition}
\begin{remark}
Note in the following proof, that this result holds also if $u \in L^\infty(\R^n)$ and $\lapn u \in L^2(\R^n)$, the setting of \cite{DR09Sphere}.
\end{remark}

\begin{proofP}{\ref{pr:eleq}}
Let $\varphi \in C_0^\infty(D,\R^m)$. Recall that in Definition \ref{def:critpt} we have set
\[
 u_t = \begin{cases}
        u + t d\pi_u [\varphi] + o(t)\quad &\mbox{in $D$,}\\
	u\quad &\mbox{in $\R^n \backslash D$}.\\
       \end{cases}
\]
Then $u_t$ belongs to $\Hf(\R^n,\R^{m})$ and $u_t \in \S^{m-1}$ a.e. in $D$. Hence, Euler-Lagrange equations of the functional $E_n$ defined in \eqref{eq:energy} look like
\[
 \intl_{\R^n} \lapn u \cdot \lapn d\pi_u [\varphi] = 0, \qquad \mbox{for any $\varphi \in C_0^\infty(D)$.}
\]
In particular, for any $v \in C_0^\infty(D)$ such that $v \in T_u \S^{m-1}$ a.e. (i.e. $d\pi_u[v] = v$ in $D$)
\begin{equation}\label{eq:el:vtangential}
 \intl_{\R^n} \lapn u \cdot \lapn v = 0.
\end{equation}
Let $\psi_{ij} \in C_0^\infty(\tilde{D},\R)$, $1 \leq i,j \leq m$, $\psi_{ij} = - \psi_{ij}$. Then $v^j := \psi_{ij} u^i \in \Hf(\R^n)$, $1 \leq j \leq m$. Moreover, $u \cdot v = 0$. As for $x \in D$ the vector $u(x) \in \R^m$ is orthogonal to the tangential space of $\S^{m-1}$ at the point $u(x)$, this implies $v \in T_u \S^{m-1}$. Consequently, \eqref{eq:el:vtangential} holds for this $v$ by approximation\footnote{In fact, approximate this $v \in \Hf(\R^n)$ by $v_k \in C_0^\infty(\R^n)$, see Lemma~\ref{la:Tartar07:Lemma15.10}. By the Interpolation theorem we have for $\eta \in C_0^\infty(D)$, $\eta \equiv 1$ on $\tilde{D}$ that $\Vert \eta v_k -v \Vert_{\Hf} = \Vert \eta (v_k -v) \Vert_{\Hf} \leq C_\eta \Vert v_k -v \Vert_{\Hf}$}.\\
Let $\eta$ be the cutoff function from above, i.e. $\eta \in C_0^\infty(D)$, $\eta \equiv 1$ on an open neighborhood of the ball $\tilde{D} \subset D$ and set $w := \eta u$.\\
Because of $\supp \psi \subset \tilde{D}$ we have that $v^j = w^i \psi_{ij}$. Thus, by \eqref{eq:el:vtangential} 
\begin{equation}\label{eq:el:cutpde}
 \intl_{\R^n} \lapn w^j\ \lapn (w^i \psi_{ij}) = \intl_{\R^n} \lapn (w^j-u^j)\ \lapn (w^i \psi_{ij}).
\end{equation}
Observe that $w^i \in L^\infty(\R^n) \cap \Hf(\R^n)$ and by choice of $\eta$ and $\tilde{D}$, there exists $d > 0$ such that $\dist \brac{\supp (w^j-u^j), \tilde{D}}> d$. Hence, Lemma~\ref{la:bs:localizing} implies that there is $\tilde{a}_{j} \in L^2(\R^n)$, $\Vert \tilde{a} \Vert_{L^2(\R^n)} \leq C_{u,D,\tilde{D},\eta}$ such that
\begin{equation}\label{eq:el:localizinga}
 \intl_{\R^n} \lapn (w^j-u^j)\ \lapn \varphi = \intl_{\R^n} \tilde{a}_{j} \varphi \quad \mbox{for all $\varphi \in C_0^\infty(\tilde{D})$}.
\end{equation}
Consequently, for $a_{ij} := \tilde{a}_j w^i \in L^2(\R^n)$, (again by approximation)
\[
 \intl_{\R^n} \lapn (w^j-u^j)\ \lapn (w^i \varphi) = \intl_{\R^n} a_{ij} \varphi \quad \mbox{for all $\varphi \in C_0^\infty(\tilde{D})$}.
\]
Hence, \eqref{eq:el:cutpde} can be written as
\begin{equation}\label{eq:el:cutpde2}
 \intl_{\R^n} \lapn w^j\ \lapn (w^i \psi_{ij}) = \intl_{\R^n} a_{ij} \psi_{ij},
\end{equation}
which is valid for every $\psi_{ij} \in C_0^\infty(\tilde{D})$ such that $\psi_{ij} = -\psi_{ji}$.\\
Moving on, we have just by the definition of $H(\cdot,\cdot)$,
\begin{equation}\label{eq:el:prdrle} 
 \lapn (w^i \psi_{ij}) = \lapn w^i\ \psi_{ij} + w^i\ \lapn \psi_{ij} + H(w^i,\psi_{ij}).
\end{equation}
Hence, putting \eqref{eq:el:cutpde2} and \eqref{eq:el:prdrle} together
\[
\begin{ma}
&&-\intl_{\R^n} w^i\ \lapn w^j\ \lapn \psi_{ij}\\
 &=& - \intl_{\R^n} a_{ij} \psi_{ij} + \intl_{\R^n} \lapn w^j \ \lapn w^i\ \psi_{ij} +  \intl_{\R^n} \lapn w^j\ H(w^i,\psi_{ij})\\
&\overset{\psi_{ij} = -\psi_{ji}} {=}& - \intl_{\R^n} a_{ij} \psi_{ij} + \intl_{\R^n} \lapn w^j\ H(w^i,\psi_{ij}).
\end{ma}
\]
\end{proofP}

\begin{remark}
The only change one has to do here, if $u \not \in L^2(\R^n)$ but e.g. $u \in L^\infty(\R^n)$ is to prove \eqref{eq:el:localizinga} by an alternative for Lemma~\ref{la:bs:localizing}. In fact, if we assume only $f = w^j-u^j \in L^\infty(\R^n)$, we can still estimate for any $\varphi \in C_0^\infty(\tilde{D})$ and suitably chosen $\eta_{r,x_0}^k$
\[
\begin{ma}
 &&\intl f\ \lap^{\frac{n}{2}} \varphi\\
&\overset{\sref{L}{la:bs:disjsuppGen}}{\aleq{}}& \sum_{k=1}^\infty (2^k r)^{-2n} \Vert \eta_{r,x_0}^k f \Vert_{L^1}\ \Vert \varphi \Vert_{L^1}\\
&\aleq{}& \abs{\tilde{D}}^{\frac{1}{2}}\ \Vert \varphi \Vert_{L^2}\ \Vert f \Vert_{L^\infty} \sum_{k=1}^\infty (2^k r)^{-n}\\
&\aeq{}&  C_{D,\tilde{D}}\ \Vert \varphi \Vert_{L^2}\ \Vert f \Vert_{L^\infty}.
\end{ma}
\]
Thus, in exactly the same way as in the proof of Lemma~\ref{la:bs:localizing} we conclude the existence of $\tilde{a}$ as in \eqref{eq:el:localizinga}.
\end{remark}

\section{Homogeneous Norm for the Fractional Sobolev Space}\label{sec:homognormhn2}
We recall from Section~\ref{ss:idlaps} the definition of the ``homogeneous norm'' $[u]_{D,s}$: If $s \geq 0$, $s \not \in \N_0$,
\[
\left ([u]_{D,s} \right )^2 :=  \intl_{D} \intl_{D} \frac{\abs{\nabla^{\lfloor s \rfloor}u(z_1) - \nabla^{\lfloor s \rfloor}u(z_2)}^2}{\abs{z_1-z_2}^{n+2(s-\lfloor s \rfloor)}} \ dz_1\ dz_2.
\]
Otherwise, $[u]_{D,s}$ is just $\Vert \nabla^s u \Vert_{L^2(D)}$. \\
\subsection{Comparison results for the homogeneous norm}
The goal of this section is the following lemma which compares for balls $B$ the size of $[u]_{B,\frac{n}{2}}$ to the size of $\Vert \lapn u \Vert_{L^2(B)}$. Obviously, these two semi-norms are not equivalent. In fact, take for instance any nonzero $u \in H^{\frac{n}{2}}(\R^n)$ with support outside of $B$. Then $[u]_{B,\frac{n}{2}}$ vanishes, but $\lapn u$ can not be constantly zero (cf. Lemma~\ref{la:bs:uniqueness}). Anyway, these two semi-norms can be compared in the following sense:
\begin{lemma}\label{la:comps01}
There is a uniform $\gamma > 0$ such that for any $\varepsilon >0$, $n \in \N$, there exists a constant $C_\varepsilon > 0$ such that for any $v \in \Hf(\R^n)$, $B_r \equiv B_r(x) \subset \R^n$
\[
\begin{ma}
 [v]_{B_r,\frac{n}{2}} \leq \varepsilon [v]_{B_{8r},\frac{n}{2}} &+& C_\varepsilon \Big [\Vert \lapn v \Vert_{L^2(B_{16r})}\\
&&+ \sum_{k=1}^\infty 2^{-nk}  \Vert \eta_{8r}^{k} \lapn v \Vert_{L^2}\\
&&+ \suml_{j=-\infty}^\infty 2^{-\gamma \abs{j}}\ [v]_{\tilde{A}_j,\frac{n}{2}} \Big ] \\
\end{ma}
\]
where $\tilde{A}_j = B_{2^{j+5}r} \backslash B_{2^{j-5}r}$.
\end{lemma}
\begin{proofL}{\ref{la:comps01}}
It suffices to prove this for $v \in \Sw(\R^n)$, as $\Sw(\R^n)$ is dense in $\Hf(\R^n)$. Set $N := \lceil \frac{n}{2} \rceil-1$, $s := \frac{n}{2}-N \in \{\frac{1}{2},1\}$, and let $P_{2r}$ be the polynomial of degree $N$ such that the mean value condition \eqref{eq:meanvalueszero} holds for $N$ and $B_{2r}$. Let at first $n$ be odd. Set $\tilde{v} := \eta_{2r} (v-P_{2r})$. Note that
\begin{equation}
\label{eq:comps01:veqvmp}
\tilde{v} = v - P_{2r} \quad \mbox{on $B_r$}.
\end{equation}
Consequently,
\[
\begin{ma} 
 \brac{[v]_{B_r,\frac{n}{2}} }^{2} &\overset{\eqref{eq:comps01:veqvmp}}{=}& \brac{[\tilde{v}]_{B_r,\frac{n}{2}}}^{2}\\
&\overset{s:=\frac{1}{2}}{\leq}& \sum_{\abs{\alpha} = N}\ \intl_{\R^n}\ \intl_{\R^n} \frac{ (\partial^\alpha \tilde{v}(x) - \partial^\alpha \tilde{v}(y))(\partial^\alpha \tilde{v}(x) - \partial^\alpha \tilde{v}(y) )}{\abs{x-y}^{n+2s}}\ dx\ dy\\ 
&\overset{\sref{P}{pr:eqpdeflapscpr}}{\aeq{}}& \sum_{\abs{\alpha} = N}\ \intl_{\R^n} \laps{s} \partial^\alpha \tilde{v} \ \laps{s} \partial^\alpha \tilde{v}.
\end{ma}
\]
Thus,
\[
 \brac{[v]_{B_r,\frac{n}{2}}}^2 \aleq{N} \Vert \lapn \tilde{v} \Vert_{L^2}\  \sup_{\ontop{\varphi \in C_0^\infty(B_{4r}(0))}{\Vert \lapn \varphi \Vert_{L^2} \leq 1}} \intl_{\R^n} \lapn \tilde{v}\ M\lapn \varphi,
\]
where $M$ is a zero-multiplier operator. One checks that by a similar argument this also holds for $n$ even. Using Young's inequality,
\[
\begin{ma}
 [v]_{B_r,\frac{n}{2}} &\aleq{N}& \varepsilon \Vert \lapn \tilde{v} \Vert_{L^2} + \frac{1}{\varepsilon} \sup_{\ontop{\varphi \in C_0^\infty(B_{4r})}{\Vert \lapn \varphi \Vert_{L^2} \leq 1}} \intl_{\R^n} \lapn \tilde{v}\ M\lapn \varphi\\
&\overset{\sref{L}{la:poincmv}}{\aleq{}}& \varepsilon [v]_{B_{8r},\frac{n}{2}} + \frac{1}{\varepsilon} \sup_{\ontop{\varphi \in C_0^\infty(B_{4r})}{\Vert \lapn \varphi \Vert_{L^2} \leq 1}} \intl_{\R^n} \lapn \tilde{v}\ M\lapn \varphi.\\
\end{ma}
\]
For such a $\varphi \in C_0^\infty(B_{4r})$, $\Vert \lapn \varphi \Vert_{L^2} \leq 1$ we decompose
\[
\begin{ma}
 &&\intl_{\R^n} \lapn \tilde{v}\ M\lapn \varphi\\
&\overset{\sref{P}{pr:lapspol}}{=}& \intl_{\R^n} \lapn v \ M\lapn \varphi\\
&& - \sum_{k=1}^\infty\ \intl_{\R^n} \lapn \brac{\eta_{2r}^k (v-P_{2r})} \ M\lapn \varphi\\
&=&\intl_{\R^n} \lapn v \ \eta_{8r}M\lapn \varphi\\
&& + \sum_{k=1}^\infty\ \intl_{\R^n} \lapn v \ \eta^k_{8r}M\lapn \varphi\\
&& - \sum_{k=1}^\infty\ \intl_{\R^n} \lapn \brac{\eta_{2r}^k (v-P_{2r})} \ M\lapn \varphi\\
&=:& I + \sum_{k=1}^\infty II_k - \sum_{k=1}^\infty III_k.
\end{ma}
\]
In fact, to apply Proposition~\ref{pr:lapspol} correctly, we should have used a similar argument as in the proof of Lemma~\ref{la:hwwlocest}. That is, we should have approximated $v$ by compactly supported functions, then for such functions we should have decomposed for some $\tilde{\eta}_\rho$, $\rho \geq  \rho_0$, where $\rho_0$ depends on the support of $v$, $r$, $x$ such that $B_{2\rho}(0)$ contains the support of $v$ and $\tilde{v}$,
\[
 \lapn \tilde{v} = \lapn v + \lapn (\tilde{v}-v) = \lapn v + \sum_{k = 1}^\infty \lapn \brac{\eta_{2r}^k \tilde{\eta}_\rho (v-P_{2r})} + \lapn \brac{\tilde{\eta}_\rho P}.
\]
Then one would have applied Remark \ref{rem:cutoffPolbdd} to see that $\Vert M \lapn \brac{\tilde{\eta}_\rho P} \Vert_{L^\infty}$ tends to zero as $\rho \to \infty$. But we will omit the details, and continue instead.\\
Obviously, using H\"ormander's theorem \cite{Hoermander60},
\[
 \abs{I} \aleq{} \Vert \lapn v \Vert_{L^2(B_{8r})}.
\]
Moreover, for any $k \in \N$ by Lemma~\ref{la:bs:disjsuppGen} and Poincar\'{e}'s inequality, Lemma~\ref{la:poinc},
\[
\begin{ma}
 \abs{II_k} &\aleq{}& \brac{2^kr}^{-n}\  \Vert \eta_{8r}^{k} \lapn v \Vert_{L^2}\ r^{n}\\
&=& 2^{-nk}\  \Vert \eta_{8r}^{k} \lapn v \Vert_{L^2}.
\end{ma}
\]
As for $III_k$, let for $k \in \N$, $P_{2r}^k$ the polynomial which makes $v-P_{2r}^k$ satisfy the mean value condition \eqref{eq:meanvalueszero} on $B_{2^{k+2}r} \backslash B_{2^{k}r}$. If $k \geq 3$,
\[
\begin{ma}
 \abs{III_k} &\overset{\sref{L}{la:bs:disjsuppGen}}{\aleq{}}& r^{-\frac{n}{2}}\ \brac{2^{k}}^{-\frac{3}{2}n}\ \Vert \eta_{2r}^k (v-P_{2r}) \Vert_{L^2}\\
&\aleq{}& r^{-\frac{n}{2}}\ 2^{-\frac{3}{2}nk}\ \brac{\Vert \eta_{2r}^k (v-P^k_{2r}) \Vert_{L^2} + 2^{k\frac{n}{2}} r^{\frac{n}{2}} \Vert \eta_{2r}^k (P_{2r}-P^k_{2r}) \Vert_{L^\infty}}\\
&\overset{\sref{L}{la:poincmvAn}}{\aleq{}}& \ 2^{-nk}\ \brac{ [v]_{\tilde{A}_k,\frac{n}{2}} + \Vert \eta_{2r}^k (P_{2r}-P^k_{2r}) \Vert_{L^\infty}}.
\end{ma}
\]
This and Lemma~\ref{la:mvestbrakShrpr} imply for a $\gamma > 0$,
\[
 \sum_{k=3}^\infty III_k \aleq{}  \suml_{j=-\infty}^\infty 2^{-\abs{j} \gamma}\ [v]_{\tilde{A}_j,\frac{n}{2}}.
\]
It remains to estimate $III_1$, $III_2$ (where we can not use the disjoint support lemma, Lemma~\ref{la:bs:disjsuppGen}). Let from now on $k=1$ or $k=2$. By Lemma~\ref{la:poincmvAn}
\[
 \Vert \lapn\eta_{2r}^k (v-P^k_{2r}) \Vert_{L^2} \aleq{} [v]_{\tilde{A}_k,\frac{n}{2}},
\]
so
\[
\begin{ma}
 III_k &\leq& \Vert \lapn \brac{\eta_{2r}^k (v-P^k_{2r}) }\Vert_{L^2} + \Vert \lapn \brac{\eta_{2r}^k \left (P^k_{2r} - P_{2r} \right ) }\Vert_{L^2}\\
&\aleq{}&  [v]_{\tilde{A}_k,\frac{n}{2}} + \Vert \lapn\brac{\eta_{2r}^k \left (P^k_{2r} - P_{2r} \right )} \Vert_{L^2}.
 \end{ma}
\]
The following will be similar to the calculations in the proof of Lemma~\ref{la:poincmv} and Proposition~\ref{pr:mvpoinc}. Set
\[
w_{\alpha,\beta}^k := \partial^\alpha \eta_{2r}^k\ \partial^\beta \brac{P^k_{2r} - P_{2r}}.
\]
We calculate for odd $n \in \N$,
\[
\Vert \lapn\brac{\eta_{2r}^k \brac{P^k_{2r} - P_{2r}} }\Vert_{L^2}^2 \aleq{} \sum_{\abs{\alpha} + \abs{\beta} = \frac{n-1}{2}} [w_{\alpha,\beta}^k]_{\R^n,\frac{1}{2}}^2.
\]
Note that $\supp w_{\alpha,\beta}^k \subset B_{2^{k+2}r} \backslash B_{2^k r}$, so
\[
\begin{ma}
&&[w^k_{\alpha,\beta}]_{\R^n,\frac{1}{2}}^2\\
&\aleq{}& \Vert w^k_{\alpha,\beta} \Vert^2_{L^\infty} \intl_{\tilde{A}_k} \intl_{\R^n \backslash B_{40r}} \frac{1}{\abs{x-y}^{n+1}}\ dx\ dy\\
&&+ \Vert \nabla w^k_{\alpha,\beta}\Vert^2_{L^\infty}\ \intl_{\tilde{A}_k} \intl_{B_{\frac{1}{4}r}} \frac{1}{\abs{x-y}^{n-1}}\ dx\ dy\\
&&+ \Vert \nabla w^k_{\alpha,\beta}\Vert^2_{L^\infty}\ \intl_{\tilde{A}_k} \intl_{B_{40r} \backslash B_{\frac{1}{4}r}} \frac{1}{\abs{x-y}^{n-1}}\ dx\ dy\\
&\aleq{}& \Vert w^k_{\alpha,\beta} \Vert^2_{L^\infty} r^{n-1} + r^{n+1} \Vert \nabla w^k_{\alpha,\beta} \Vert^2_{L^\infty}\\
&\aleq{}& \max_{\abs{\delta} \leq \frac{n+1}{2}} r^{2\abs{\delta}} \Vert \partial^\delta (P_{2r} - P^k_{2r})\Vert_{L^\infty(\supp \eta_{2r}^k)}^2\\
&\aeq{}& \max_{\abs{\delta} \leq N} r^{2\abs{\delta}} \Vert \partial^\delta (P_{2r} - P^k_{2r})\Vert_{L^\infty(\supp \eta_{2r}^k)}^2.
\end{ma}
\]
Taking the square root, we have shown that
\[
\sum_{k=1}^2 \Vert \lapn\brac{\eta_{2r}^k \brac{P^k_{2r} - P_{2r}} }\Vert_{L^2} \aleq{} \max_{\abs{\delta} \leq N} r^{\abs{\delta}} \sum_{k=1}^2 \Vert \partial^\beta (P_{2r} - P^k_{2r})\Vert_{L^\infty(\supp \eta_{2r}^k)}.
\]
Of course, the same holds true if $n \in \N$ is even. Now, in the proof of Lemma~\ref{la:mvestbrakShrpr}, more precisely in \eqref{eq:mvpoincSiagammaClaim}, it was shown that 
\[
\begin{ma}
&&\sum_{k=1}^2 \Vert \partial^\delta (P_{2r} - P^k_{2r}) \Vert_{L^\infty(\tilde{A}_k)}\\
&\aleq{}&\sum_{k=1}^\infty 2^{-nk}\ \Vert \partial^\delta (P_{2r} - P^k_{2r}) \Vert_{L^\infty(\tilde{A}_k)}\\
&=&\sum_{k=1}^\infty 2^{-nk}\ \Vert \partial^\delta (Q^{\abs{\delta}}_{2r} - Q^{\abs{\delta}}_{k} \Vert_{L^\infty(\tilde{A}_k)}\\
&\overset{\eqref{eq:mvpoincSiagammaClaim}}{\aleq{}}& r^{-\abs{\delta}} \sum_{j=-\infty}^\infty 2^{-(n-N) \abs{j}} [v]_{\tilde{A}_j,\frac{n}{2}}.
\end{ma}
\]
Here, of course, we have set \[Q_{2r}^{\abs{\delta}} := Q_{B_{2r},N}^{\abs{\delta}}\] and \[Q_{k}^{\abs{\delta}} := Q_{B_{2^{k+2}r} \backslash B_{2^{k}r},N}^{\abs{\delta}}.\]
This concludes the proof.
\end{proofL}

\subsection{Localization of the homogeneous Norm}
For the convenience of the reader, we will repeat the proof of the following result in \cite{DR09Sphere}.
\begin{lemma}\label{la:homogloc}(\cite[Theorem A.1]{DR09Sphere})\\
For any $s > 0$ there is a constant $C_s > 0$ such that the following holds. For any $v \in \Sw(\R^n)$, $r > 0$, $x \in \R^n$,
\[
\brac{[v]_{B_{r}(x),s}}^2 \leq C_s \sum_{k=-\infty}^{-1} \brac{[v]_{A_k,s}}^2.
\]
Here $A_k$ denotes $B_{2^{k+1}r}(x) \backslash B_{2^{k-1}r}(x)$.
\end{lemma}
\begin{proofL}{\ref{la:homogloc}}
This is obvious for any $s \in \N$. Moreover, it suffices to prove the case $s \in (0,1)$, as for $\tilde{s} > 1$,
\[
 [v]_{D,\tilde{s}} = [\nabla^{\lfloor \tilde{s} \rfloor}v]_{D,\tilde{s}-\lfloor \tilde{s} \rfloor} \quad \mbox{for any domain $D \subset \R^n$}.
\]
So let $s \in (0,1)$. Denote
\[
\tilde{A}_k := B_{2^{k+1}r}(x) \backslash B_{2^{k}r}(x),
\]
and set
\[
(v)_k := \mvintl_{A_k} v,\quad \mbox{and}\quad (v)_{\tilde{k}} := \mvintl_{\tilde{A}_k} v,
\]
as well as
\[
[v]_k := [v]_{A_k,s},\quad \mbox{and} \quad [v]_r := [v]_{B_r(x),s}.
\]
With these notations,
\[
\begin{ma}
[v]_{r} &\leq& \sum_{k,l=-\infty}^{-1}\ \intl_{\tilde{A}_k} \intl_{\tilde{A}_l} \frac{\abs{v(x)-v(y)}^2}{\abs{x-y}^{n+2s}}\ dx\ dy\\
&\leq& 3 \sum_{k=-\infty}^{-1} [v]_k^2\\
&&+2\sum_{k=-\infty}^{-1} \sum_{l=-\infty}^{k-2}\ \intl_{\tilde{A}_k} \intl_{\tilde{A}_l} \frac{\abs{v(x)-v(y)}^2}{\abs{x-y}^{n+2s}}\ dx\ dy.\\
\end{ma}
\]
For $x \in \tilde{A}_k$ and $y \in \tilde{A}_l$ and $l \leq k-2$,
\[
\begin{ma}
&&\frac{\abs{v(x)-v(y)}^2}{\abs{x-y}^{n+2s}}\\
&\aleq{s}& \brac{2^{k}r}^{-n-2s} \abs{v(x)-v(y)}^2\\
&\aleq{s}& \brac{2^{k}r}^{-n-2s} \brac{\abs{v(x)-(v)_{\tilde{k}}}^2 + \abs{v(y) - (v)_{\tilde{l}}}^2 + \abs{(v)_{\tilde{l}} - (v)_{\tilde{k}}}^2}\\
&\aleq{s}& \brac{2^{k}r}^{-n-2s} \brac{\abs{v(x)-(v)_{\tilde{k}}}^2 + \abs{v(y) - (v)_{\tilde{l}}}^2 + \abs{l-k}\sum_{i=l}^{k-1}\abs{(v)_{\tilde{i}} - (v)_{\widetilde{i+1}}}^2}.\\
&=:& I + II + III.
\end{ma}
\]
As for $I$ and $II$, we have
\[
\intl_{\tilde{A}_k} \abs{v-(v)_{\tilde{k}}}^2 \aleq{} \frac{1}{\abs{A_{\tilde{k}}}} \brac{2^k r}^{n+2s} [v]_k^2
\]
and 
\[
\intl_{\tilde{A}_l} \abs{v-(v)_{\tilde{l}}}^2 \aleq{} \frac{1}{\abs{A_{\tilde{l}}}} \brac{2^l r}^{n+2s} [v]_l^2.
\]
Consequently,
\[
\begin{ma}
&&\sum_{k=-\infty}^{-1} \sum_{l=-\infty}^{k-2} \intl_{\tilde{A}_k} \intl_{\tilde{A}_l} I\\
&\leq& \sum_{k=-\infty}^{-1} \sum_{l=-\infty}^{k-2} \frac{\abs{\tilde{A}_l}}{\abs{\tilde{A}_k}} [v]_k^2\\
&\aleq{}& \sum_{k=-\infty}^{-1} [v]_k^2 \sum_{l=-\infty}^{k-2} 2^{l-k}\\
&\aleq{}& \sum_{k=-\infty}^{-1} [v]_k^2.
\end{ma}
\]
Similarly,
\[
\begin{ma}
&&\sum_{l=-\infty}^{-1} \sum_{k=l+1}^{-1} \intl_{\tilde{A}_k} \intl_{\tilde{A}_l} II\\
&\aleq{s}& \sum_{l=-\infty}^{-1} \sum_{k=l+1}^{-1} 2^{2(l-k)s}  [v]_l^2 \\
&\aleq{s}& \sum_{l=-\infty}^{-1} [v]_l^2.\\
\end{ma}
\]
As for $III$, we have
\[
\begin{ma}
&&\abs{(v)_{\tilde{i}} - (v)_{\widetilde{i+1}}}^2\\
&\aleq{}& \brac{2^i r}^{-2n}\ 2^{i(n+2s)}r^{n+2s}\ [v]_i^2\\
&=& 2^{(-n+2s)i}\ r^{-n+2s}\ [v]_i^2.
\end{ma}
\]
This implies that we have to estimate
\[
\begin{ma}
&&\sum_{k=-\infty}^{-1} \sum_{l=-\infty}^{k-2} \sum_{i=l}^{k-1} (k-l) 2^{-k(n+2s)}r^{-n-2s} \abs{A_l} \abs{A_k} 2^{(-n+2s)i} r^{-n+2s} [v]_i^2\\
&=& \sum_{i=-\infty}^{-2} 2^{(-n+2s)i}\ [v]_i^2 \sum_{l=-\infty}^{i} \sum_{k=i+1}^{-1} (k-l)\ 2^{-2ks}\ 2^{ln}.
\end{ma}
\]
Now, for any $a \in \Z$, $q \in [0,1)$
\[
 \sum_{k=a}^\infty k q^k =          
q^a \sum_{k=0}^\infty (k+a) q^{k}
= q^a \brac{\sum_{k=0}^\infty k q^{k} + a \sum_{k=0}^\infty q^{k}}
\leq C_q\ q^a\ (a+1)
\]
Consequently for any $l \leq i$,
\[
\begin{ma}
&&\sum_{k=i+1}^0 (k-l)\ 2^{-2ks}\\
&\leq&  2^{-2ls} \sum_{k= i+1}^\infty (k-l)\ 2^{-2(k-l)s}\\
&=& 2^{-2ls} \sum_{\tilde{k}= i+1-l}^\infty \tilde{k}\ 2^{-2\tilde{k}s}\\
&\aleq{s}& 2^{-2ls} (i-l+2)\ 2^{-2s(i-l)}\\
&=& 2^{-2si}\ (i-l+2),\\
\end{ma}
\]
and
\[
\begin{ma}
&&\sum_{l=-\infty}^i 2^{ln} (i-l+2)\\
&=& 2^{ni} \sum_{l=-\infty}^i 2^{(l-i)n} (i-l+2)\\
&\leq& 2^{ni} \sum_{l=-\infty}^0 2^{ln} \brac{2-l}\\
&\aeq{}&2^{ni}.
\end{ma}
\]
Thus,
\[
\sum_{i=-\infty}^{-2} 2^{(-n+2s)i}\ [v]_i^2 \sum_{l=-\infty}^{i}\sum_{k=i+1}^{-1}  (k-l)\ 2^{-2ks}\ 2^{ln}
\aleq{} \sum_{i=-\infty}^{-2} [v]_i^2.
\]
\end{proofL}
\begin{remark}\label{rem:akegal}
By the same reasoning as in Lemma~\ref{la:homogloc}, one can also see that for two Annuli-families of different width, say $A_{k} := B_{2^{k+\lambda}r} \backslash B_{2^{k-\lambda r}}$ and $\tilde{A}_k := B_{2^{k+\Lambda}r} \backslash B_{2^{k-\Lambda r}}$ we can compare
\[
[v]_{A_k,s} \leq C_{\lambda,\Lambda,s} \sum_{l=k-N_{\lambda,\Lambda}}^{k+N_{\lambda,\Lambda}} [v]_{\tilde{A}_l,s}.
\]
In particular we don't have to be too careful about the actual choice of the width of the family $A_k$ for quantities like
\[
\sum_{k=-\infty}^\infty 2^{-\gamma \abs{k}} [v]_{A_k,s},
\]
as long as we can afford to deal with constants depending on the change of width, i.e. if we can afford to have e.g.
\[
C_{\Lambda,\lambda,\gamma,s} \sum_{l=-\infty}^\infty 2^{-\gamma \abs{l}} [v]_{\tilde{A}_l,s};
\]
In fact this is because of
\[
\begin{ma}
&&\sum_{k=-\infty}^\infty 2^{-\gamma \abs{k}} [v]_{A_k,s}\\
&\leq& \sum_{k=-2N+1}^{2N-1} [v]_{A_k,s} + \sum_{k=-\infty}^{-2N} 2^{\gamma k}\ [v]_{A_k,s} + \sum_{k=2N}^{\infty} 2^{-\gamma k}\ [v]_{A_k,s}\\

&\aleq{\Lambda,\lambda}& \sum_{k=-2N+1}^{2N-1} \sum_{l=k-N}^{k+N} [v]_{\tilde{A}_l,s} + \sum_{k=-\infty}^{-2N} \sum_{l=k-N}^{k+N} 2^{\gamma k}\ [v]_{\tilde{A}_l,s}\\
&& + \sum_{k=2N}^{\infty} \sum_{l=k-N}^{k+N} 2^{-\gamma k}\ [v]_{\tilde{A}_l,s}\\

&\aleq{\Lambda,\lambda}& 4N 2^{3\gamma N} \sum_{l=-3N}^{3N} 2^{-\gamma \abs{l}}\ [v]_{\tilde{A}_l,s} + 2^{\gamma N} \sum_{k=-\infty}^{-2N} \sum_{l=k-N}^{k+N} 2^{\gamma l}\ [v]_{\tilde{A}_l,s}\\
&& + 2^{\gamma N}\sum_{k=2N}^{\infty} \sum_{l=k-N}^{k+N}2^{-\gamma l}\ [v]_{\tilde{A}_l,s}\\
&\aleq{\Lambda,\lambda}&  \sum_{l=-3N}^{3N} 2^{-\gamma \abs{l}}\ [v]_{\tilde{A}_l,s} + 2N \sum_{l=-\infty}^{-N} 2^{\gamma l}\ [v]_{\tilde{A}_l,s} + 2N\ \sum_{l=N}^{\infty} 2^{-\gamma l}\ [v]_{\tilde{A}_l,s}\\
&\leq& C_{\Lambda,\lambda,\gamma} \sum_{l=-\infty}^\infty 2^{-\gamma \abs{l}}\ [v]_{\tilde{A}_l,s}.
\end{ma}
\]
Of course, the same argument holds for $[v]_{A_k,s}$ replaced by $\Vert \laps{s} v \Vert_{L^2(A_k)}$, too.
\end{remark}
\section{Growth Estimates: Proof of Theorem~\ref{th:regul}}\label{sec:growth}
In this section, we derive growth estimates from equations \eqref{eq:structureeq} and \eqref{eq:ELeq}, similar to the usual Dirichlet-Growth estimates.
\begin{lemma}\label{la:estStrEq}
Let $w \in \Hf(\R^n,\R^m)$, $\varepsilon > 0$. Then there exist constants $\Lambda > 0$, $R > 0$, $\gamma > 0$ such that if $w$ is a solution of \eqref{eq:structureeq}, then for any $x_0 \in \R^n$, $r \in (0,R)$
\[
\begin{ma}
 &&\Vert w \cdot \lapn w \Vert_{L^2(B_r(x_0))}\\
&\leq& \varepsilon \brac{\Vert \lapn w \Vert_{L^2(B_{4\Lambda r})} + [w]_{B_{4\Lambda r},\frac{n}{2}}}\\
&&+ C_{\Lambda, w} \brac{r^\frac{n}{2} + \sum_{k=1}^\infty 2^{-\gamma k} \Vert \lapn w \Vert_{L^2(A_k)}
+ \sum_{k=-\infty}^\infty 2^{-\gamma \abs{k}} [w]_{A_k,\frac{n}{2}}}.
\end{ma}
\]
Here, $A_k = B_{2^{k+1} r}(x_0) \backslash B_{2^{k-1} r}(x_0)$.
\end{lemma}
\begin{proofL}{\ref{la:estStrEq}}
By \eqref{eq:structureeq},
\[
\Vert w \cdot \lapn w \Vert_{L^2(B_r)} \leq \Vert H(w,w) \Vert_{L^2(B_r)} + \Vert \lapn \eta^2 \Vert_{L^2(B_r)}.
\]
As $\lapn \eta^2$ is bounded (by a similar argument as the one in the proof of Proposition~\ref{pr:etarkgoodest}),
\[
\Vert \lapn \eta^2 \Vert_{L^2(B_r)} \leq C_{\eta} r^{\frac{n}{2}}.
\]
We conclude by applying Lemma~\ref{la:hwwlocest}, using also Remark \ref{rem:akegal}.
\end{proofL}
The next lemma is a simple consequence of H\"older and Poincar\'{e} inequality, Lemma~\ref{la:poinc}.
\begin{lemma}\label{la:growth:avp}
Let $a \in L^2(\R^n)$. Then
\[
\intl_{\R^n} a\ \varphi \leq C\ r^{\frac{n}{2}}\ \Vert a \Vert_{L^2(\R^n)}\ \Vert \lapn \varphi \Vert_{L^2(\R^n)}
\]
for any $\varphi \in C_0^\infty(B_r(x_0))$, $r > 0$.
\end{lemma}

\begin{lemma}\label{est:ELeq}
For any $w \in \Hf\cap L^\infty(\R^n,\R^m)$ and any $\varepsilon > 0$ there are constants $\Lambda > 0$, $R > 0$ such that if $w$ is a solution to \eqref{eq:ELeq} for some ball $\tilde{D} \subset \R^n$ then for any $B_{\Lambda r}(x_0) \subset \tilde{D}$, $r \in (0,R)$ and any skew-symmetric $\alpha \in \R^{n\times n}$, $\abs{\alpha} \leq 2$,
\[
\Vert w^i \alpha_{ij} \lapn w^{j} \Vert_{L^2(B_r(x_0))} \leq \varepsilon \Vert \lapn w \Vert_{L^2(B_{\Lambda r}(x_0))} + C_{\varepsilon,\tilde{D},w} \brac{r^{\frac{n}{2}} + \sum_{k=1}^\infty 2^{-nk}\ \Vert \lapn w \Vert_{L^2(A_k)}}.
\]
Here, $A_k = B_{2^{k+1} r}(x_0) \backslash B_{2^{k-1} r}(x_0)$.
\end{lemma}
\begin{proofL}{\ref{est:ELeq}}
Let $\delta = C \varepsilon> 0$ for a uniform constant $C$ which will be clear later. Set $\Lambda_1 > 1$ ten times the uniform constant $\Lambda$ from Theorem~\ref{th:localest} and choose $\Lambda_2 > 10$ such that
\begin{equation}\label{eq:estEleq:Lambda2}
\brac{\Lambda_2}^{-\frac{1}{2}} \Vert \lapn w \Vert_{L^2(\R^n)} \leq \delta.
\end{equation}
We then define $\Lambda := 10 \Lambda_1 \Lambda_2$. Choose $R > 0$ such that 
\begin{equation}\label{eq:estEleq:wnormsleqdelta}
[w]_{B_{10\Lambda r},\frac{n}{2}}+\Vert \lapn w \Vert_{L^2(B_{10\Lambda r})} \leq \delta \mbox{\quad for any $x_0 \in \R^n$, $r \in (0,R)$}.
\end{equation}
Fix now any $r \in (0,R)$, $x_0 \in \R^n$ such that $B_{\Lambda r}(x_0) \subset \tilde{D}$. For the sake of brevity, we set $v := w^i \alpha_{ij} \lapn w^j$. By Theorem~\ref{th:localest}
\[
\Vert v \Vert_{L^2(B_r)} \leq \Vert \eta_r v \Vert_{L^2} \leq C \sup_{\ontop{\varphi \in C_0^\infty(B_{\Lambda_1 r}(x_0))}{\Vert \lapn \varphi \Vert_{L^2} \leq 1}} \intl \eta_r\ v\ \lapn \varphi.
\]
We have for such a $\varphi \in C_0^\infty(B_{\Lambda_1 r}(x_0))$, $\Vert \lapn \varphi \Vert_{L^2} \leq 1$, 
\[ 
\begin{ma}
\intl_{\R^n} \eta_r v\ \lapn \varphi &=& \intl v\ \lapn \varphi + \intl (\eta_r-1)\ v\ \lapn \varphi\\
&=:& I + II.
\end{ma}
\]
In order to estimate $II$, we use the compact support of $\varphi$ in $B_{{\Lambda_1} r}$ and apply Corollary \ref{co:bs:disjsupp} and Poincar\'{e}'s inequality, Lemma~\ref{la:poinc}:
\[
\begin{ma}
II &=& \intl (\eta_r-1) v\ \lapn \varphi\\
&\overset{\ontop{\sref{C}{co:bs:disjsupp}}{\sref{L}{la:poinc}}}{\leq}& C_{{\Lambda_1}} \sum_{k=1}^\infty 2^{-nk}\ \Vert \eta^k_{r} v \Vert_{L^2}\ \Vert \lapn \varphi \Vert_{L^2(\R^n)}\\
&\leq& C_{{\Lambda_1}} \sum_{k=1}^\infty 2^{-nk}\ \Vert \eta^k_{r} v \Vert_{L^2}\\
&\leq& C_{{\Lambda_1}} \Vert w \Vert_{L^\infty}\ \sum_{k=1}^\infty 2^{-nk}\ \Vert \eta^k_{r} \lapn w \Vert_{L^2}\\
\end{ma}
\] 
In fact, this inequality is first true for $k \geq K_{\Lambda_1}$ (when we can guarantee a disjoint support of $\eta_{r}^k$ and $\varphi$). By choosing a possibly bigger constant $C_{\Lambda_1}$ it holds also for any $k \geq 1$.\\
The remaining term $I$ is controlled by the PDE \eqref{eq:ELeq}, setting $\psi_{ij} := \alpha_{ij} \varphi$ which is an admissible test function:
\[
\begin{ma}
I &\overset{\eqref{eq:ELeq}}{=}& \intl_{\R^n} a_{ij}\ \alpha_{ij}\ \varphi - \alpha_{ij}\ \intl_{\R^n} \lapn w^j\ H(w^i,\varphi)\\
&=:& I_1 - \alpha_{ij}\intl_{\R^n} \eta_{4{\Lambda_1} r}\ \lapn w^j\ H(w^i,\varphi) - \alpha_{ij}\sum_{k=1}^\infty\ \intl_{\R^n} \eta^k_{4{\Lambda_1} r}\ \lapn w^j\ H(w^i,\varphi)\\
&=:& I_1 - I_2 - \sum_{k=1}^\infty I_{3,k}.
\end{ma}
\]
By Lemma~\ref{la:growth:avp},
\[
I_1 \leq C_{\Lambda_1} r^\frac{n}{2}\ \Vert a \Vert_{L^2}.
\]
By Lemma~\ref{la:hvphilocest} (taking $r = \Lambda_1 r$ and $\Lambda = \Lambda_2$) and the choice of $\Lambda_2$ and $R$, \eqref{eq:estEleq:Lambda2} and \eqref{eq:estEleq:wnormsleqdelta},
\[
I_2 \aleq{} \delta\ \Vert \eta_{4{\Lambda_2} r} \lapn w \Vert_{L^2}.
\]
As for $I_{3,k}$, because the support of $\varphi$ and $\eta^k_{4{\Lambda_1} r}$ is disjoint, by Lemma~\ref{la:bs:disjsuppGen},
\[
\begin{ma}
&&\intl_{\R^n} \eta^k_{4{\Lambda_1} r} \lapn w^j H(w^i,\varphi)\\
&=& \intl_{\R^n} \eta^k_{4{\Lambda_1} r} \lapn w^j \brac{\lapn (w^i \varphi) - w^i \lapn \varphi} \\
&\overset{\sref{L}{la:bs:disjsuppGen}}{\aleq{}}& C_{\Lambda_1}\ \brac{2^k r}^{-n} \Vert \eta^k_{4{\Lambda_1} r} \lapn w^j \Vert_{L^2}
\Vert w \Vert_{L^\infty}\ r^{n}\\
&\aeq{}& \Vert w \Vert_{L^\infty}\ 2^{-nk}\ \Vert \eta^k_{4{\Lambda_1} r} \lapn w^j \Vert_{L^2}.\\
\end{ma}
\]
Using Remark \ref{rem:akegal}, we conclude.
\end{proofL}

\begin{lemma}\label{la:west}
Let $w \in \Hf\cap L^\infty(\R^n,\R^m)$ satisfy \eqref{eq:structureeq}  and \eqref{eq:ELeq} (for some ball $\tilde{D}$, and some $\eta$). Assume furthermore that $w(y) \in \S^{m-1}$ for almost every $y \in \tilde{D}$. Then for any $\varepsilon > 0$ there is $\Lambda > 0$, $R > 0$ and $\gamma > 0$, such that for all $r \in (0,R)$, $x_0 \in \R^n$ such that $B_{\Lambda r}(x_0) \subset \tilde{D}$,
\[
\begin{ma}
&&[w]_{B_r,\frac{n}{2}} + \Vert \lapn w \Vert_{L^2(B_r)}\\
& \leq& \varepsilon \brac{[w]_{B_{\Lambda r},\frac{n}{2}} + \Vert \lapn v \Vert_{L^2(B_{\Lambda r})}}\\
&&+ C_{\varepsilon} \sum_{k=-\infty}^\infty 2^{-\gamma \abs{k}} \brac{[w]_{A_k,\frac{n}{2}} + \Vert \lapn w \Vert_{L^2(A_k)}}\\
&& + C_{\varepsilon} r^{\frac{n}{2}}.
\end{ma}
\]
Here, $A_k = B_{2^{k+1} r}(x_0) \backslash B_{2^{k-1} r}(x_0)$.
\end{lemma}
\begin{proofL}{\ref{la:west}}
Let $\varepsilon > 0$ be given and $\delta := \delta_\varepsilon$ to be chosen later. Take from Lemma~\ref{la:estStrEq} and Lemma~\ref{est:ELeq} the smallest $R$ to be our $R > 0$ and the biggest $\Lambda$ to be our $\Lambda > 20$, such that the following holds: For any skew symmetric matrix $\alpha \in \R^{n\times n}$, $\abs{\alpha} \leq 2$ and any $B_{\Lambda r}(x_0) \equiv B_{\Lambda} \subset \tilde{D}$, $r \in (0,R)$ and for a certain $\gamma > 0$
\[
\begin{ma}
&&\Vert w \cdot \lapn w \Vert_{L^2(B_{16r})} + \Vert w^i \alpha_{ij} \lapn w^{j} \Vert_{L^2(B_{16r})}\\
&\leq& \delta \brac{\Vert \lapn w \Vert_{L^2(B_{\Lambda r})} + [w]_{B_{\Lambda r},\frac{n}{2}}}\\
&&+ C_{\delta, w} \brac{r^\frac{n}{2} + \sum_{k=-\infty}^\infty 2^{-\gamma \abs{k}} \brac{\Vert \lapn w \Vert_{L^2(A_k)} + [w]_{A_k,\frac{n}{2}}}}.
\end{ma}
\] 
In particular, as $\abs{w} = 1$ on $B_{16r}(x_0) \subset \tilde{D}$ we have
\begin{equation}\label{eq:lapnwallest}
\Vert \lapn w \Vert_{L^2(B_{16r})} \leq \delta \brac{\Vert \lapn w \Vert_{L^2(B_{\Lambda r})} + [w]_{B_{\Lambda r},\frac{n}{2}}}
+ C_{\delta, w} \brac{r^\frac{n}{2} + \sum_{k=-\infty}^\infty 2^{-\gamma \abs{k}} \brac{\Vert \lapn w \Vert_{L^2(A_k)} + [w]_{A_k,\frac{n}{2}}}}.
\end{equation}
Then, by Lemma~\ref{la:comps01} we have for a certain $\gamma > 0$ (possibly smaller than the one chosen before)
\[
\begin{ma}
&&[w]_{B_r,\frac{n}{2}} + \Vert \lapn w \Vert_{L^2(B_r)}\\
&\overset{\sref{L}{la:comps01}}{\leq}& 
\varepsilon [w]_{B_{16r}} + C_{\varepsilon} \left (\Vert \lapn w \Vert_{L^2(B_{16r})} + \sum_{k=-\infty}^\infty 2^{-\gamma \abs{k}} \brac{[w]_{A_k,\frac{n}{2}} + \Vert \lapn w \Vert_{L^2(A_k)}} \right )\\
&\overset{\eqref{eq:lapnwallest}}{\aleq{}}& \varepsilon [w]_{B_{16r}} + \delta C_{\varepsilon} \brac{\Vert \lapn w \Vert_{L^2(B_{\Lambda r})} + [w]_{B_{\Lambda r},\frac{n}{2}}}\\
&& +C_{\varepsilon,\delta,w, \tilde{D}} \left (r^\frac{n}{2} + \sum_{k=-\infty}^\infty 2^{-\gamma \abs{k}} \brac{[w]_{A_k,\frac{n}{2}} + \Vert \lapn w \Vert_{L^2(A_k)}} \right ).\\
\end{ma}
\]
Thus, if we set $\delta := \brac{C_{\varepsilon}}^{-1} \varepsilon$, the claim is proven.
\end{proofL}

Finally, we can prove Theorem~\ref{th:regul}, which is an immediate consequence of the following theorem and the Euler-Lagrange-Equations, Lemma~\ref{pr:eleq}.
\begin{theorem}\label{th:wreg}
Let $w \in \Hf(\R^n) \cap L^\infty$ as in Lemma~\ref{la:west}. Then for any $E \subset \tilde{D}$ with positive distance from $\partial \tilde{D}$ there is $\beta > 0$ such that $w \in C^{0,\beta}(E)$.
\end{theorem}
\begin{proofT}{\ref{th:wreg}}
Squaring the estimate of Lemma~\ref{la:west}, we have for arbitary $\varepsilon > 0$ some $\Lambda > 0$ (which we can chose w.l.o.g. to be $2^{K_\Lambda -1}$ for some $K_\Lambda \in \N$), $R > 0$ and $\gamma > 0$ and any $B_r(x_0) \subset \R^n$ where $B_{\Lambda r}(x_0) \subset \tilde{D}$, $r \in (0,R]$
\[
\begin{ma}
&&\brac{[w]_{B_r,\frac{n}{2}}}^2 + \brac{\Vert \lapn w \Vert_{L^2(B_r)}}^2\\
&\leq& 4\varepsilon^2 \brac{[w]_{B_{\Lambda r},\frac{n}{2}}^2 + \Vert \lapn w \Vert_{L^2(B_{\Lambda r})}^2}\\
&&+ C_{\varepsilon} \sum_{k=-\infty}^\infty 2^{-\gamma \abs{k}} \brac{[w]_{A_k(r),\frac{n}{2}}^2 + \Vert \lapn w \Vert^2_{L^2(A_k(r))}}\\
&&+ C_{\varepsilon} r^{n}.
\end{ma}
\]
Here,
\[
 A_k(r) \equiv A_k(r,x_0) = B_{2^{k+1}r}(x_0) \backslash B_{2^{k-1}r}(x_0).
\]
Set
\[
a_k(r) \equiv a_k(r,x_0) := [w]_{A_k(r),\frac{n}{2}}^2 + \Vert \lapn w \Vert_{L^2(A_k(r))}^2.
\]
Then, for some uniform $C_1 > 0$ and $c_1 > 0$ and $K \equiv K_\Lambda \in \N$ such that $2^{K_\Lambda-1} = \Lambda$
\[
\Vert \lapn w \Vert_{L^2(B_{\Lambda r})}^2 \leq C_1 \sum_{k=-\infty}^{K_\Lambda} a_k(r),
\]
and by Lemma~\ref{la:homogloc} also
\[
[w]_{B_{\Lambda r},\frac{n}{2}}^2 \leq C_1 \sum_{k=-\infty}^{K_\Lambda} a_k(r),
\]
and of course,
\[
[w]_{B_r,\frac{n}{2}}^2 + \Vert \lapn w \Vert_{L^2(B_r)}^2  \geq c_1 \sum_{k=-\infty}^{-1} a_k(r),
\]
as well as $\Vert a_k(r) \Vert_{l^1(\Z)} \aleq{} \Vert \lapn w \Vert_{L^2(\R^n)}^2$. Choosing $\varepsilon > 0$ sufficiently small to absorb the effects of the independent constants $c_1$ and $C_1$, this implies
\begin{equation}\label{eq:wreg:growthak}
\sum_{k=-\infty}^{-1} a_k(r) \leq \frac{1}{2}\sum_{k=-\infty}^{K_\Lambda} a_k(r) + C\sum_{k=-\infty}^\infty 2^{-\gamma\abs{k}} a_k(r) + C r^{n}
\end{equation}
This is valid for any $B_r(x_0) \subset B_{\Lambda r}(x_0) \subset \tilde{D}$, where $r \in (0,R)$. Let $E$ be a bounded subset of $\tilde{D}$ with proper distance to the boundary $\partial \tilde{D}$. Let $R_0 \in (0,R)$ such that for any $x_0 \in E$ the ball $B_{2\Lambda R_0}(x_0) \subset \tilde{D}$. Fix some arbitrary $x_0 \in E$. Let now for $k \in \Z$,
\[
b_k \equiv b_k(x_0) := [w]_{A_k(\frac{R_0}{2}),\frac{n}{2}}^2 + \Vert \lapn w \Vert_{L^2(A_k(\frac{R_0}{2}))}^2 = a_k(\frac{R_0}{2}).
\]
Then for any $N \leq 0$,
\[
 \begin{ma}
  \sum_{k=-\infty}^N b_k &=& \sum_{k=-\infty}^N a_{k}(\frac{R_0}{2})\\
&=& \sum_{k=-\infty}^{-1} a_{k+(N+1)}(\frac{R_0}{2})\\
&=& \sum_{k=-\infty}^{-1} a_{k}(2^{N}R_0)\\
&\overset{\eqref{eq:wreg:growthak}}{\leq}& \frac{1}{2}\sum_{k=-\infty}^{K_{\Lambda} } a_k(2^{N}R_0) + C \sum_{k=-\infty}^\infty 2^{-\gamma\abs{k}} a_k(2^N R_0) + C\ R_0^n\ 2^{nN}\\
&\leq& \frac{1}{2}\sum_{k=-\infty}^{K_{\Lambda}+N+1} a_k(\frac{R_0}{2}) + C\ 2^\gamma \sum_{k=-\infty}^\infty 2^{-\gamma\abs{k-N}} a_k(\frac{R_0}{2}) + C\ R_0^n\ 2^{nN}\\
&=&  \frac{1}{2}\sum_{k=-\infty}^{K_{\Lambda}+N+1} b_k + C\ 2^\gamma \sum_{k=-\infty}^\infty 2^{-\gamma\abs{k-N}} b_k + C\ R_0^n\ 2^{nN}
 \end{ma}
\]
Consequently, by Lemma~\ref{la:iteration}, for a $N_0 < 0$ and a $\beta > 0$ (not depending on $x_0$),
\[
\sum_{k=-\infty}^N b_k \leq C\ 2^{\beta N}, \quad \mbox{for any $N \leq N_0$}.
\]
This implies in particular for $\tilde{R}_0 = 2^{N_0}R_0$ (again using Lemma~\ref{la:homogloc} )
\[
[v]_{B_r(x_0),\frac{n}{2}} \leq C_{R_0}\ r^{\frac{\beta}{2}} \mbox{\quad for all $r < \tilde{R}_0$ and $x_0 \in E$}.
\]
Finally, Dirichlet Growth Theorem, Theorem~\ref{la:it:dg}, implies that $v \in C^{0,\beta}(E)$.
\end{proofT}

\renewcommand{\thesection}{A}
\renewcommand{\thesubsection}{A.\arabic{subsection}}
\section{Ingredients for the Dirichlet Growth Theorem}
\subsection{Iteration Lemmata}
With the same argument as in \cite[Proposition A.1]{DR09Sphere} the following Iteration Lemma can be proven.
\begin{lemma}\label{la:driteration}
Let $a_k \in l^1(\Z)$, $a_k \geq 0$ for any $k \in \Z$ and assume that there is $\alpha > 0$, $\Lambda > 0$ such that for any $N \leq 0$
\begin{equation}\label{eq:dritersumknak}
 \sum_{k=-\infty}^N a_k \leq \Lambda \left (\sum_{k=N+1}^\infty 2^{\gamma(N+1-k)} a_k + 2^{\alpha N} \right ).
\end{equation}
Then there is $\beta \in (0,1)$, $\Lambda_2 > 0$ such that for any $N \leq 0$
\[
 \sum_{k=-\infty}^N a_k \leq 2^{\beta N} \Lambda_2.
\]
\end{lemma}
\begin{proofL}{\ref{la:driteration}}
Set for $N \leq 0$
\[
A_N := \sum_{k=-\infty}^N a_k.
\]
Then obviously,
\[
a_k = A_k - A_{k-1}.
\]
Equation \eqref{eq:dritersumknak} then reads as (note that $A_N \in l^\infty(\Z)$)
\[
\begin{ma}
A_N &\leq& \Lambda \brac{\sum_{k=N+1}^\infty 2^{\gamma(N+1-k)} \brac{A_k-A_{k-1}} + 2^{\alpha N} }\\
&=& \Lambda  \brac{\sum_{k=N+1}^\infty 2^{\gamma(N+1-k)} A_k - \sum_{k=N+2}^\infty 2^{\gamma(N-(k-1))} A_{k-1} - A_N + 2^{\alpha N} }\\
&=& \Lambda  \brac{\sum_{k=N+1}^\infty 2^{\gamma(N+1-k)} A_k - \sum_{k=N+1}^\infty 2^{\gamma(N-k)} A_{k} - A_N + 2^{\alpha N} }\\
&=& \Lambda  \brac{\sum_{k=N+1}^\infty 2^{\gamma(N+1-k)} A_k - 2^{-\gamma}\sum_{k=N+1}^\infty 2^{\gamma(N-k+1)} A_{k} - A_N + 2^{\alpha N} }\\
&=& \Lambda  \brac{(1-2^{-\gamma})\sum_{k=N+1}^\infty 2^{\gamma(N+1-k)} A_k - A_N + 2^{\alpha N} }.
\end{ma}
\]
This calculation is correct as $(A_k)_{k \in \Z} \in l^\infty(\Z)$ and $\brac{2^{\gamma {N+1-k}}}_{k = N}^\infty \in l^1([N,N+1,\ldots,\infty])$ because of the condition $\gamma > 0$. Otherwise we could not have used linearity for absolutely convergent series.\\
We have shown that \eqref{eq:dritersumknak} is equivalent to
\[
A_N \leq \frac{\Lambda}{1+\Lambda} \brac{1-2^{-\gamma}} \sum_{k=N+1}^\infty 2^{\gamma(N+1-k)} A_k + \frac{\Lambda}{1+\Lambda} 2^{\alpha N}.
\]
Set $\tau := \frac{\Lambda}{\Lambda+1}\brac{1-2^{-\gamma}}$. Then, for all $N \leq 0$,
\begin{equation}\label{eq:it:1step}
A_N \leq \tau \sum_{k=N+1}^\infty 2^{\gamma (N+1-k)} A_k + 2^{\alpha N}.
\end{equation}
Set
\[
\tau_k := \begin{cases}
						1 \quad &\mbox{if $k = 0$},\\
						\tau (\tau + 2^{-\gamma})^{k-1}\quad &\mbox{if $k \geq 1$}.
					\end{cases}
\]
Then for any $K \geq 0$, $N \leq 0$,
\begin{equation}\label{eq:it:IA}
A_{N-K} \leq \tau_{K+1} \sum_{k=N+1}^\infty 2^{\gamma(N+1-k)} A_k + \sum_{k=0}^{K} \tau_k 2^{\alpha (N-K+k)}.
\end{equation}
In fact, this is true for $K = 0$, $N \leq 0$ by \eqref{eq:it:1step}. Moreover, if we assume that \eqref{eq:it:IA} holds for some $K \geq 0$ and all $N \leq 0$, we compute
\[
\begin{ma}
&&A_{N-K-1}\\
&=& A_{(N-1)-K}\\
&\overset{\eqref{eq:it:IA}}{\leq}& \tau_{K+1} \sum_{k=N}^\infty 2^{\gamma(N-k)} A_k + \sum_{k=0}^K \tau_k 2^{\alpha (N-1-K+k)}\\
&=&\tau_{K+1} \brac{ A_N + 2^{-\gamma}\sum_{k=N+1}^\infty 2^{\gamma(N+1-k)} A_k}\\
&&+ \sum_{k=0}^K \tau_k 2^{\alpha (N-1-K+k)}\\
&\overset{\eqref{eq:it:1step}}{\leq}& \tau_{K+1} \brac{ \tau \sum_{k=N+1}^\infty 2^{\gamma (N+1-K)} A_k + 2^{\alpha N} + 2^{-\gamma}\sum_{k=N+1}^\infty 2^{\gamma(N+1-k)} A_k}\\
&&+ \sum_{k=0}^K \tau_k 2^{\alpha (N-1-K+k)}\\
&\leq& \tau_{K+1} (\tau + 2^{-\gamma}) \sum_{k=N+1}^\infty 2^{\gamma(N+1-k)} A_k + \tau_{K+1} 2^{\alpha N}\\
&&+ \sum_{k=0}^K \tau_k 2^{\alpha (N-(K+1)+k)}\\
&=& \tau_{K+2} \sum_{k=N+1}^\infty 2^{\gamma(N+1-k)} A_k + \sum_{k=0}^{K+1} \tau_k 2^{\alpha (N-(K+1)+k)}.
\end{ma}
\]
This proves \eqref{eq:it:IA} for any $K \geq 0$ and $N \leq 0$. As $\tau_{K} \leq 1$,
\[
A_{N-K} \leq C_\gamma \tau_{K+1} A_\infty + 2^{\alpha N} C_\alpha.
\]
So for any $\tilde{N} \leq 0$,
\[
\begin{ma}
A_{\tilde{N}} &=& A_{(\tilde{N}+\left \lfloor \frac{\abs{\tilde{N}}}{2} \right \rfloor) - \left \lfloor \frac{\abs{\tilde{N}}}{2} \right \rfloor}\\
&\leq& C_\gamma \brac{A_\infty + 1}\ \tau_{\left \lfloor \frac{\abs{\tilde{N}}}{2} \right \rfloor} + 2^{\alpha (\tilde{N}+\left \lfloor \frac{\abs{\tilde{N}}}{2} \right \rfloor)}\\
&\leq& C_{\gamma,\alpha}\ \brac{A_\infty + 1}\ \brac{\tau_{\left \lfloor \frac{\abs{\tilde{N}}}{2} \right \rfloor} + 2^{-\alpha \frac{\abs{\tilde{N}}}{2}}}.
\end{ma}
\]
Using now that $\tau_{k} \leq 2^{-\theta k}$ for all $k \geq 0$ and some $\theta > 0$, have shown that
\[
A_{\tilde{N}} \leq C_{\gamma,\alpha} \brac{A_\infty+1}\ 2^{\mu \tilde{N}}.
\]
for some small $\mu > 0$.
\end{proofL}

As a consequence the following Iteration Lemma holds, too.
\begin{lemma}\label{la:iteration}
For any $\Lambda_1,\Lambda_2,\gamma > 0$, $L \in \N$ there exists a constant $\Lambda_3 > 0$ and an integer $\bar{N} \leq 0$ such that the following holds. Let $(a_k) \in l^1(\Z)$, $a_k \geq 0$ for any $k \in \Z$ such that for every $N \leq 0$,
\begin{equation}\label{eq:it:bigguy}
 \suml_{k=-\infty}^N a_k \leq \frac{1}{2} \suml_{k=-\infty}^{N+L} a_k + \Lambda_1 \suml_{k=-\infty}^N 2^{\gamma \brac{k-N}} a_k + \Lambda_2 \suml_{k=N+1}^\infty 2^{\gamma(N-k)} a_k + \Lambda_2 2^{\gamma N}.
\end{equation}
Then for any $N \leq \bar{N}$,
\[
 \suml_{k=-\infty}^N a_k \leq \Lambda_3 \suml_{k=N+1}^{\infty} 2^{\gamma (N-k)} a_k + \Lambda_3 2^{\gamma N}
\]
and consequently by Lemma~\ref{la:driteration} for some $\beta \in (0,1)$, $\Lambda_4 > 0$ (depending only on $\Vert a_k \Vert_{l^1(\Z)}$, $\Lambda_3$) and for any $N \leq \bar{N}$
\[
 \sum_{k=-\infty}^N a_k \leq \Lambda_4 2^{\beta N}.
\]
\end{lemma}
\begin{proofL}{\ref{la:iteration}}
Firstly, \eqref{eq:it:bigguy} implies by absorption of $\frac{1}{2} \sum_{k=-\infty}^N a_k$ to the right hand side,
\[
\begin{ma}
 &&\suml_{k=-\infty}^N a_k\\
&\overset{\eqref{eq:it:bigguy}}{\leq}& 2 \suml_{k=N+1}^{N+L} a_k + 2\Lambda_1 \suml_{k=-\infty}^N 2^{\gamma (k-N)} a_k\\
&&+ 2\Lambda_2 \suml_{k=N+1}^\infty 2^{\gamma(N-k)} a_k + \Lambda_2 2^{\gamma N}\\
&\leq& 2^{\gamma L+1} \suml_{k=N+1}^{N+L} 2^{\gamma(N-k)} a_k + 2\Lambda_1 \suml_{k=-\infty}^N 2^{\gamma (k-N)} a_k\\
&&+ 2\Lambda_2 \suml_{k=N+1}^\infty 2^{\gamma(N-k)} a_k +\Lambda_2 2^{\gamma N}\\
&\leq& 2\Lambda_1 \suml_{k=-\infty}^N 2^{\gamma (k-N)} a_k\\
&&+ \brac{2^{\gamma L+1} + 2\Lambda_2}\ \suml_{k=N+1}^\infty 2^{\gamma(N-k)} a_k +\Lambda_2 2^{\gamma N}.\\
\end{ma}
\]
Next, choose $K \in \N$ such that $2^{-\gamma K} \leq \frac{1}{4\Lambda_1}$. Then,
\[
\begin{ma}
 &&\suml_{k=-\infty}^N a_k\\
 &\leq& 2\Lambda_1 \suml_{k=-\infty}^{N-K} 2^{\gamma (k-N)} a_k + 2\Lambda_1 \suml_{k=N-K+1}^N 2^{\gamma (k-N)} a_k\\
&&\quad+ \brac{2^{\gamma L+1} + 2\Lambda_2}\ \suml_{k=N+1}^\infty 2^{\gamma(N-k)} a_k + \Lambda_2 2^{\gamma N}\\
&\leq& \frac{1}{2} \suml_{k=-\infty}^{N-K} a_k + 2\Lambda_1 \suml_{k=N-K+1}^N a_k\\
&&+ \brac{2^{\gamma L+1} + 2\Lambda_2}\ \suml_{k=N+1}^\infty 2^{\gamma(N-k)} a_k +  \Lambda_2 2^{\gamma N}.
\end{ma}
\]
Consequently, again by absorbing
\[
\begin{ma}
 &&\suml_{k=-\infty}^{N-K} a_k\\
 &\leq& 4\Lambda_1 \suml_{k=N-K+1}^N a_k + \brac{2^{\gamma L+2} + 4\Lambda_2}\ \suml_{k=N+1}^\infty 2^{\gamma(N-k)} a_k +2 \Lambda_2 2^{\gamma N}\\
&\leq& 4\Lambda_1 2^{\gamma K} \suml_{k=N-K+1}^N 2^{\gamma(N-K-k)} a_k\\
&&\quad + 2^{\gamma K}\brac{2^{\gamma L+2} + 4\Lambda_2}\ \suml_{k=N+1}^\infty 2^{\gamma(N-K-k)} a_k +2 \Lambda_2 2^{\gamma N}\\
&\leq& \brac{4\Lambda_1 2^{\gamma K} + 2^{\gamma K}\brac{2^{\gamma L+2} + 4\Lambda_2}}\ \suml_{k=N-K+1}^\infty 2^{\gamma(N-K-k)} a_k\\
&&+2 \Lambda_2 2^K\ 2^{\gamma {N-K}}\\
&=:& \Lambda_3\ \brac{\suml_{k=N-K+1}^\infty 2^{\gamma(N-K-k)} a_k + 2^{\gamma \brac{N-K}}}.
\end{ma}
\]
This is valid for any $N \leq 0$, so for any $\tilde{N} \leq -K$
\[
 \suml_{k=-\infty}^{\tilde{N}} a_k \leq \Lambda_3\ \brac{\suml_{k=\tilde{N}+1}^\infty 2^{\gamma(\tilde{N}-k)} a_k + 2^{\gamma \tilde{N}}}.
\]
We conclude by Lemma~\ref{la:driteration}.
\end{proofL}

\subsection{A fractional Dirichlet Growth Theorem}
In this section we will state and prove a Dirichlet Growth-Type theorem using mainly Poincar\'{e}'s inequality. For an approach by potential analysis, we refer to \cite{Adams75}, in particular \cite[Corollary after Proposition 3.4]{Adams75}.\\
Let us introduce some quantities related to Morrey- and Campanato spaces as treated in \cite{GiaquintaMI83} for some domain $D \subset \R^n$, $\lambda > 0$
\[
 J_{D,\lambda,R}(v) := \sup_{\ontop{x \in D}{0 < \rho < R}}\ \brac{\rho^{-\lambda}\ \intl_{D \cap B_{\rho}(x)} \abs{v}^2}^{\frac{1}{2}}
\]
and
\[
 M_{D,\lambda,R}(v) := \sup_{\ontop{x \in D}{0 < \rho < R}}\ \brac {\rho^{-\lambda}\ \intl_{D \cap B_{\rho}(x)} \abs{v-(v)_{D \cap B_\rho (x)}}^2}^{\frac{1}{2}}.
\]
Moreover, let us denote by $C^{0,\alpha}(D)$, $\alpha \in (0,1)$ all H\"older continuous functions with the exponent $\alpha$. Then the following relations hold:
\begin{lemma}[Integral Characterization of H\"older continuous functions]\label{la:it:cshoelder}
(See \cite[Theorem III.1.2]{GiaquintaMI83})\\
Let $D \subset \R^n$ be a smoothly bounded set, and $\lambda \in (n,n+2)$, $v \in L^2(D)$. Then $v \in C^{0,\alpha}(D)$ for $\alpha = \frac{\lambda-n}{2}$ if and only if for some $R > 0$
\[
 M_{D,\lambda,R}(v) < \infty.
\]
\end{lemma}

\begin{lemma}[Relation between Morrey- and Campanato spaces]\label{la:it:mscs}
(See \cite[Proposition III.1.2]{GiaquintaMI83})\\
Let $D \subset \R^n$ be a smoothly bounded set, and $\lambda \in (1,n)$, $v \in L^2(D)$. Then for a constant $C_{D,\lambda} > 0$
\[
 J_{D,\lambda,R}(v) \leq C_{D,\lambda,R}\ \brac{\Vert v \Vert_{L^2(D)} + M_{D,\lambda,R}(v)}.
\]
\end{lemma}

As a consequence of Lemma~\ref{la:it:mscs} we have
\begin{lemma}\label{la:it:mdleqmdnN}
Let $D \subset \R^n$ be a convex, smoothly bounded domain. Set $N := \lceil \frac{n}{2} \rceil -1$. Then if $v \in L^2(D)$, $\lambda \in (n,n+2)$,
\[
 M_{D,\lambda,R}(v) \leq  C_{D,\lambda,R} \brac{\Vert v \Vert_{H^{N}(D)} + \suml_{\abs{\alpha} = N} M_{D ,\lambda-2N,R}(\partial^\alpha v)}.
\]
\end{lemma}
\begin{proofL}{\ref{la:it:mdleqmdnN}}
For any $r \in (0,R)$, $x \in D$ set $B_r \equiv B_r(x)$. As $D$ is convex, also $B_r \cap D$ is convex, so by classic Poincar\'e inequality on convex sets, Lemma~\ref{la:poincCMV},
\[
\begin{ma}
\intl_{D \cap B_r} \abs{v-(v)_{D \cap B_r}}^2 &\overset{\sref{L}{la:poincCMV}}{\leq}& C \diam(D \cap B_r)^2\ \intl_{D \cap B_r} \abs{\nabla v}^2\\
&\aleq{}& r^2 \intl_{D \cap B_r} \abs{\nabla v}^2.
\end{ma}
\]
Consequently,
\[
 M_{D,\lambda,R}(v) \leq C_n\ J_{D, \lambda-2,R}(\nabla v).
\]
As $\lambda \in (n,n+2)$, that is in particular $\lambda - 2 < n$, by Lemma~\ref{la:it:mscs},
\[
 J_{D, \lambda-2,R}(\nabla v) \leq C_{D,R,\lambda}\ \left (\Vert \nabla v \Vert_{L^2(D)} + M_{D,\lambda,R}(\nabla v) \right ).
\]
Iterating this estimate $N$ times, using that $\lambda - 2N > 0$, we conclude.
\end{proofL}

Finally, we can prove a sufficient condition for H\"older continuity on $D$ expressed by the growth of $\lapn v$:
\begin{lemma}[Dirichlet Growth Theorem]\label{la:it:dg}
Let $D \subset \R^n$ be a smoothly bounded, convex domain, let $v \in H^{\frac{n}{2}}(\R^n)$ and assume there are constants $\Lambda > 0$, $\alpha \in (0,1)$, $R > 0$ such that
\begin{equation}\label{eq:smallgrowthlapnv}
 \sup_{\ontop{r \in (0,R)}{x \in D}} r^{-\alpha} [v]_{B_r(x),\frac{n}{2}} \leq \Lambda.
\end{equation}
Then $v \in C^{0,\alpha}(D)$.
\end{lemma}
\begin{proofL}{\ref{la:it:dg}}
We only treat the case where $n$ is odd, the even dimension case is similar. Set $N := \lfloor \frac{n}{2} \rfloor$. We have for any $x \in D$, $r \in (0,R)$, $D_r \equiv D_r(x):= B_r(x) \cap D$, using that the boundary of $D$ is smooth and thus $\abs{D_r(x)} \geq c_D \abs{B_r(x)}$ for any $x \in D$ (because there are no sharp outward cusps in $D$)
\[
\begin{ma}
 &&\intl_{D_r} \abs{\nabla^N v(x) - \brac{\nabla^N v}_{D_r}}^2\\
&\aleq{}& \frac{\brac{\diam(D_r)}^{2(n-N)}} {\abs{D_r}}\ \intl_{D_r} \intl_{D_r} \frac{\abs{\nabla^N v(x)-\nabla^N v(y)}^2}{\abs{x-y}^{2(n-N)}}\ dx\ dy\\
&\aleq{}& r^{n-2N} \brac{[v]_{B_r(x),\frac{n}{2}}}^2\\
&\overset{\eqref{eq:smallgrowthlapnv}}{\aleq{}}& r^{n-2N+2\alpha} \Lambda^2.
\end{ma}
\]
Thus, for $\lambda = n+2\alpha \in (n,n+2)$ 
\[
M_{D ,\lambda-2N,R}(\nabla^N v) \aleq{} \Lambda.
\]
By Lemma~\ref{la:it:mdleqmdnN} this implies
\[
 M_{D ,\lambda,R}(v) \aleq{} \Lambda+\Vert v \Vert_{H^{N}(D)} < \infty
\]
which by Lemma~\ref{la:it:cshoelder} is equivalent to $v \in C^{0,\alpha}(D)$.
\end{proofL}

\bibliographystyle{alpha}%
\bibliography{bib}%
\vspace{2em}
\begin{tabbing}
\quad\=Armin Schikorra\\
\>RWTH Aachen University\\
\>Institut f\"ur Mathematik\\
\>Templergraben 55\\
\>52062 Aachen\\
\>Germany\\
\\
\>email: schikorra@instmath.rwth-aachen.de\\
\>page: www.instmath.rwth-aachen.de/$\sim$schikorra
\end{tabbing}
\end{document}